\documentclass{amsart}
\usepackage{amsmath,amsthm,amssymb,bm}
\usepackage{xcolor}
\usepackage{graphicx}
\usepackage{mathrsfs,esint}
\usepackage{comment}
\usepackage{enumerate}
\usepackage{hyperref}    
\hypersetup{
     colorlinks = true,
     linkcolor = black,
     anchorcolor =green!50!black,
     citecolor = green!50!black,
     filecolor =green!50!black,
     urlcolor = green!50!black
}
\usepackage{cleveref}

\newtheorem{thm}{Theorem}[section]
\newtheorem{prop}{Proposition}[section]
\newtheorem{coro}{Corollary}[section]
\newtheorem{lem}{Lemma}[section]

\newtheorem{defn}{Definition}[section]
\newtheorem{assu}{Assumption}[section]

\newtheorem{rem}{Remark}[section]

\newcommand{\beq}{\begin{equation}}
\newcommand{\eeq}{\end{equation}}

\newcommand{\dist}{\text{dist}}

\newcommand{\supp}{\text{supp}}

\def\cS{\mathcal{S}}

\def\C{\mathbb{C}}

\def\N{\mathbb{N}}

\def\R{\mathbb{R}}
\def\S{\mathbb{S}}

\def\Z{\mathbb{Z}}

\def\ab{\bm a}
\def\bb{\bm b}
\def\cb{\bm c}

\def\eb{\bm e}
\def\fb{\bm f}

\def\nb{\bm n}

\def\ub{\bm u}
\def\vb{\bm v}
\def\wb{\bm w}
\def\xb{\bm x}
\def\yb{\bm y}
\def\zb{\bm z}

\def\nub{\bm\nu}
\def\xib{\bm\xi}
\def\phib{\bm\phi}
\def\etab{\bm\eta}

\newcommand{\al}{\alpha}

\newcommand{\del}{\delta}
\newcommand{\ep}{\epsilon}

\newcommand{\Om}{\Omega}

\title[Nonlocal half-ball vector operators on bounded domains]{Nonlocal half-ball vector operators on bounded domains: Poincar\'e inequality and its applications}
\author{Zhaolong Han}
\address{Department of Mathematics, University of California, San Diego, CA 92093, United States} 
\email{zhhan@ucsd.edu}
\author{Xiaochuan Tian}
\address{Department of Mathematics, University of California, San Diego, CA 92093, United States} 
\email{xctian@ucsd.edu}

\date{}
\begin{document}

\begin{abstract}
This work contributes to nonlocal vector calculus as an indispensable mathematical tool for 
the study of nonlocal models that arises in a variety of applications. 
We define the nonlocal half-ball gradient, divergence and curl operators with general kernel functions (integrable or fractional type with finite or infinite supports) and study the associated nonlocal vector identities. 
We study the nonlocal function space on bounded domains associated with zero Dirichlet boundary conditions and the half-ball gradient operator and show it is a separable Hilbert space with smooth functions dense in it. 
A major result is the nonlocal Poincar\'e inequality, based on which a few applications are discussed, and these include applications to nonlocal convection-diffusion, nonlocal correspondence model of linear elasticity, and nonlocal Helmholtz decomposition on bounded domains. 
\end{abstract}



\maketitle

\tableofcontents

\section{Introduction}
In recent decades, nonlocal models that account for interactions occurring at a distance have been increasingly popular in many scientific fields.
In particular, they appear widely in applications in continuum mechanics, probability and finance, image processing and population
dynamics, and have been shown to more faithfully and effectively model observed phenomena that involve
possible discontinuities, singularities and other anomalies \cite{askari2008peridynamics,benson2000application,buades2010image,fuentes2003nonlocal,HaBo11,klafter2005anomalous,scalas2000fractional}.

One type of nonlocal problem is featured with generalizing the integer-order scaling laws that appear in PDEs to scaling laws of non-integer orders.
This type of problem usually involves integral operators with fractional kernels that are supported in the whole space (i.e., with infinite nonlocal interactions), such as the fractional Laplace operator that models non-standard diffusion of a fractional order \cite{bakunin2008turbulence,kassmann2017intrinsic,meerschaert2006fractional}.   
Another type of nonlocal problem focuses on finite range interactions and connects to PDEs by localization of nonlocal interactions \cite{Du2019,DGLZ12}.
A prominent example is peridynamics, a nonlocal continuum model in solid mechanics, which is shown to be consistent with the classical elasticity theory by localization \cite{MeDu14b,Silling00,SEWX07}. 
Other nonlocal models in this type include nonlocal (convection-)diffusion and nonlocal Stokes equations with finite nonlocal interactions that are inspired by peridynamics.  
Nonlocal vector calculus is developed in \cite{DGLZ13a} and is used for reformulating nonlocal problems under a more systematic framework analogous to classical vector calculus \cite{DGLZ12,gunzburger2010nonlocal}. 
See \cite{d2021towards,du2022nonlocal} for surveys on connecting the fractional and nonlocal vector calculus.

There are two commonly used frameworks in nonlocal vector calculus \cite{DGLZ13a}; one involves two-point nonlocal (difference) operators and another involves one-point nonlocal (integral) operators.
The two-point nonlocal gradient operator and its adjoint operator are used to reformulate nonlocal diffusion and the bond-based peridynamics models \cite{Du2019}.
The one-point nonlocal gradient operator, on the other hand, is also used in a variety of applications including nonlocal advection equation, nonlocal Stokes equation and the peridynamics correspondence models \cite{du2018stability,DuTi20,lee2019asymptotically,silling2017stability}. 
In some sense, the one-point nonlocal operators, including nonlocal gradient, divergence and curl operators, are more convenient to use as modeling tools since they can be directly used in  place of their classical counterparts appearing in PDEs. 
However, the mathematical properties of nonlocal models involving these operators 
are not readily guaranteed without careful investigation.
For example, instability of the peridynamics correspondence model is observed in which nonlocal deformation gradient is used to replace the classical deformation gradient, and is later explained in \cite{du2018stability} as a result of lack of conscious choice of the interaction kernels in the nonlocal gradient operators.  
Singular kernels are proposed in \cite{du2018stability} for the remedy of instability which resembles the kernel functions in the Riesz fractional gradient in terms of singularity at origin \cite{bellido2023nonlocal,shieh2015new}. 
Later on, nonlocal gradient operators with hemispherical interaction neighborhoods are used in \cite{lee2020nonlocal} so that the singularity in kernel functions is no longer a necessity for the corresponding nonlocal Dirichlet energies to be stable.  
Both \cite{du2018stability} and \cite{lee2020nonlocal} work on functions defined on periodic cells to facilitate Fourier analysis. 
The starting point of this work is to establish a functional analysis framework that extends the Fourier analysis in \cite{lee2020nonlocal} and apply it to
nonlocal Dirichlet boundary value problems. 
With a general setting, we work with kernels that include both the Riesz fractional type (with infinite support) and the compactly supported type inspired by peridynamics.
We remark that in a recent work \cite{bellido2023nonlocal}, the authors consider the truncated Riesz fractional type kernels defined with full spherical support
and study the properties of the corresponding function spaces by establishing a nonlocal fundamental theorem of calculus, and no such formula exists for kernels with hemispherical interaction neighborhoods which are our main focus in this work.

The major contribution of this work is the study of the functional analysis properties of the nonlocal space associated with 
the half-ball nonlocal gradient operator defined on bounded domains. 
We show that the space is a Hilbert space, and more importantly, it is separable with smooth functions dense in it, a property on which many applications are based. 
Another major result is the Poincar\'e inequality on functions with zero Dirichlet boundary conditions.
We spend two whole sections on the proof it, one for the case of integrable kernels with compact support and another for more general kernels, including non-integrable kernels and kernels with infinite supports.
Poincar\'e inequality is crucial for the study of boundary value problems. 
Indeed, we illustrate its use in three applications. 
The first application is the well-posedness of a class of nonlocal convection-diffusion equations defined via nonlocal half-ball gradient and divergence.
Secondly, we study the nonlocal correspondence model of linear elasticity, where we also show a nonlocal Korn's inequality for functions with Dirichlet boundary conditions.
Note that the convergence of Galerkin approximations to these equations is natural, although we do not illustrate it in detail due to the length of the paper, 
as a result of the separability of the associated nonlocal energy spaces.  
The last application is a nonlocal version of Helmholtz decomposition for vector fields defined on bounded domains, 
a result of the solvability of the nonlocal Poisson type problem and some nonlocal vector identities involving gradient, divergence and curl 
which we also established in this paper.   
We remark that Helmholtz decomposition for one-point nonlocal operators is also studied in \cite{d2022connections,haar2022new,lee2020nonlocal},
but only periodic domains or the whole space are considered in these works.

{\it Outline of the paper}.
We start with the principal value definition of the nonlocal half-ball gradient, divergence and curl operators for measurable functions in \Cref{sec:operators}, 
and the corresponding distributional gradient, divergence and curl operators are followed.   
Fourier symbols of these operators are studied for later use and some nonlocal vector identities for smooth functions are established in the section. 
In \Cref{sec:space}, we define the nonlocal function space associated with the Dirichlet integral defined via the distributional nonlocal half-ball gradient, 
an analogue of the $H^1_0$ Sobolev space in the local case, and show it is a separable Hilbert space. 
Ingredients such as closedness under multiplication with smooth functions, continuity of translation and mollification are established to prove the density result.
In addition, we show that the distributional divergence and curl are well-defined quantities in the $L^2$ sense in the nonlocal function space for vector fields,
an analogue of the fact that $H^1\subset H(\text{div})$  and $H^1\subset H(\text{curl})$ in the local case. Thus the vector identities also hold for functions in the nonlocal function spaces. 
The nonlocal Poincar\'e inequality is proved for integrable kernels with compact support in \Cref{sec:Poincare_int_kernel}, based on which the nonlocal Poincar\'e inequality is shown for more general kernels in \Cref{sec:Poincare_nonint_kernel}.
\Cref{sec:applications} contains three applications of our functional analysis framework, including applications to nonlocal convection-diffusion, nonlocal linear elasticity, and nonlocal Helmholtz decomposition on bounded domains. 
Finally, we conclude in \Cref{sec:conclusion}.

\section{Nonlocal half-ball vector operators}
\label{sec:operators}
We introduce the nonlocal half-ball vector operators in this section and discuss their properties. 
In the following, we let $\nub\in\R^d$ be a fixed unit vector. Denote by $\chi_{\nub}(\zb)$ the characteristic function of the half-space $\mathcal{H}_{\nub}:=\{\zb\in \mathbb{R}^d:\zb\cdot\nub\geq 0\}$ parameterized by the unit vector $\nub$.

Throughout the paper, we adopt the following notations in linear algebra. For two column vectors $\ab,\bb\in\R^d\cong \R^{d\times 1}$, $\ab\cdot\bb$ is the dot product and $\ab\times\bb$ is the cross product if $d=3$. If $\ab\in\R^d$ and $\bb\in \R^N$, then the tensor product of $\ab$ and $\bb$ is a $d\times N$ matrix, given by $\ab\otimes\bb=(a_i b_j)_{1\le i\le d,1\le j\le N}$. For two matrices $A,B\in\mathcal{M}_{m,n}(\R)$, we define $A:B=\sum_{i=1}^m \sum_{j=1}^n a_{ij}b_{ij}$.

\subsection{Definitions and integration by parts}
Following the notion of nonlocal nonsymmetric operator defined in \cite{lee2020nonlocal}, we define the nonlocal half-ball vector operators as follows.

Throughout the paper, we assume that $w$ satisfies the following conditions:
\beq
\label{eq:kernelassumption}
\left\{
\begin{aligned}
& w\in L^1_{\mathrm{loc}}(\R^d\backslash \{\bm 0\}),\ w\geq 0, \ w\ \text{is radial}; \\
&\text{there exists}\ \epsilon_0\in (0,1) \text{ such that } w(\xb)>0 \text{ for } 0<|\xb|\leq \ep_0; \\
&   \int_{\R^d}\min(1,|\xb|)w(\xb)d\xb=\int_{|\xb|\le 1}w(\xb)|\xb|d\xb+\int_{|\xb|>1}w(\xb)d\xb=:M_w^1+M_w^2<\infty.
\end{aligned}
\right.
\eeq

\begin{rem}
There are two typical types of kernels used in the literature that satisfy \cref{eq:kernelassumption}. One type of kernel is those with compact supports, e.g., $\supp\ w \subset B_\del(\bm{0})$ for some $\del>0$, where $\delta$ represents the finite length of nonlocal interactions.
Compactly supported kernels are used in peridynamics and the related studies, see e.g., \cite{bellido2023nonlocal,Du2019,lee2020nonlocal,silling2017stability}.  Another type of kernel has non-compact supports, e.g., $w(\xb) =C |\xb|^{-d-\alpha}$ for $\alpha\in (0,1)$, which relates to the Riesz fractional derivatives studied in \cite{bellido2023nonlocal,shieh2015new,shieh2018new}. 
For $d=1$ with $w(\xb) =C |\xb|^{-1-\alpha}$, the nonlocal half-ball operators in this work directly relate to the Marchaud one-sided derivatives studied in \cite{allen2016parabolic,kilbas1993fractional,stinga2020one}.
Tempered fractional operators are discussed in \cite{d2021towards,sabzikar2015tempered} where $w(\xb) =C e^{-\lambda |\xb|}|\xb|^{-d-\alpha}$ for $\lambda>0$  and $\alpha\in (0,1)$.
\end{rem}

\begin{defn}
\label{def:operators_PV}
Given a measurable vector-valued function $\ub:\R^d\to \R^N$, the action of \textbf{nonlocal half-ball gradient operator} $\mathcal{G}^{\nub}_w$ on $\ub$ is defined as
 \begin{equation}\label{hbgrad}
    \mathcal{G}^{\nub}_w \ub(\xb):=\lim_{\epsilon\to 0}\int_{\R^d\backslash B_\epsilon(\xb)}\chi_{\nub}(\yb-\xb)\frac{\yb-\xb}{|\yb-\xb|}\otimes(\ub(\yb)-\ub(\xb))w(\yb-\xb)d\yb,\quad \xb\in \R^d,
\end{equation}
 where $\mathcal{G}^{\nub}_w \ub:\R^d\to\R^{d\times N}$. Given a measurable matrix-valued function $\vb:\R^d\to\R^{d\times N}$, the action of \textbf{nonlocal half-ball divergence operator} $\mathcal{D}^{\nub}_w$ on $\vb$ is defined as
 \begin{equation}\label{hbdiv}
    \mathcal{D}_w^{\nub} \vb(\xb):=\lim_{\epsilon\to 0}\int_{\R^d\backslash B_\epsilon(\xb)}\chi_{\nub}(\yb-\xb)\left[\frac{\yb^T-\xb^T}{|\yb-\xb|}(\vb(\yb)-\vb(\xb))\right]^T w(\yb-\xb)d\yb,\quad\xb\in \R^d,
\end{equation}
 where $\mathcal{D}^{\nub}_w \vb:\R^d\to\R^{N}$. If $d=3$ and $\vb:\R^3\to\R^3$, then the action of \textbf{nonlocal half-ball curl operator} $\mathcal{C}^{\nub}_w$ on $\vb$ is defined as
 \begin{equation}\label{hbcurl}
     \mathcal{C}_w^{\nub} \vb(\xb):=\lim_{\epsilon\to 0}\int_{\R^3\backslash B_\epsilon(\xb)}\chi_{\nub}(\yb-\xb)\frac{\yb-\xb}{|\yb-\xb|}\times(\vb(\yb)-\vb(\xb))w(\yb-\xb)d\yb,\quad \xb\in \R^3,
 \end{equation}
  where $\mathcal{C}^{\nub}_w \vb:\R^3\to\R^3$.
\end{defn}

\begin{rem}
    Suppose $\mathrm{supp}\ w\subset \overline{B_1(\bm 0)}$. For the affine function $\ub(\xb)=A\xb+\bb$ where $A\in \R^{N\times d}$ and $\bb\in \R^{N}$, it follows that \[
    \begin{split}
        \mathcal{G}^{\nub}_w \ub(\xb)&=\int_{B_1(\bm 0)}\chi_{\nub}(\zb)\frac{\zb}{|\zb|}\otimes(A\zb)w(\zb)d\zb\\
        &=\left(\int_{B_1(\bm 0)}\chi_{\nub}(\zb)|\zb|w(\zb)\frac{\zb}{|\zb|}\otimes \frac{\zb}{|\zb|} d\zb \right)A^T\\
        &=\frac{M_w^1}{2d}A^T,\quad\forall\xb\in\R^d,
    \end{split}
    \]
    where we used\[
    \begin{split}
        &\int_{B_1(\bm 0)}\chi_{\nub}(\zb)|\zb|w(\zb)\frac{\zb}{|\zb|}\otimes \frac{\zb}{|\zb|} d\zb\\ 
        =&\int_{B_1(\bm 0)}\chi_{-\nub}(\zb)|\zb|w(\zb)\frac{\zb}{|\zb|}\otimes \frac{\zb}{|\zb|} d\zb\qquad(\text{change of variable}\ \zb'=-\zb)\\
        =&\frac{1}{2}\int_{B_1(\bm 0)}\left(\chi_{\nub}(\zb)+\chi_{-\nub}(\zb)\right)|\zb|w(\zb)\frac{\zb}{|\zb|}\otimes \frac{\zb}{|\zb|} d\zb\\
        =&\frac{1}{2}\int_{B_1(\bm 0)}|\zb|w(\zb)\frac{\zb}{|\zb|}\otimes \frac{\zb}{|\zb|} d\zb\qquad( \chi_{\nub}(\zb)+\chi_{-\nub}(\zb)=1)\\
        =&\frac{1}{2}\int_{0}^1 \int_{\S^{d-1}}r^d w(r)\etab\otimes \etab d\etab dr\\
        =&\frac{1}{2}\left(\int_{0}^1 r^d w(r)dr\right)\frac{1}{d}\omega_{d-1}I_d\\
        =&\frac{M_w^1}{2d}I_d.
    \end{split}
    \]
    Here $\omega_{d-1}$ is the surface area of $(d-1)$-sphere $\S^{d-1}$ and $I_d$ is the $d\times d$ identity matrix.

    One may further show that the localizations of these nonlocal operators are their local counterparts multiplied by a constant $\frac{M_w^1}{2d}$, which justifies this definition. Specially, let $w_\del(\xb)=\frac{1}{\del^{d+1}}w(\frac{\xb}{\del})$ and $u\in C^2_c(\R^d)$, then by Taylor expansion one can prove that \[
    \mathcal{G}^{\nub}_{w_\del} u(\xb)\to \frac{M_w^1}{2d}\nabla u(\xb),\quad \del\to 0,\quad \forall \xb\in\R^d,
    \]
    where $\nabla u(\xb)=\left(\frac{\partial u_j}{\partial x_i}(\xb)\right)_{1\le i\le d,\, 1\le j\le N}$ is the gradient matrix of $u$, i.e., the transpose of the Jacobian matrix of $u$. Similarly, for $\vb\in C^2_c(\R^d;\R^N)$, \[
    \mathcal{D}^{\nub}_{w_\del}\vb(\xb)\to \frac{M_w^1}{2d}\nabla\cdot\vb(\xb),\quad \del\to 0,\quad \forall \xb\in\R^d,
    \]where the divergence vector of $\vb$ given by $\nabla\cdot \vb(\xb)=\left(\sum_{j=1}^d \frac{\partial v_{ji}}{\partial x_j}\right)_{1\le i\le N}$ is a column vector in $\R^N$,
    and if $d=N=3$,
    \[
    \mathcal{C}^{\nub}_{w_\del}\vb(\xb)\to \frac{M_w^1}{2d}\mathrm{curl}\, \vb(\xb),\quad \del\to 0,\quad \forall \xb\in\R^3.
    \]
\end{rem}

Note that in \Cref{def:operators_PV}, the integrals are understood in the principal value sense. For smooth functions with compact support, the above integrals are just Lebesgue integrals, and moreover, the action of nonlocal operators yields smooth functions whose derivatives are $L^p$ functions for $1\le p\le \infty$.
We summarize these results in the following lemma. 
The proof is similar to that of Proposition 1 in \cite{d2022connections} and hence omitted.  

\begin{lem}\label{lem:opforsmoothfct}
Suppose that $u\in C^\infty_c(\R^d)$ and $\vb\in C^\infty_c(\R^d;\R^d)$. Then $\mathcal{G}^{\nub}_w u$, $\mathcal{D}_w^{\nub} \vb$ and $\mathcal{C}_w^{\nub} \vb$ ($d=3$) are $C^\infty$ functions with
\begin{equation}\label{Gwithoutpv}
    \mathcal{G}^{\nub}_w u(\xb)=\int_{\R^d}\chi_{\nub}(\yb-\xb)\frac{\yb-\xb}{|\yb-\xb|}(u(\yb)-u(\xb))w(\yb-\xb)d\yb,\quad \xb\in \R^d,
\end{equation}
\begin{equation}\label{Dwithoutpv}
    \mathcal{D}_w^{\nub} \vb(\xb)=\int_{\R^d}\chi_{\nub}(\yb-\xb)\frac{\yb-\xb}{|\yb-\xb|}\cdot (\vb(\yb)-\vb(\xb))w(\yb-\xb)d\yb,\quad\xb\in \R^d,
\end{equation}
and if $d=3$,
\begin{equation}\label{Cwithoutpv}
    \mathcal{C}_w^{\nub} \vb(\xb)=\int_{\R^3}\chi_{\nub}(\yb-\xb)\frac{\yb-\xb}{|\yb-\xb|}\times(\vb(\yb)-\vb(\xb))w(\yb-\xb)d\yb,\quad \xb\in \R^3.
\end{equation}
For $p\in [1,\infty]$ and multi-index $\alpha\in \N^d$, there is a constant $C$ depending on $p$ such that the following estimates hold:

\begin{equation}\label{Lpestgradsmooth}
    \|D^\alpha\mathcal{G}_w^{\nub} u\|_{L^p(\R^d;\R^d)}\le C\left( M_w^1\|\nabla D^\alpha u\|_{L^p(\R^d;\R^d)}+M_w^2\|D^\alpha u\|_{L^p(\R^d)}\right),
\end{equation}
\begin{equation}\label{Lpestdivsmooth}
    \|D^\alpha\mathcal{D}_w^{\nub} \vb\|_{L^p(\R^d)} \le C\left( M_w^1\|\nabla D^\alpha \vb\|_{L^p(\R^d;\R^{d\times d})}+M_w^2\|D^\alpha\vb\|_{L^p(\R^d;\R^d)}\right),
\end{equation}
and if $d=3$,
\begin{equation}\label{Lpestcurlsmooth}
    \|D^\alpha\mathcal{C}_w^{\nub} \vb\|_{L^p(\R^3;\R^3)}\le C\left( M_w^1\|\nabla D^\alpha\vb\|_{L^p(\R^3;\R^{3\times 3})}+M_w^2\|D^\alpha\vb\|_{L^p(\R^3;\R^3)}\right).
\end{equation}
\end{lem}

If we replace smooth functions with compact support by $W^{1,p}$ functions, then the action of nonlocal operators still yield $L^p$ functions and the equalities \eqref{Gwithoutpv}-\eqref{Cwithoutpv} hold for a.e. $x\in\R^d$. The proof uses some ideas of Proposition 2.1(2) in \cite{mengesha2016characterization} and is left to the appendix.
\begin{lem}\label{lem:opforW1pfct}
    Let $p\in [1,\infty]$. Then $\mathcal{G}^{\nub}_w:W^{1,p}(\R^d)\to L^p(\R^d;\R^d)$, $\mathcal{D}^{\nub}_w:W^{1,p}(\R^d;\R^d)\to L^p(\R^d)$ and $\mathcal{C}^{\nub}_w:W^{1,p}(\R^3;\R^3)\to L^p(\R^3;\R^3)$ are bounded linear operators. Moreover, there exists a constant $C>0$ depending on $p$ such that
    \begin{equation}\label{LpestgradW1p}
    \|\mathcal{G}_w^{\nub} u\|_{L^p(\R^d;\R^d)}\le C\left( M_w^1\|\nabla  u\|_{L^p(\R^d;\R^d)}+M_w^2\| u\|_{L^p(\R^d)}\right),\quad u\in W^{1,p}(\R^d),
\end{equation}
\begin{equation}\label{LpestdivW1p}
    \|\mathcal{D}_w^{\nub} \vb\|_{L^p(\R^d)} \le C\left( M_w^1\|\nabla  \vb\|_{L^p(\R^d;\R^{d\times d})}+M_w^2\|\vb\|_{L^p(\R^d;\R^d)}\right),\quad \vb\in W^{1,p}(\R^d;\R^d),
\end{equation}
and if $d=3$,
\begin{equation}\label{LpestcurlW1p}
    \|\mathcal{C}_w^{\nub} \vb\|_{L^p(\R^3;\R^3)}\le C\left( M_w^1\|\nabla \vb\|_{L^p(\R^3;\R^{3\times 3})}+M_w^2\|\vb\|_{L^p(\R^3;\R^3)}\right), \quad \vb\in W^{1,p}(\R^3;\R^3).
\end{equation}
In addition, equalities \eqref{Gwithoutpv}-\eqref{Cwithoutpv} hold for a.e. $x\in\R^d$.
\end{lem}

Analogous to the local operator, the integration by parts formula holds. Here we provide three types of integration by parts with proofs in the appendix. Note that the corresponding conditions in 
\Cref{prop:hbibp} (1)(2)(3) hold provided $\ub\in C^\infty_c(\R^d;\R^N)$, $\ub\in C^\infty_c(\R^d;\R^{d\times N})$ and $\ub\in C^\infty_c(\R^3;\R^3)$, respectively.

\begin{prop}[Nonlocal ``half-ball" integration by parts]\label{prop:hbibp}\hfill
\begin{enumerate}
    \item  Suppose $\ub\in L^1(\R^d;\R^N)$, and $w(\xb-\yb)\left|\ub(\xb)-\ub(\yb)\right|\in L^1(\R^d\times \R^d)$. Then $\mathcal{G}^{\nub}_w \ub\in L^1(\R^d;\R^{d\times N})$ and for any $\vb\in C_c^1(\R^d;\R^{d\times N})$,
\begin{equation}
    \int_{\R^d} \mathcal{G}^{\nub}_w \ub(\xb):\vb(\xb)d\xb=-\int_{\R^d} \ub(\xb)\cdot\mathcal{D}^{-\nub}_w \vb(\xb)d\xb.
\end{equation}
    \item Suppose $\ub\in L^1(\R^d;\R^{d\times N})$, and $w(\xb-\yb)\left|\ub(\xb)-\ub(\yb)\right|\in L^1(\R^d\times \R^d)$. Then $\mathcal{D}^{\nub}_w \ub\in L^1(\R^d;\R^N)$ and for any $\vb\in C_c^1(\R^d;\R^N)$,
\begin{equation}
    \int_{\R^d} \mathcal{D}^{\nub}_w \ub(\xb)\cdot\vb(\xb)d\xb=-\int_{\R^d} \ub(\xb):\mathcal{G}^{-\nub}_w \vb(\xb)d\xb.
\end{equation}
    \item  Let $d=3$. Suppose $\ub\in L^1(\R^3;\R^3)$, and $w(\xb-\yb)\left|\ub(\xb)-\ub(\yb)\right|\in L^1(\R^3\times \R^3)$. Then $\mathcal{C}^{\nub}_w \ub\in L^1(\R^3;\R^3)$ and for any $\vb\in C_c^1(\R^3;\R^3)$,
\begin{equation}
    \int_{\R^3} \mathcal{C}^{\nub}_w \ub(\xb)\cdot\vb(\xb)d\xb=\int_{\R^3} \ub(\xb)\cdot\mathcal{C}^{-\nub}_w \vb(\xb)d\xb.
\end{equation}  
\end{enumerate}
\end{prop}

\begin{rem}
\label{rem:div_equiv}
As seen from the proof of \Cref{prop:hbibp} in the appendix, 
an equivalent definition of the divergence operator in \cref{hbdiv} is given as
\[
 \mathcal{D}_w^{\nub} \vb(\xb)=\lim_{\epsilon\to 0}\int_{\R^d\backslash B_\epsilon(\xb)}\left[\frac{\yb^T-\xb^T}{|\yb-\xb|} \left(\chi_{\nub}(\yb-\xb)\vb(\yb)+\chi_{\nub}(\xb-\yb) \vb(\xb)\right)\right]^T w(\yb-\xb)d\yb,
\]
for $\xb\in \R^d$.
\end{rem}
We point out that nonlocal gradient, divergence and curl can be defined for complex-valued functions via \cref{hbgrad}, \cref{hbdiv} and \cref{hbcurl}, respectively, where the dot product in \cref{hbdiv} is understood as the inner product in $\C^d$, i.e., $\zb\cdot\wb:=\zb^T\overline{\wb}$ for $\zb,\wb\in\C^d$, and the cross product in \cref{hbcurl} is understood as the cross product in $\C^d$. This extension will be useful in the proof of \Cref{prop:hbgradkerisnull}.

\subsection{Distributional nonlocal operators}
Previously, we defined nonlocal nonsymmetric operators in the principal value sense. It turns out that this notion is not enough to define nonlocal Sobolev spaces. Instead, we need the notion of distributional nonlocal gradient as the notion of weak derivative in the local setting. One way to define it is via its adjoint operator, i.e., nonlocal nonsymmetric divergence operator defined in the last subsection.

Following the idea in \cite{mengesha2016characterization}, we define the distributional nonlocal operators as follows.

\begin{defn}
Let $1\le p\le\infty$. Given $\ub\in L^p(\R^d;\R^N)$, we define the \textbf{distributional nonlocal gradient} $\mathfrak{G}_w^{\nub} \ub\in (C^\infty_c(\R^d;\R^{d\times N}))'$ as
\begin{equation}
    \langle \mathfrak{G}_w^{\nub} \ub,\bm\phi\rangle:=-\int_{\R^d} \ub(\xb)\cdot\mathcal{D}_w^{-\nub}\bm\phi(\xb)d\xb ,\quad\forall \bm\phi\in C_c^\infty(\R^d;\R^{d\times N}).
\end{equation}
Given $\ub\in L^p(\R^d;\R^{d\times N})$, we define the \textbf{distributional nonlocal divergence} $\mathfrak{D}_w^{\nub} \ub\in (C^\infty_c(\R^d;\R^N))'$ as
\begin{equation}
    \langle \mathfrak{D}_w^{\nub} \ub,\bm\phi\rangle:=-\int_{\R^d} \ub(\xb):\mathcal{G}_w^{-\nub}\bm\phi(\xb)d\xb ,\quad\forall \bm\phi\in C_c^\infty(\R^d;\R^N).
\end{equation}
 If $d=3$ with $\ub\in L^p(\R^3;\R^3)$, we define the \textbf{distributional nonlocal curl} $\mathfrak{C}_w^{\nub} \ub\in (C^\infty_c(\R^3;\R^3))'$ as
\begin{equation}
    \langle \mathfrak{C}_w^{\nub} \ub,\bm\phi\rangle:=\int_{\R^3} \ub(\xb)\cdot\mathcal{C}_w^{-\nub}\bm\phi(\xb)d\xb ,\quad\forall \bm\phi\in C_c^\infty(\R^3;\R^3).
\end{equation} 
\end{defn}

\begin{rem}
For $u\in L^p(\R^d)$, $\mathfrak{G}_w^{\nub} u$ is indeed a distribution as for any compact set $K\subset \R^d$ and $\bm\phi\in C_c^\infty(\R^d;\R^{d})$ with support contained in $K$,
\begin{align*}
    |\langle \mathfrak{G}_w^{\nub} u,\bm\phi\rangle|&\le \|u\|_{L^p(\R^d)}\|\mathcal{D}_w^{-\nub}\bm\phi\|_{L^{p'}(\R^d)} \\
    &\le 
    C\left( M_w^1\|\nabla \bm\phi\|_{L^{p'}(\R^d;\R^{d\times d})}+M_w^2\|\bm\phi\|_{L^{p'}(\R^d;\R^d)}\right)\|u\|_{L^p(\R^d)}\\
    &\le C|K|^{\frac{1}{p'}}\left( M_w^1\|\nabla \bm\phi\|_{L^{\infty}(\R^d;\R^{d\times d})}+M_w^2\|\bm\phi\|_{L^{\infty}(\R^d;\R^d)}\right)\|u\|_{L^p(\R^d)},
\end{align*}
where $p'=\frac{p}{p-1}$ ($p'=\infty$ for $p=1$ and $p'=1$ for $p=\infty$) and \cref{Lpestdivsmooth} is used in the above inequalities. Similarly, it can be shown that $\mathfrak{D}^{\nub}_w  \ub$ and $\mathfrak{C}^{\nub}_w  \ub$ are distributions using \cref{Lpestgradsmooth} and \cref{Lpestcurlsmooth}.
\end{rem}

From the integration by parts formulas in \Cref{prop:hbibp}, we immediately have the following results when the distributional operators $\mathfrak{G}^{\nub}_w \ub$, $\mathfrak{D}^{\nub}_w  \ub$ and $\mathfrak{C}^{\nub}_w  \ub$ coincide with  $\mathcal{G}^{\nub}_w \ub$, $\mathcal{D}^{\nub}_w \ub$ and $\mathcal{C}^{\nub}_w \ub$, respectively. 

\begin{coro}\label{cor:pvopeqdistop}\hfill
\begin{enumerate}
    \item Suppose $\ub\in L^1(\R^d;\R^N)$ and $w(\xb-\yb)\left|\ub(\xb)-\ub(\yb)\right|\in L^1(\R^d\times \R^d)$,  then $\mathcal{G}^{\nub}_w \ub=\mathfrak{G}^{\nub}_w \ub$ in $L^1(\R^d;\R^{d\times N})$.
    \item Suppose $\ub\in L^1(\R^d;\R^{d\times N})$ and $w(\xb-\yb)\left|\ub(\xb)-\ub(\yb)\right|\in L^1(\R^d\times \R^d)$,  then $\mathcal{D}^{\nub}_w \ub=\mathfrak{D}^{\nub}_w  \ub$ in $L^1(\R^d;\R^N)$.
    \item Suppose $\ub\in L^1(\R^3;\R^3)$ and $w(\xb-\yb)\left|\ub(\xb)-\ub(\yb)\right|\in L^1(\R^3\times \R^3)$,  then $\mathcal{C}^{\nub}_w \ub=\mathfrak{C}^{\nub}_w  \ub$ in $L^1(\R^3)$.
\end{enumerate}
\end{coro}

\subsection{Fourier symbols of nonlocal operators}
In this subsection, we study the Fourier symbols of nonlocal operators defined in the previous subsection. These results will be used in the analysis in the subsequent sections. 

Define 
\begin{equation}\label{deffouriersymboldiv}
\bm\lambda_w^{\nub}(\bm\xi):=\int_{\mathbb{R}^d}\chi_{\nub}(\zb)\frac{\zb}{|\zb|}w(\zb)(e^{2\pi i\bm\xi\cdot \zb}-1)d\zb,\quad\bm\xib\in\R^d.
\end{equation}
It is immediate that $\bm\lambda_w^{-\nub}(\bm\xi)=-\overline{\bm\lambda_w^{\nub}(\bm\xi)}$ for $\bm\xib\in\R^d$. In fact,  $\bm\lambda_w^{\nub}$ is the Fourier symbol of  $\mathcal{G}_w^{\nub}$, $\mathcal{D}^{\nub}_w $ and $\mathcal{C}_w^{\nub}$ in the sense described below.
We now present this fact without proof since the proof is straightforward. Indeed, first prove the result for smooth functions with compact support and then use \eqref{LpestgradW1p}-\eqref{LpestcurlW1p} for $p=2$ and density of $C^\infty_c(\R^d)$ in $H^1(\R^d)$. Similar results can also be found in \cite{lee2020nonlocal}.

\begin{lem}\label{lem:FT}
Let $\ub\in H^1(\mathbb{R}^d;\R^N)$ and $\vb\in H^1(\mathbb{R}^d;\R^{d\times N})$. 
The Fourier transform of the nonlocal gradient operator $\mathcal{G}_w^{\nub}$ acting on $\ub$ is given by
\begin{equation}\label{eq:fourierhbgrad}
    \mathcal{F}(\mathcal{G}_w^{\nub} \ub)(\bm\xi)=\bm\lambda_w^{\nub}(\bm\xi)\otimes \hat{\ub}(\bm\xi),\quad\bm \xi\in\R^d,
\end{equation}
and the Fourier transform of the nonlocal divergence operator $\mathcal{D}^{\nub}_w $ acting on $\vb$ is given by
\begin{equation}\label{eq:fourierhbdiv}
    \mathcal{F}(\mathcal{D}^{\nub}_w  \vb)(\bm\xi)=\left(\bm\lambda_w^{\nub}(\bm\xi)^T\hat{\vb}(\bm\xi)\right)^T,\quad\bm \xi\in\R^d.
\end{equation}
If, in particular, $d=3$ and  $\vb\in H^1(\mathbb{R}^3;\R^3)$, then the Fourier transform of the nonlocal curl operator $\mathcal{C}_w^{\nub}$ acting on $\vb$ is given by
\begin{equation}\label{eq:fourierhbcurl}
    \mathcal{F}(\mathcal{C}_w^{\nub} \vb)(\bm\xi)=\bm\lambda_w^{\nub}(\bm\xi)\times\hat{\vb}(\bm\xi),\quad\bm \xi\in\R^3.
\end{equation}
\end{lem}

Now we write out the real and imaginary part of $\bm\lambda_w^{\nub}(\bm\xi)$ explicitly and show that the imaginary part is a scalar multiple of $\bm\xi$. Moreover, the upper bound of $\bm\lambda_w^{\nub}(\bm\xi)$ is linear in $|\bm\xi|$. The proof of the following lemma is omitted since it follows from Lemma 2.3 and the last part of Theorem 2.4 in \cite{lee2020nonlocal}.
\begin{lem}\label{lem:lambdaexplicitform}
The Fourier symbol $\bm\lambda_w^{\nub}(\bm\xi)$ can be expressed as \[\bm\lambda_w^{\nub}(\bm\xi)=\Re (\bm\lambda_w^{\nub})(\bm\xi)+i \Im(\bm\lambda_w^{\nub})(\bm\xi),\]
where \begin{equation}\label{lambdareal}
    \Re (\bm\lambda_w^{\nub})(\bm\xi)=\int_{\mathbb{R}^d}\chi_{\nub}(\zb)\frac{\zb}{|\zb|}w(\zb)(\cos(2\pi \bm\xi\cdot \zb)-1)d\zb,
\end{equation}
\begin{equation}\label{lambdaim}
    \Im (\bm\lambda_w^{\nub})(\bm\xi)=\int_{\mathbb{R}^d}\chi_{\nub}(\zb)\frac{\zb}{|\zb|}w(\zb)\sin(2\pi \bm\xi\cdot \zb)d\zb,
\end{equation}
and $\Im (\bm\lambda_w^{\nub})(\bm\xi)=\Lambda_w(|\bm\xi|)\frac{\bm\xi}{|\bm\xi|}$ with
\begin{equation}\label{defLambda}
    \Lambda_w(|\bm\xi|)=\frac{1}{2}\int_{\R^d} \frac{w(\zb)}{|\zb|}z_1\sin(2\pi |\bm\xi|z_1)d\zb.
\end{equation}
Moreover, 
\begin{equation}\label{lambdaupbdd}
    \left|\bm\lambda_w^{\nub}(\bm\xi)\right|\le \sqrt{2}\left(2\pi M_w^1|\bm\xi|+M_w^2
\right),\quad\forall\bm\xib\in\R^d.
\end{equation}
\end{lem}

In the following, we present two other observations of the Fourier symbol $\bm\lambda_w^{\nub}$ that are useful in \Cref{sec:Poincare_int_kernel}. 
The first result concerns the positivity of $|\bm\lambda_w^{\nub}|$ away from the origin, and the second result asserts that $\bm\lambda_w^{\nub}$ is a smooth function.

\begin{prop}\label{prop:lambdatransandpositive}
For every $d\times d$ orthogonal matrix $R$,
\begin{equation}\label{eq:foursymborthotrans}
    \bm\lambda_w^{R\nub}(\bm\xi)=R\bm\lambda_w^{\nub}(R^T\bm\xi),\quad\forall\bm\xi\neq \bm 0.
\end{equation}
The same formula holds for both $\Re (\bm\lambda_w^{\nub})(\bm\xi)$ and $\Im (\bm\lambda_w^{\nub})(\bm\xi)$. Consequently, 
\begin{equation}\label{lambpost}
    \left|\bm\lambda_w^{\nub}(\bm\xi)\right|>0,\quad\forall\bm\xi\neq \bm 0.
\end{equation}
\end{prop}
\begin{proof}
\Cref{eq:foursymborthotrans} can be easily seen from a change of variable. For a fixed unit vector $\nub\in\mathbb{R}^d$, there exists an orthogonal matrix $R_{\nub}$ such that $\nub=R_{\nub}\eb_1$. By \eqref{eq:foursymborthotrans}, for $\xib\neq \bm 0$,
\begin{align*}
    |{\bm\lambda}_w^{\nub}(\xib)|&=|R_{\nub}{\bm\lambda}_w^{\eb_1}(R_{\nub}^T\xib)|=|{\bm\lambda}_w^{\eb_1}(R_{\nub}^T\xib)|\ge |{\bm\lambda}_w^{\eb_1}(R_{\nub}^T\xib)\cdot \eb_1|\ge |\Re({\bm\lambda}_w^{\eb_1}(R_{\nub}^T\xib)\cdot \eb_1)|\\
    &=\int_{\{z_1>0\}}\frac{z_1}{|\zb |}w(\zb )\left(1-\mathrm{cos}\left(2\pi(R_{\nub}^T\xib\cdot \zb \right)\right)d\zb >0,
\end{align*}
where the last inequality holds because the integrand is nonnegative and the set 
\[\{\zb \in\mathbb{R}^d:z_1>0,\ (R_{\nub}^T\xib)\cdot \zb \in\mathbb{Z}\}\] is a set of measure zero in $\mathbb{R}^d$. Thus, \eqref{lambpost} holds.
\end{proof}

\begin{prop}\label{prop:fouriersymsmooth}
Suppose the kernel function $w$ satisfies \cref{eq:kernelassumption}, and in addition, the support of $w$ is a compact set in $\R^d$.
Then the Fourier symbol $\bm\lambda_w^{\nub}\in C^\infty(\mathbb{R}^d;\mathbb{C}^d)$.
\end{prop}
\begin{proof}
Notice that for any multi-index $\bm\gamma$ with $|\bm\gamma|>0$,
\[D^{\bm\gamma}({\bm\lambda}_w^{\nub})(\xib)=\int_{\R^d}\chi_{\nub}(\zb)\frac{\zb}{|\zb|}w(\zb)(2\pi i\zb)^{\bm\gamma} e^{2\pi i\xib\cdot \zb} d\zb.\]
Since $w$ is a compactly supported kernel function, the integrand on the right-hand side of the above equation can be controlled by the integrable function $|\zb|w(\zb)$. Hence, ${\bm\lambda}_w^{\nub}\in C^\infty(\mathbb{R}^d;\mathbb{C}^d)$.
\end{proof}

\subsection{Nonlocal vector identities for smooth functions }\label{sec:nvismoothfct}
In this subsection, we present some nonlocal vector identities for smooth functions with compact support. These results will be generalized for a larger class of functions in \Cref{sec:space} and become crucial for applications in \Cref{sec:applications}.

The following lemma shows that $\mathcal{C}^{\nub}_w \circ\mathcal{G}^{\nub}_w =0$ and $\mathcal{D}^{\nub}_w \circ \mathcal{C}^{\nub}_w =0$, analogous to $\mathrm{curl}\circ\mathrm{grad}=0$ and $\mathrm{div}\circ\mathrm{curl}=0$ in the local setting.  
\begin{lem}\label{lem:vanishidentity}
Let $d=3$. Then for $u\in C^\infty_c(\R^3)$ and $\vb\in C^\infty_c(\R^3;\R^3)$,
\begin{equation}\label{curlgradeq0}
    \mathcal{C}^{\nub}_w \mathcal{G}^{\nub}_w u(\xb)=0,\quad\text{a.e.}\ \xb\in\R^3,
\end{equation} and
\begin{equation}\label{divcurleq0}
    \mathcal{D}^{\nub}_w \mathcal{C}^{\nub}_w \vb(\xb)=0,\quad\text{a.e.}\ \xb\in\R^3.
\end{equation}
\end{lem}
\begin{proof}
First note that by \Cref{lem:opforsmoothfct}, $\mathcal{G}^{\nub}_w u\in H^1(\R^3;\R^3)$ and $\mathcal{C}^{\nub}_w \vb\in H^1(\R^3;\R^3)$. Then the conditions for \Cref{lem:FT} hold and one can apply the Fourier transform to $L^2$ functions $\mathcal{C}^{\nub}_w \mathcal{G}^{\nub}_w u$ and $\mathcal{D}^{\nub}_w \mathcal{C}^{\nub}_w \vb$. By \Cref{lem:FT}, \cref{curlgradeq0} and \cref{divcurleq0} follows from \[\bm\lambda_w^{\nub}\times(\bm\lambda_w^{\nub}\hat{u})=0\] and \[(\bm\lambda_w^{\nub})^T(\bm\lambda_w^{\nub}\times\hat{\vb})=0\]respectively.
\end{proof}

Next, we show two nonlocal vector identities analogous to the following vector calculus identities in local setting\footnote{In \cref{veccalid1loc}, the two types of curls in 2D are defined as  
\begin{equation*}
    \mathrm{Curl}\ \vb:=\frac{\partial v_2}{\partial x_1}-\frac{\partial v_1}{\partial x_2} \text{ and } \textbf{Curl}\ \phi:=\left(\frac{\partial \phi}{\partial x_2},-\frac{\partial \phi}{\partial x_1}\right)^T,
\end{equation*}
for a vector field $\vb$ and a scalar field $\phi$. In \cref{veccalid2loc}, the curl of a vector field $\vb$ in 3D is defined as 
\[\textbf{Curl}\ \vb:=\left(\frac{\partial v_3}{\partial x_2}-\frac{\partial v_2}{\partial x_3},\frac{\partial v_1}{\partial x_3}-\frac{\partial v_1}{\partial x_3},\frac{\partial v_2}{\partial x_1}-\frac{\partial v_1}{\partial x_2}\right)^T.\]}:

\begin{equation}\label{veccalid1loc}
    \nabla\cdot(\nabla\vb)=\nabla(\nabla\cdot \vb)-\textbf{Curl}\ \mathrm{Curl}\ \vb,\quad d=2;
\end{equation}
\begin{equation}\label{veccalid2loc}
    \nabla\cdot(\nabla\vb)=\nabla(\nabla\cdot \vb)-\textbf{Curl}\ \textbf{Curl}\ \vb,\quad d=3.
\end{equation}

\begin{lem}\label{lem:veccalid1}
For $\ub\in C^\infty_c(\R^2;\R^2)$,
\begin{equation}\label{veccalid1}
    \mathcal{D}_w^{-\nub}\mathcal{G}_w^{\nub} \ub  =\mathcal{G}_w^{\nub} \mathcal{D}_w^{-\nub} \ub-
    \begin{pmatrix}
    0 & 1\\
    -1 & 0
    \end{pmatrix}
    \mathcal{G}_w^{-\nub} \mathcal{D}_w^{\nub} \left[\begin{pmatrix}
    0 & 1\\
    -1 & 0
    \end{pmatrix}\ub\right].
\end{equation}
 \end{lem}
 \begin{proof} As remarked at the beginning of the proof of \Cref{lem:vanishidentity}, it is valid to apply the Fourier transform. Applying the Fourier transform and \Cref{lem:FT},  the left hand side of \cref{veccalid1} becomes $-|\bm\lambda_w^{\nub}(\xib)|^2\hat{\ub}(\xib)$ and the right hand side becomes 
\begin{align*}
    &\quad\ \bm\lambda_w^{\nub} (\xib) \bm\lambda_w^{-\nub} (\xib)^T \hat{\ub}(\xib)-
    \begin{pmatrix}
    0 & 1\\
    -1 & 0
    \end{pmatrix}
    \bm\lambda_w^{-\nub} (\xib)\bm\lambda_w^{\nub} (\xib)^T \begin{pmatrix}
    0 & 1\\
    -1 & 0
    \end{pmatrix} \hat{\ub}(\xib)\\
    &=\bigg[-\begin{pmatrix}
    \lambda_1(\xib)\overline{\lambda_1(\xib)} & \lambda_1(\xib)\overline{\lambda_2(\xib)}\\
    \lambda_2(\xib)\overline{\lambda_1(\xib)} & \lambda_2(\xib)\overline{\lambda_2(\xib)}
    \end{pmatrix}\\
    &\qquad\qquad\qquad\quad +\begin{pmatrix}
    0 & 1\\
    -1 & 0
    \end{pmatrix}\begin{pmatrix}
    \overline{\lambda_1(\xib)}\lambda_1(\xib) & \overline{\lambda_1(\xib)}\lambda_2(\xib)\\
    \overline{\lambda_2(\xib)}\lambda_1(\xib) & \overline{\lambda_2(\xib)}\lambda_2(\xib)
    \end{pmatrix}\begin{pmatrix}
    0 & 1\\
    -1 & 0
    \end{pmatrix}\bigg]\hat{\ub}(\xib)
    \\
    &=\left[-\begin{pmatrix}
    \lambda_1(\xib)\overline{\lambda_1(\xib)} & \lambda_1(\xib)\overline{\lambda_2(\xib)}\\
    \lambda_2(\xib)\overline{\lambda_1(\xib)} & \lambda_2(\xib)\overline{\lambda_2(\xib)}
    \end{pmatrix}+\begin{pmatrix}
    -\overline{\lambda_2(\xib)}\lambda_2(\xib) & \overline{\lambda_2(\xib)}\lambda_1(\xib)\\
    \overline{\lambda_1(\xib)}\lambda_2(\xib) & -\overline{\lambda_1(\xib)}\lambda_1(\xib)
    \end{pmatrix}\right]\hat{\ub}(\xib)\\
    &=-|\bm\lambda_w^{\nub} (\xib)|^2\hat{\ub}(\xib).
\end{align*}
Therefore, \cref{veccalid1}
 holds for $\ub\in C_c^\infty(\R^2;\R^2)$.
 \end{proof}
 
 \begin{lem}\label{lem:veccalid2}
For $\ub\in C^\infty_c(\R^3;\R^3)$,
\begin{equation}\label{veccalid2}
     \mathcal{D}_w^{-\nub}\mathcal{G}_w^{\nub} \ub =\mathcal{G}_w^{\nub} \mathcal{D}_w^{-\nub} \ub-\mathcal{C}_{w}^{-\nub}\mathcal{C}_w^{\nub} \ub.
\end{equation}
\end{lem}
\begin{proof}
Applying the Fourier transform to \cref{veccalid2} and using \Cref{lem:FT} yield
\begin{align*}
    &\quad\ \mathcal{F}(\mathcal{G}_w^{\nub} \mathcal{D}_w^{-\nub} \ub-\mathcal{C}_{w}^{-\nub} \mathcal{C}_w^{\nub}\ub)(\xib)\\
    &=\bm\lambda_w^{\nub} (\xib) \bm\lambda_w^{-\nub} (\xib)^T \hat{\ub}(\xib)-\bm\lambda_w^{-\nub} (\xib)\times(\bm\lambda_w^{\nub}(\xib)\times\hat{\ub}(\xib))\\
    &=\bm\lambda_w^{\nub} (\xib) \bm\lambda_w^{-\nub} (\xib)^T \hat{\ub}(\xib)-(\bm\lambda_w^{-\nub} (\xib)^T\hat{\ub}(\xib))\bm\lambda_w^{\nub}(\xib)+\bm\lambda_w^{-\nub} (\xib)^T\bm\lambda_w^{\nub}(\xib)\hat{\ub}(\xib)\\
    &=-|\bm\lambda_w^{\nub} (\xib)|^2\hat{\ub}(\xib)=\mathcal{F}(\mathcal{D}_w^{-\nub}(\mathcal{G}_w^{\nub} \ub))(\xib),
\end{align*}
where we used $\bm\lambda_w^{-\nub}(\xib)=-\overline{\bm\lambda_w^{\nub} (\xib)}$.
\end{proof}

\section{Nonlocal Sobolev-type  spaces}
\label{sec:space}
In this section, we define the nonlocal Sobolev-type spaces in which we prove the Poincar\'e inequality. The notion is defined via the distributional nonlocal gradient introduced in the previous section, motivated by the definition of classical Sobolev spaces. A similar notion was introduced in \cite{Stefani20} for fractional gradient.
For simplicity, we only consider the case $p=2$, while the definitions and results in \Cref{subsec:def_space} below can be extended to a general $p\in [1,\infty)$.

\subsection{Definitions and properties of nonlocal Sobolev-type spaces}
\label{subsec:def_space}
For the rest of the paper, we adopt the convention that a domain is an open connected set (not necessarily bounded). Let $\Omega\subset \mathbb{R}^d$ be a domain and $N\in \Z^+ $ a positive integer. 
Given a kernel function $w$ satisfying \cref{eq:kernelassumption} and a unit vector $\nub\in\R^d$, define the associated energy space $\cS_w^{\nub}(\Omega;\R^N)$ by
\begin{equation}\label{deffctsp_2}
    \cS_w^{\nub}(\Omega;\R^N):=\{\ub\in L^2(\mathbb{R}^d;\R^N): \ub=\bm{0} \text{ a.e. on }\R^d\backslash\Omega, \; \mathfrak{G}^{\nub}_w \ub\in L^2(\R^d;\R^{d\times N})\},
\end{equation}
equipped with norm \[\|\ub\|_{\cS_w^{\nub}(\Om;\R^N)} :=\left(\|\ub\|_{L^2(\R^d;\R^N)}^2+\|\mathfrak{G}^{\nub}_w \ub\|_{L^2(\R^d;\R^{d\times N})}^2\right)^{1/2},\]
as well as the corresponding inner product. 
For any $\Om\subset\R^d$, it is not hard to see that $\cS_w^{\nub}(\Om;\R^N)$ is a closed subspace of $\cS_w^{\nub}(\R^d;\R^N)$.
When $N=1$, we simply denote $\cS_w^{\nub}(\Om):=\cS_w^{\nub}(\Om;\R)$. Notice that any function in $\cS_w^{\nub}(\Om;\R^N)$ is a vector field where each component of it is a function in $\cS_w^{\nub}(\Om)$. 
For the rest of this section, we will show $\cS_w^{\nub}(\Om;\R^N)$ is a separable Hilbert space for certain domain $\Omega$.
Since functions in
$\cS_w^{\nub}(\Om;\R^N)$ can be understood componentwise as functions in $\cS_w^{\nub}(\Om)$, we will work with $\cS_w^{\nub}(\Om)$ for the rest of this subsection and the following results also hold for $\cS_w^{\nub}(\Om;\R^N)$ where $N\in\Z^+$. 
The results of this subsection can also be easily extended to a general $p\in [1,\infty)$.

\begin{rem}
    $\cS_w^{\nub}(\Om)$ is a nonlocal analogue of the Sobolev space $H_0^1(\Om)$.
 If the kernel function $w$ has compact support, e.g., $\supp\ w \subset B_\del(\bm{0})$ for $\del>0$, then $\mathfrak{G}_w^{\nub}u$ vanishes outside $\Om_\del:=\{\xb\in \R^d:\dist(\xb,\Om)<\del \}$. In this case, we may equivalently define $\cS_w^{\nub}(\Om)$ as functions in $L^2(\Om_{2\del})$ that vanish on $\Om_{2\del}\backslash\Om$ with $\mathfrak{G}_w^{\nub}u\in L^2(\Om_\del;\R^d)$.   
\end{rem}

\begin{thm}\label{thm:fctspcmplt} 
Let $\Om\subset\R^d$ be a domain.
The function space $\cS_w^{\nub}(\Omega)$ is a Hilbert space.
\end{thm}
\begin{proof}
It suffices to prove that $\cS_w^{\nub}(\R^d)$ is complete. Let $\{u_k\}_{k\in\mathbb{N}}$ be a Cauchy sequence in $\cS_w^{\nub}(\R^d)$. Since $\{u_k\}_{k\in\mathbb{N}}$ is a Cauchy sequence in $L^2(\R^d)$, there exists $u\in L^2(\R^d)$ such that $u_k\to u$ in $L^2(\R^d)$ and $\vb\in L^2(\R^d;\R^d)$ such that $\mathfrak{G}^{\nub}_w u_k\to \vb$ in $L^2(\R^d;\R^d)$. Now we show $\mathfrak{G}^{\nub}_w u=\vb$ in the sense of distributions. By definition, for any $\bm\phi\in C^\infty_c(\R^d;\mathbb{R}^d)$, it suffices to show 
\begin{equation}\label{eq:1}
    -\int_{\R^d} u(\xb)\mathcal{D}^{-\nub}_w \bm\phi(\xb) d\xb=\int_{\R^d}\vb(\xb)\cdot\bm\phi(\xb)d\xb.
\end{equation}
For $k\in\mathbb{N}$, we have
\begin{equation}\label{eq:2}
    -\int_{\R^d} u_k(\xb)\mathcal{D}^{-\nub}_w \bm\phi(\xb) d\xb=\int_{\R^d}\mathfrak{G}^{\nub}_w u_k(\xb)\cdot\bm\phi(\xb)d\xb.
\end{equation}
Since $\bm\phi\in C^\infty_c(\R^d;\mathbb{R}^d)$, by \Cref{lem:opforsmoothfct}, we know $\mathcal{D}^{-\nub}_w \bm\phi\in L^2(\R^d)$. Then taking $k$ to infinity in \eqref{eq:2} yields \eqref{eq:1}. Thus, $\mathfrak{G}^{\nub}_w u=\vb\in L^2(\R^d;\R^d)$ and $u_k\to u$ in $\cS_w^{\nub}(\R^d)$. Hence, $\cS_w^{\nub}(\R^d)$ is a Hilbert space. Since $\cS_w^{\nub}(\Om)$ is a closed subspace of $\cS_w^{\nub}(\R^d)$, the normed space $\cS_w^{\nub}(\Om)$ is also complete. 
\end{proof}

We next present a density result on $\cS_w^{\nub}(\Omega)$ which is crucial in many applications.
If $\Om\neq \R^d$, the density result holds for domains that are bounded with continuous boundaries or epigraphs. We say $\Om$ is an epigraph if there exists a continuous function $\zeta:\R^{d-1}\to\R$ such that (up to a  rigid motion),
\[ \Om = \{ \xb=(\xb',x_d)\in\R^d\,|\, x_d > \zeta(\xb')\}.\]
If $\Om$ is a bounded domain with a continuous boundary, then its boundary can be covered by finitely many balls where each patch is characterized by an epigraph. 

\begin{thm}\label{thm:density}
Let $\Om$ be a bounded domain with a continuous boundary, an epigraph, or $\R^d$. Let $C_c^\infty(\Om)$ denote the space of smooth functions defined on $\R^d$ with compact support contained in $\Om$. Then $C_c^\infty(\Om)$ is dense in $\cS_w^{\nub}(\Omega)$.
\end{thm}

The main ingredients of the proof of \Cref{thm:density} are several lemmas stated below about cut-off, translation and mollification in nonlocal Sobolev spaces which we present in the following. 
First of all, a generalized `product rule' for the nonlocal operators is useful.
\begin{prop}\label{prop:productrulediv}
For $\varphi\in C^\infty_c(\R^d)$ and $\phib\in C^\infty_c(\R^d;\R^d)$,
\begin{equation}
    \mathcal{D}^{-\nub}_w (\varphi\phib)=\varphi\mathcal{D}^{-\nub}_w\phib+\mathcal{G}^{\nub}_w\varphi\cdot\phib+S(\varphi,\phib),
\end{equation}where $S(\varphi,\phib):\R^d\to\R$ is a function given by
\begin{equation}
    S(\varphi,\phib)(\xb):=\int_{\R^d}\frac{\yb-\xb}{|\yb-\xb|}\cdot(\chi_{\nub}(\xb-\yb)\phib(\yb)-\chi_{\nub}(\yb-\xb)\phib(\xb))(\varphi(\yb)-\varphi(\xb))w(\yb-\xb)d\yb .
\end{equation}
Similarly, \[\mathcal{G}^{\nub}_w(\varphi \psi)=\varphi \mathcal{G}^{\nub}_w\psi+\psi \mathcal{G}^{\nub}_w\varphi+S_{\mathcal{G}}(\varphi,\psi),\quad\forall\varphi,\psi\in C^\infty_c(\R^d),\]
    where \[S_{\mathcal{G}}(\varphi,\psi)(\xb):=\int_{\R^d}\chi_{\nub}(\zb)\frac{\zb}{|\zb|}(\varphi(\xb+\zb)-\varphi(\xb))(\psi(\xb+\zb)-\psi(\xb))w(\zb)d\zb,\quad\xb\in\R^d.\]
\end{prop}
\begin{proof}
We only prove the produce rule for $\mathcal{D}^{-\nub}_w$ as the product rule for $\mathcal{G}^{\nub}_w$ is similar and simpler.
First note that the function $S(\varphi,\phib)$ is well-defined with the pointwise estimate
\begin{equation}\label{eq:ptwsS}
    |S(\varphi,\phib)(\xb)|\le \int_{\R^d} 2\|\phib\|_{L^\infty(\R^d;\R^d)} \cdot 2 \|\varphi \|_{W^{1,\infty}(\R^d)}\min(1,|\yb-\xb|)w(\yb-\xb)d\yb<\infty.
\end{equation} Observe that for $\xb,\yb\in\R^d$ and $\epsilon>0$,
\begin{align*}
    &\ \quad\chi_{[\epsilon,\infty)}(|\yb-\xb|)\{\chi_{\nub}(\xb-\yb)\varphi(\yb)\phib(\yb)+\chi_{\nub}(\yb-\xb)\varphi(\xb)\phib(\xb)\}\\
    &=\resizebox{\hsize}{!}
{$\chi_{[\epsilon,\infty)}(|\yb-\xb|)\{[\chi_{\nub}(\xb-\yb)\phib(\yb)+\chi_{\nub}(\yb-\xb)\phib(\xb)]\varphi(\xb)+\chi_{\nub}(\xb-\yb)\phib(\yb)(\varphi(\yb)-\varphi(\xb))\}$}\\
    &=\resizebox{\hsize}{!}
{$\chi_{[\epsilon,\infty)}(|\yb-\xb|)\{[\chi_{\nub}(\xb-\yb)\phib(\yb)+\chi_{\nub}(\yb-\xb)\phib(\xb)]\varphi(\xb)+\chi_{\nub}(\yb-\xb)\phib(\xb)(\varphi(\yb)-\varphi(\xb))$}\\
    &\quad\quad+[\chi_{\nub}(\xb-\yb)\phib(\yb)-\chi_{\nub}(\yb-\xb)\phib(\xb)](\varphi(\yb)-\varphi(\xb))\}.
\end{align*}
Therefore
\begin{align*}
    &\quad\ \int_{\R^d\backslash B_\epsilon(\xb)} \frac{\yb-\xb}{|\yb-\xb|}\cdot(\chi_{\nub}(\xb-\yb)\varphi(\yb)\vb(\yb)+\chi_{\nub}(\yb-\xb)\varphi(\xb)\vb(\xb))w(\yb-\xb)d\yb\\
    &=\varphi(\xb)\int_{\R^d\backslash B_\epsilon(\xb)}\frac{\yb-\xb}{|\yb-\xb|}\cdot(\chi_{\nub}(\xb-\yb)\vb(\yb)+\chi_{\nub}(\yb-\xb)\vb(\xb))w(\yb-\xb)d\yb\\
    &\quad+\phib(\xb)\cdot\int_{\R^d\backslash B_\epsilon(\xb)}\chi_{\nub}(\yb-\xb)(\varphi(\yb)-\varphi(\xb))\frac{\yb-\xb}{|\yb-\xb|}w(\yb-\xb)d\yb\\
    &\quad+\int_{\R^d\backslash B_\epsilon(\xb)}\frac{\yb-\xb}{|\yb-\xb|}\cdot(\chi_{\nub}(\xb-\yb)\phib(\yb)-\chi_{\nub}(\yb-\xb)\phib(\xb))(\varphi(\yb)-\varphi(\xb))w(\yb-\xb)d\yb.
\end{align*}
Thus, taking the limit as $\epsilon\to 0$, by definition of the nonlocal gradient operator in \cref{hbgrad} and the equivalent form of the divergence operator in \Cref{rem:div_equiv}, we have
\begin{equation}
    \mathcal{D}^{-\nub}_w (\varphi\phib)(\xb)=\varphi(\xb)\mathcal{D}^{-\nub}_w \phib(\xb)+\mathcal{G}^{\nub}_w \varphi(\xb)\cdot\phib(\xb)+S(\varphi,\phib)(\xb),
\end{equation}
where we used
\begin{align*}
    &\lim_{\epsilon\to 0}\int_{\R^d\backslash B_\epsilon(\xb)}\frac{\yb-\xb}{|\yb-\xb|}\cdot(\chi_{\nub}(\xb-\yb)\phib(\yb)-\chi_{\nub}(\yb-\xb)\phib(\xb))(\varphi(\yb)-\varphi(\xb))w(\yb-\xb)d\yb\\
    &=\int_{\R^d} \frac{\yb-\xb}{|\yb-\xb|}\cdot(\chi_{\nub}(\xb-\yb)\phib(\yb)-\chi_{\nub}(\yb-\xb)\phib(\xb))(\varphi(\yb)-\varphi(\xb))w(\yb-\xb)d\yb.
\end{align*}
Indeed, similar to \eqref{eq:ptwsS}, for every $\xb\in\R^d$,
\[\begin{split}
    &\left|\chi_{[\epsilon,\infty)}(|\yb-\xb|)\frac{\yb-\xb}{|\yb-\xb|}\cdot(\chi_{\nub}(\xb-\yb)\phib(\yb)-\chi_{\nub}(\yb-\xb)\phib(\xb))(\varphi(\yb)-\varphi(\xb))w(\yb-\xb)\right|\\
    \le & 4 \|\phib\|_{L^\infty(\R^d;\R^d)}\|\varphi \|_{W^{1,\infty}(\R^d)}\min(1,|\yb-\xb|)w(\yb-\xb),
\end{split}\]
so the above limit is justified by the dominated convergence theorem.
\end{proof}

The generalized produce rule presented above is helpful in showing the following result, which says that $\cS_w^{\nub}(\R^d)$ is closed under the multiplication with $C^\infty_c(\R^d)$.
\begin{lem}[Closedness under multiplication with bump functions]\label{lem:mult}
For $u\in \cS_w^{\nub}(\R^d)$
 and $\varphi\in C^\infty_c(\R^d)$, $\varphi u\in \cS_w^{\nub}(\R^d)$ and 
\begin{equation}\label{eq:multiest}
    \|\varphi u\|_{\cS_w^{\nub}(\R^d)}\le C \|\varphi\|_{W^{1,\infty}(\R^d)}\|u\|_{\cS_w^{\nub}(\R^d)},\quad \forall \varphi\in C^\infty_c(\R^d),\ u\in \cS_w^{\nub}(\R^d),
\end{equation}
where $C$ depends on $d$, $M_w^1$ and $M_w^2$. 
As a result, for $u\in \cS_w^{\nub}(\Om)$ and $\varphi\in C^\infty_c(\R^d)$ with $\supp\,\varphi\subset\Om$, $\varphi u\in \cS_w^{\nub}(\Om)$ for any domain $\Om\subset\R^d$.
\end{lem}
\begin{proof}
First, notice that  
\[\|\varphi u\|_{L^2(\R^d)}\le \|\varphi\|_{L^\infty(\R^d)}\|u\|_{L^2(\R^d)}.\]
Therefore we only need to show $\mathfrak{G}^{\nub}_w (\varphi u)\in L^2(\R^d;\R^d)$. The rest of proof in fact shows a `product rule' for nonlocal distributional gradient $\mathfrak{G}^{\nub}_w$ using the `product rule' of nonlocal divergence $\mathcal{D}^{-\nub}_w$ derived in \Cref{prop:productrulediv}. Since $u\in \cS_w^{\nub}(\R^d)$, there exists $\wb=\mathfrak{G}^{\nub}_w  u\in L^2(\R^d;\R^d)$ such that
\begin{equation}
    \int_{\R^d} \wb(\xb)\cdot\bm\phi(\xb)d\xb =-\int_{\R^d} u(\xb)\mathcal{D}^{-\nub}_w \bm\phi(\xb)d\xb ,\quad\forall\bm\phi\in C^\infty_c(\R^d;\R^d).
\end{equation}
To show $\mathfrak{G}^{\nub}_w (\varphi u)\in L^2(\R^d;\R^d)$, it suffices to find $\vb\in L^2(\R^d;\R^d)$ such that
\begin{equation}\label{eq:v}
    \int_{\R^d} \vb(\xb)\cdot\bm\phi(\xb)d\xb =-\int_{\R^d} \varphi(\xb)u(\xb)\mathcal{D}^{-\nub}_w \bm\phi(\xb)d\xb ,\quad\forall\bm\phi\in C^\infty_c(\R^d;\R^d).
\end{equation}
By \Cref{prop:productrulediv}, we have 
\[\varphi\mathcal{D}^{-\nub}_w\bm\phi=\mathcal{D}^{-\nub}_w (\varphi\bm\phi)-\mathcal{G}^{\nub}_w\varphi\cdot\bm\phi-S(\varphi,\bm\phi),\]thus, for any $\bm\phi\in C^\infty_c(\R^d;\R^d)$,
\begin{align*}
    &\quad\, -\int_{\R^d} \varphi(\xb)u(\xb)\mathcal{D}^{-\nub}_w \bm\phi(\xb)d\xb \\
    &=-\int_{\R^d} u(\xb)\mathcal{D}^{-\nub}_w (\varphi\bm\phi)(\xb)d\xb +\int_{\R^d} u(\xb)\mathcal{G}^{\nub}_w\varphi(\xb)\cdot\bm\phi(\xb)d\xb +\int_{\R^d} u(\xb)S(\varphi,\bm\phi)(\xb)d\xb \\
    &=\int_{\R^d} \wb(\xb)\varphi(\xb)\cdot\bm\phi(\xb)d\xb +\int_{\R^d} u(\xb)\mathcal{G}^{\nub}_w\varphi(\xb)\cdot\bm\phi(\xb)d\xb +\int_{\R^d} H(u,\varphi)(\xb)\cdot\bm\phi(\xb)d\xb ,
\end{align*}
where we use $\varphi\bm\phi\in C^\infty_c(\R^d;\R^d)$ and $H(u,\varphi):\R^d\to \R^d$ is a vector-valued function whose expression will be given in a moment. Comparing this with \eqref{eq:v}, we notice that the vector-valued function $\vb$ should be 
\begin{equation}\label{eq:defv}
    \vb(\xb)=\wb(\xb)\varphi(\xb)+u(\xb)\mathcal{G}^{\nub}_w\varphi(\xb)+H(u,\varphi)(\xb).
\end{equation}
It remains to show this function is in $L^2(\R^d;\R^d)$.
By the definition of $S(\varphi,\bm\phi)$,
\begin{align*}
    &\quad\; S(\varphi,\bm\phi)(\xb)\\
    &=\int_{\R^d}\frac{\yb-\xb}{|\yb-\xb|}\cdot(\chi_{\nub}(\xb-\yb )\bm\phi(\yb)-\chi_{\nub}(\yb-\xb)\bm\phi(\xb))(\varphi(\yb)-\varphi(\xb))w(\yb-\xb)d\yb \\
    &=\int_{\R^d}\frac{\yb-\xb}{|\yb-\xb|}\chi_{\nub}(\xb-\yb )\cdot\bm\phi(\yb)(\varphi(\yb)-\varphi(\xb))w(\yb-\xb)d\yb \\
    &\quad-\int_{\R^d}\frac{\yb-\xb}{|\yb-\xb|}\chi_{\nub}(\yb-\xb)(\varphi(\yb)-\varphi(\xb))w(\yb-\xb)d\yb \cdot\bm\phi(\xb)\\
    &=:H_1(\varphi,\bm\phi)(\xb)+H_2(\varphi)(\xb)\cdot\bm\phi(\xb).
\end{align*}
Note that both $H_1(\varphi,\bm\phi)$ and $H_2(\varphi)$ are well-defined maps on $\R^d$ due to the similar reason for the the pointwise estimate \eqref{eq:ptwsS} for $S(\varphi,\bm\phi)$. For example,
\begin{equation}\label{eq:LinfH2}
\begin{split}
    &|H_2(\varphi)(\xb)|=\left|-\int_{\R^d}\frac{\yb-\xb}{|\yb-\xb|}\chi_{\nub}(\yb-\xb)(\varphi(\yb)-\varphi(\xb))w(\yb-\xb)d\yb \right|\\
    \le & \int_{\R^d} (2\|\varphi\|_{L^\infty(\R^d)}\chi_{|\yb-\xb|>1}(\yb)+\|\nabla\varphi\|_{L^\infty(\R^d)} |\yb-\xb|\chi_{B_1(\xb)}(\yb))w(\yb-\xb)d\yb\\ \leq & 2\| \varphi\|_{W^{1,\infty}(\R^d)} (M_w^1+M_w^2),\ \xb\in\R^d.
 \end{split}   
\end{equation}
Observe that by Fubini's theorem 
\begin{align*}
    &\quad\ \int_{\R^d} u(\xb)H_1(\varphi,\bm\phi)(\xb)d\xb \\
    &=\int_{\R^d} u(\xb)\int_{\R^d}\frac{\yb-\xb}{|\yb-\xb|}\chi_{\nub}(\xb-\yb )\cdot\bm\phi(\yb)(\varphi(\yb)-\varphi(\xb))w(\yb-\xb)d\yb d\xb \\
    &=\int_{\R^d} \bm\phi(\yb)\cdot F(u,\varphi)(\yb)d\yb ,
\end{align*}
where $F(u,\varphi):\R^d\to \R^d$ is given by
\[F(u,\varphi)(\yb):=\int_{\R^d} u(\xb)\frac{\yb-\xb}{|\yb-\xb|}\chi_{\nub}(\xb-\yb )(\varphi(\yb)-\varphi(\xb))w(\yb-\xb)d\xb .\]
Using Holder's inequality, one can show $F(u,\varphi)\in L^2(\R^d;\R^d)$. Indeed, 
\begin{align*}
    &\quad\left(\int_{\R^d} |F(u,\varphi)(\yb)|^2 d\yb \right)^{\frac{1}{2}} \\
    &\leq 2\| \varphi\|_{W^{1,\infty}(\R^d)}\left( \int_{\R^d} \left| \int_{\R^d} |u(\xb)| \min(1,|\yb-\xb|)w(\yb-\xb)d\xb \right|^2 d\yb  \right)^{\frac{1}{2}}  \\
    &\le 2\| \varphi\|_{W^{1,\infty}(\R^d)}  (M_w^1+M_w^2)^{\frac{1}{2}}\left(\int_{\R^d} \int_{\R^d} |u(\xb)|^2\min(1,|\yb-\xb|)w(\yb-\xb)d\xb d\yb \right)^\frac{1}{2}\\
    &\leq 2\| \varphi\|_{W^{1,\infty}(\R^d)} (M_w^1+M_w^2)\|u\|_{L^2(\R^d)}<\infty.
\end{align*}
Combining the above discussions, we obtain 
\begin{equation}
    H(u,\varphi)(\xb)=F(u,\varphi)(\xb)+u(\xb)H_2(\varphi)(\xb),
\end{equation}with
\[
\begin{split}
\|H(u,\varphi)\|_{L^2(\R^d;\R^d)}&\le \|F(u,\varphi)\|_{L^2(\R^d;\R^d)}+\|H_2(\varphi)\|_{L^\infty(\R^d;\R^d)}\|u\|_{L^2(\R^d)}\\
&\le 4\| \varphi\|_{W^{1,\infty}(\R^d)} (M_w^1+M_w^2)\|u\|_{L^2(\R^d)}.
\end{split}
\]
Therefore, by \eqref{eq:defv}, $\vb=\wb\varphi+u\mathcal{G}^{\nub}_w\varphi+H(u,\varphi)\in L^2(\R^d;\R^d)$ and
\begin{align*}
   &\quad\; \|\vb\|_{L^2(\R^d;\R^d)}\\
    &\le \|\wb\|_{L^2(\R^d;\R^d)}\|\varphi\|_{L^\infty(\R^d)}+\|u\|_{L^2(\R^d)}\|\mathcal{G}^{\nub}_w\varphi\|_{L^\infty(\R^d;\R^d)}+\|H(u,\varphi)\|_{L^2(\R^d;\R^d)}\\
    &\le \|\wb\|_{L^2(\R^d;\R^d)}\|\varphi\|_{L^\infty(\R^d)}+\|u\|_{L^2(\R^d)} C(M_w^1+M_w^2)\| \varphi\|_{W^{1,\infty}(\R^d)}\\
    &\quad+4\| \varphi\|_{W^{1,\infty}(\R^d)} (M_w^1+M_w^2)\|u\|_{L^2(\R^d)}\\
    &\le C\|\varphi\|_{W^{1,\infty}(\R^d)}\|u\|_{\cS_w^{\nub}(\R^d)},
\end{align*}
where we have used \Cref{lem:opforsmoothfct}.
This combined with the $L^2$ estimate on $\varphi u$ leads to \cref{eq:multiest}. 
\end{proof}

Next, we present two results regarding the translation and mollification of functions in $\cS_w^{\nub}(\R^d)$, which are standard techniques useful for proving density of smooth functions. 
For $f:\R^d\to \R^k$ and a given vector $\ab\in\R^d$, denote the translation operator $\tau_{\ab} f (\xb):=f(\xb + \ab)$.  
In addition, we let $\eta_\epsilon$ be the standard mollifiers for $\epsilon>0$, i.e. $\eta_\epsilon(\xb)=\frac{1}{\epsilon^d}\eta(\frac{x}{\epsilon})$ where $\eta\in C^\infty_c(\R^d)$ and $\int_{\R^d}\eta(\xb)d\xb =1$.
The statement of the following two lemmas are new but the proofs follow the standard arguments of similar results in the classical Sobolev spaces. We therefore leave their proofs in the appendix.  
\begin{lem}[Continuity of translation]\label{lem:contitrans}
For $u\in \cS_w^{\nub}(\R^d)$ and $\ab\in\R^d$, $\tau_{\ab} u \in \cS_w^{\nub}(\R^d)$ and
 \[
 \lim_{|\ab|\to 0}\|\tau_{\ab} u-u\|_{\cS_w^{\nub}(\R^d)}=0.
 \]
\end{lem}

\begin{lem}[Mollification in $\cS_w^{\nub}(\R^d)$]\label{lem:mollif}
For $u\in \cS_w^{\nub}(\R^d)$ and $\epsilon>0$, $\eta_\epsilon*u\in \cS_w^{\nub}(\R^d)$ and
\begin{equation}
    \lim_{\epsilon\to 0} \|\eta_\epsilon*u-u\|_{\cS_w^{\nub}(\R^d)}=0.
\end{equation}
\end{lem}

With the necessary components presented in Lemmas \ref{lem:mult}, \ref{lem:contitrans} and \ref{lem:mollif}, the proof of \Cref{thm:density} uses the standard mollification and partition of unity techniques (see \cite{AdFo03,Evans2015}  for instance). 
Here we present its proof for completeness. Similar arguments can be found in \cite{fiscella2015density} or \cite{Gounoue20} (Theorem 3.76(i)). 

\begin{proof}[Proof of \Cref{thm:density}]
We prove the result for $\Om$ being a bounded domain with continuous boundary. The other two cases are more straightforward. 
Since $\partial\Om$ is compact, there exist $\xb_i\in\partial\Om$, $i = 1,\cdots, N$ and $r > 0$ such that
\[\partial\Om\subset\bigcup_{i=1}^N B_{r/2}(\xb_i),\]and
\[\Om \cap B_r (\xb_i ) = \{\xb = (\xb', x_d ) \in B_r (\xb_i )\,|\,x_d > \zeta_i (\xb')\}\]
\[\Om^c \cap B_r (\xb_i ) = \{\xb = (\xb' , x_d ) \in B_r (\xb_i )\,|\,x_d \le \zeta_i (\xb')\}\]for some continuous functions $\zeta_i : \R^{d-1} \to \R$ up to relabelling the coordinates. Let 
$\Om^\circ_{r/2}: = \{x \in\Om : \mathrm{dist} (x, \partial\Om) > r/2\}$. Then,

\[\Om\subset  \bigcup_{i=1}^N B_{r}(\xb_i)\cup \Om^\circ_{r/2}.\]

Let $\{ \varphi_i \}_{i=0}^N$ be a smooth partition of unity subordinate to the above constructed sets. That is we have $\varphi_i\ge 0$, $\sum_{i=0}^N \varphi_i=1$ and $\varphi_0 \in C^\infty_c(\Om^\circ_{r/2})$ and $\varphi_i \in C^\infty_c(B_r (\xb_i ))$, $1\le i\le N$. Let $u\in \cS_w^{\nub}(\Om)$ and deﬁne
\[u^i:=\varphi_i u\quad\text{ for all }i\in\{0,\cdots,N\}.\]
By \Cref{lem:mult}, $u^i \in \cS_w^{\nub}(\Om)$. 
For $\mu>0$, we deﬁne
\[u^i_\mu(\xb)=u^i(\xb',x_d-\mu)\quad\text{for }\ i\in\{1,\cdots,N\},\ \xb=(\xb',x_d)\in \R^d.\] Fix $\sigma>0$, by Lemma \ref{lem:contitrans}, there exists $\mu\in (0,\frac{1}{2}\min_{1\le i\le N}\mathrm{dist}(\supp\ \varphi_i,\partial B_r(\xb_i)))$ such that 
\[\|u^i_\mu-u^i\|_{\cS_w^{\nub}(\R^d)}<\frac{\sigma}{2(N+1)},\quad\forall 1\le i\le N.\]
Fix this $\mu$, it follows that $\eta_\epsilon*u^i_\mu \in C^\infty_c(\Om)$ for a positive number $\epsilon$ less than $\min_{1\le i\le N}\mathrm{dist}(\supp\ u^i_\mu,\partial\Om)/2$. Indeed, since $\supp\, u^i_\mu\subset\overline{W^i_\mu}\subset\Om\cap B_r(\xb_i)$, where $W^i_\mu:=\{\zb=(\zb',z_d)\in B_r(\xb_i):z_d-\mu>\zeta_i(\zb')\}$, $1\le i\le N$, 
\[\supp (\eta_\epsilon*u^i_\mu)\subset \overline{B_\epsilon(\bm 0)+\supp\ u^i_\mu}\subset\Om.\]Since $u^i_\mu\in \cS_w^{\nub}(\R^d)$, by Lemma \ref{lem:mollif}, $\eta_\epsilon*u^i_\mu\in \cS_w^{\nub}(\R^d)$ and there exists $\epsilon>0$ such that $\eta_\epsilon * u^0 \in C^\infty_c(\Omega)$,
\[\|\eta_\epsilon*u^i_\mu-u^i_\mu\|_{\cS_w^{\nub}(\R^d)}<\frac{\sigma}{2(N+1)}\]and
\[\|\eta_\epsilon*u^0-u^0\|_{\cS_w^{\nub}(\R^d)}<\frac{\sigma}{2(N+1)}.\]
Let $v_\epsilon:=\eta_\epsilon*u^0+\sum_{i=1}^N \eta_\epsilon*u^i_\mu$, we have $v_\epsilon\in C^\infty_c(\Omega)$ and $\|v_\epsilon-u\|_{\cS_w^{\nub}(\R^d)}<\sigma$. Therefore, the lemma is proved.
\end{proof}

\subsection{Nonlocal vector inequalities and identities}
\label{subsec:vector_identities_general}
In this subsection, we derive a number of results for Sobolev-type functions using \Cref{thm:density}. The first two results are the analogs of $H^1\subset H(\mathrm{div})$ and $H^1\subset H(\mathrm{curl})$ in the local setting. We assume $\Om$ is a bounded domain with a continuous boundary, an epigraph, or $\R^d$ so that the density result holds.

\begin{prop}\label{prop:divisL2}
For $\ub\in \cS_w^{\nub}(\Omega;\R^d)$, $\mathfrak{D}_{w}^{\pm\nub} \ub\in L^2(\R^d)$ and \[\|\mathfrak{D}_{w}^{\pm\nub} \ub\|_{L^2(\R^d)}\le \|\mathfrak{G}^{\nub}_w \ub\|_{L^2(\R^d;\R^{d\times d})}.\]Thus, $\mathfrak{D}_{w}^{\pm\nub}:\cS_w^{\nub}(\Omega;\R^d)\to L^2(\R^d)$ is a bounded linear operator with operator norm no more than $1$. In addition, there exists $\{\ub^{(n)}\}_{n=1}^\infty\subset C^\infty_c(\Om;\R^d)$ such that $\ub^{(n)}\to \ub$ in $\cS_w^{\nub}(\Omega;\R^d)$ and $\mathcal{D}_w^{\pm\nub} \ub^{(n)}\to \mathfrak{D}_{w}^{\pm\nub} \ub$ in $L^2(\R^d)$ as $n\to\infty$.
\end{prop}
\begin{proof}
We first show the inequality for smooth functions with compact support, that is, assuming $\ub\in C^\infty_c(\Om;\R^d)$, 
\begin{equation}\label{divL2normlegradL2norm_smooth}
    \|\mathcal{D}_w^{\pm\nub} \ub\|_{L^2(\R^d)}\le \|\mathcal{G}_w^{\nub} \ub\|_{L^2(\R^d;\R^{d\times d})}.
\end{equation}
We only prove \cref{divL2normlegradL2norm_smooth} for $\mathcal{D}^{-\nub}_w$ since the result also holds for $\mathcal{D}^{\nub}_w$ by noticing that $\|\mathcal{G}_w^{\nub} \ub\|_{L^2(\R^d;\R^{d\times d})}=\|\mathcal{G}_w^{-\nub} \ub\|_{L^2(\R^d;\R^{d\times d})}$ using Plancherel's theorem. For $1\le i\le d$ and a scalar function $p:\R^d\to\R$, we introduce the notation $\mathcal{G}_i$ given by
\[\mathcal{G}^{\nub}_w p:=(\mathcal{G}_1 p,\cdots,\mathcal{G}_d p)^T.\] We use Fourier transform to show that 
\beq\label{eq:divL2normeqgradip}
\left\|\mathcal{D}_w^{-\nub} \ub\right\|_{L^2(\R^d)}^2=\sum_{i,j=1}^d \left(\mathcal{G}_i u_j,\mathcal{G}_j u_i\right)_{L^2(\R^d)},
\eeq
which implies the desired result by Young's inequality for products. 
By \Cref{lem:FT} and $\bm\lambda_w^{-\nub}(\bm\xi)=-\overline{\bm\lambda_w^{\nub}(\bm\xi)}$, for $\xib\in\R^d$, \[\mathcal{F}(\mathcal{D}_w^{-\nub} \ub)(\xib)=-\sum_{i=1}^d \overline{\lambda_i(\xib)}\hat{u}_i(\xib)\quad\text{and}\quad\mathcal{F}(\mathcal{G}_i u_j)(\xib)=\lambda_i(\xib)\hat{u}_j(\xib),\ 1\le i,j\le d,\]
where the Fourier symbol \[\bm\lambda_w^{\nub}(\bm\xi):=(\lambda_1(\xib),\cdots,\lambda_d(\xib))^T.\]
Therefore, by Plancherel's theorem, we conclude the proof of \cref{divL2normlegradL2norm_smooth} as
\begin{align*}
    &\sum_{i,j=1}^d \left(\mathcal{G}_i u_j,\mathcal{G}_j u_i\right)_{L^2(\R^d)}
    =\sum_{i,j=1}^d \left(\mathcal{F}(\mathcal{G}_i u_j),\mathcal{F}(\mathcal{G}_j u_i)\right)_{L^2(\R^d;\C^d)}\\  
    =&\sum_{i,j=1}^d \int_{\R^d}\lambda_i(\xib)\hat{u}_j(\xib)\overline{\lambda_j(\xib)}\overline{\hat{u}_i(\xib)}d\xib = \int_{\R^d}\sum_{j=1}^d \overline{\lambda_j(\xib)}\hat{u}_j(\xib)\sum_{i=1}^d \lambda_i(\xib)\overline{\hat{u}_i(\xib)}d\xib\\
    = &\left\|\mathcal{F}(\mathcal{D}_w^{-\nub} \ub)\right\|_{L^2(\R^d;\C^d)}^2=\left\|\mathcal{D}_w^{-\nub} \ub\right\|_{L^2(\R^d)}^2.
\end{align*}
By \Cref{thm:density}, there exists $\{\ub^{(n)}\}_{n=1}^\infty\subset C^\infty_c(\Om;\R^d)$ such that $\ub^{(n)}\to \ub$ in $\cS_w^{\nub}(\Omega;\R^d)$, and in particular, $\mathcal{G}^{\nub}_w \ub^{(n)}\to \mathfrak{G}^{\nub}_w \ub$ in $L^2(\R^d;\R^{d\times d})$ as $n\to\infty$. Applying \cref{divL2normlegradL2norm_smooth} fo $\ub^{(n)}$, $\{\mathcal{D}^{\pm\nub}_w \ub^{(n)}\}_{n\in\N}$ is a Cauchy sequence in $L^2(\R^d)$, and thus has a limit in $L^2(\R^d)$ by completeness. Since $\ub^{(n)}\to \ub$ in $L^2(\R^d;\R^d)$, by applying \Cref{prop:hbibp}(2) to $\ub^{(n)}$ one derives that the limit is $\mathfrak{D}_{w}^{\pm\nub} \ub\in L^2(\R^d)$ with the desired estimate. 
\end{proof}

\begin{prop}\label{prop:curlisL2}
Let $d=3$. 
For $\ub\in \cS_w^{\nub}(\Omega;\R^3)$, $\mathfrak{C}_w^{\pm\nub} \ub\in L^2(\R^3;\R^3)$ and \[\|\mathfrak{C}_w^{\pm\nub} \ub\|_{L^2(\R^3;\R^3)}\le \|\mathfrak{G}^{\nub}_w \ub\|_{L^2(\R^3;\R^{3\times 3})}.\]Thus, $\mathfrak{C}_{w}^{\pm\nub}:\cS_w^{\nub}(\Omega;\R^3)\to L^2(\R^3;\R^3)$ is a bounded linear operator with operator norm no more than $1$.  In addition, there exists $\{\ub^{(n)}\}_{n=1}^\infty\subset C^\infty_c(\Om;\R^3)$ such that $\ub^{(n)}\to \ub$ in $\cS_w^{\nub}(\Omega;\R^3)$ and $\mathcal{C}_w^{\pm\nub} \ub^{(n)}\to \mathfrak{C}_w^{\pm\nub} \ub$ in $L^2(\R^3;\R^3)$ as $n\to\infty$.
\end{prop}
\begin{proof}
    Using the density result \Cref{thm:density} and integration by parts formula \Cref{prop:hbibp}(3) as in the last paragraph of the proof of \Cref{prop:divisL2}, it suffices to show 
    \begin{equation}\label{curlL2normlegradL2norm_smooth}
        \|\mathcal{C}_w^{\pm\nub} \ub\|_{ L^2(\R^3;\R^3)}\le \|\mathcal{G}_w^{\nub} \ub\|_{L^2(\R^3;\R^{3\times 3})}
    \end{equation}
    for $\ub\in C^\infty_c(\Om;\R^3)$. To show \cref{curlL2normlegradL2norm_smooth}, by Fourier transform, it suffices to show that \[\|\mathcal{F}(\mathcal{C}^{\nub}_w \ub)\|_{L^2(\R^3;\R^3)}\le \|\mathcal{F}(\mathcal{G}^{\nub}_w \ub)\|_{L^2(\R^3;\R^{3\times 3})}.\]
Since $|\ab\times\bb|\le |\ab||\bb|$ for $\ab,\bb\in\C^3$, by \Cref{lem:FT} \cref{curlL2normlegradL2norm_smooth} holds as
\begin{align*}
    \|\mathcal{F}(\mathcal{C}^{\nub}_w \ub)\|_{L^2(\R^3;\R^3)}^2&=\int_{\R^3}|\bm\lambda_w^{\nub}(\xib)\times\hat{\ub}(\xib)|^2 d\xib\le \int_{\R^3}|\bm\lambda_w^{\nub}(\xib)|^2|\hat{\ub}(\xib)|^2 d\xib\\
    &=\sum_{j=1}^3\int_{\R^3}|\bm\lambda_w^{\nub}(\xib)\hat{u}_j(\xib)|^2 d\xib=\|\mathcal{F}(\mathcal{G}^{\nub}_w \ub)\|_{L^2(\R^3;\R^{3\times 3})}^2.
\end{align*}
\end{proof}

Recall that $\mathfrak{D}_w^{-\nub} \vb\in (C^\infty_c(\R^d))'\subset (C^\infty_c(\Om))'$ for $\vb\in L^2(\R^d;\R^d)$. By \Cref{thm:density} one can define $\mathfrak{D}_w^{-\nub} \vb\in (\cS_w^{\nub}(\Omega))^\ast$ for $\vb\in L^2(\R^d;\R^d)$. More precisely, given $u\in \cS_w^{\nub}(\Omega)$, define \beq\label{def:divL2todualsp}
\langle \mathfrak{D}_w^{-\nub} \vb,u\rangle:=-\int_{\R^d}\vb(\xb)\cdot \mathfrak{G}^{\nub}_wu(\xb)d\xb.
\eeq
Then \[|\langle \mathfrak{D}_w^{-\nub} \vb,u\rangle|\le \|\vb\|_{L^2(\R^d;\R^d)}\|\mathfrak{G}^{\nub}_w u\|_{L^2(\R^d;\R^{d})}\le \|\vb\|_{L^2(\R^d;\R^d)}\|u\|_{\cS_w^{\nub}(\Omega)}\]implies that $\mathfrak{D}_w^{-\nub} \vb\in (\cS_w^{\nub}(\Omega))^\ast$. Therefore, $\mathfrak{D}_w^{-\nub}:L^2(\R^d;\R^d)\to (\cS_w^{\nub}(\Omega))^\ast$ is a bounded linear operator with operator norm no more than $1$. The same property holds for $\mathfrak{G}^{\nub}_w $ and $\mathfrak{C}^{\nub}_w $ using \Cref{prop:divisL2} and \Cref{prop:curlisL2}, respectively. We summarize this observation in the following proposition.
\begin{prop}
\label{prop:operators_from_L2}    $\mathfrak{D}_w^{\mp\nub}:L^2(\R^d;\R^d)\to (\cS_w^{\pm\nub}(\Omega))^\ast$ defined by \cref{def:divL2todualsp} is a bounded linear operator with operator norm no more than $1$.
    We can similarly define $\mathfrak{G}_w^{\pm\nub}:L^2(\R^d)\to (\cS^{\nub}_w(\Om;\R^d))^\ast$ and $\mathfrak{C}_w^{\pm\nub}:L^2(\R^3;\R^3)\to (\cS^{\nub}_w(\Om;\R^3))^\ast$  and they are bounded linear operators with operator norms no more than $1$. 
\end{prop}

Based on the above results, the nonlocal vector identities in \Cref{sec:nvismoothfct} hold for functions in the space $\cS_w^{\nub}(\Om;\R^N)$. 
The vector identities shown below are crucial for establishing the nonlocal Helmholtz decomposition in \Cref{subsec:helmhotlz}.

\begin{lem}\label{lem:vanishidentity_generalfct}
Let $d=3$. Then for $u\in \cS_w^{\nub}(\Om)$ and $\vb\in \cS_w^{\nub}(\Om;\R^3)$, in the sense of distributions,
\begin{equation}\label{curlgradeq0_generalfct}
    \mathfrak{C}^{\nub}_w  \mathfrak{G}^{\nub}_w u=0,
\end{equation} and
\begin{equation}\label{divcurleq0_generalfct}
    \mathfrak{D}^{\nub}_w  \mathfrak{C}^{\nub}_w  \vb=0.
\end{equation}
\end{lem}
\begin{proof}
    Since $u\in \cS_w^{\nub}(\Om)$, $\mathfrak{G}^{\nub}_w u\in L^2(\R^3;\R^3)$.  By definition, $\mathfrak{C}^{\nub}_w  \mathfrak{G}^{\nub}_w u\in (C_c^\infty(\R^3;\R^3))'$ and for $\bm\phi\in C_c^\infty(\R^3;\R^3)$,
    \begin{equation}
    \langle \mathfrak{C}_w^{\nub} \mathfrak{G}^{\nub}_w u,\bm\phi\rangle=\int_{\R^3} \mathfrak{G}^{\nub}_w u(\xb)\cdot\mathcal{C}_w^{-\nub}\bm\phi(\xb)d\xb=\lim_{n\to\infty} \int_{\R^3} \mathcal{G}^{\nub}_w u^{(n)}(\xb)\cdot\mathcal{C}_w^{-\nub}\bm\phi(\xb)d\xb,
\end{equation} 
where the sequence $\left\{u^{(n)}\right\}_{n=1}^\infty\subset C^\infty_c(\R^3)$ is chosen according to \Cref{thm:density} such that $\mathcal{G}^{\nub}_w u^{(n)}\to \mathfrak{G}^{\nub}_w u$ in $L^2(\R^3;\R^3)$. By integration by parts formula, \[\int_{\R^3} \mathcal{G}^{\nub}_w u^{(n)}(\xb)\cdot\mathcal{C}_w^{-\nub}\bm\phi(\xb)d\xb=\int_{\R^3}u^{(n)}(\xb)\mathcal{D}^{-\nub}_w \mathcal{C}_w^{-\nub}\bm\phi(\xb)d\xb.\] 
Since $\mathcal{D}^{-\nub}_w \mathcal{C}_w^{-\nub}\bm\phi=0$ by \Cref{lem:vanishidentity}, we have $\mathfrak{C}^{\nub}_w  \mathfrak{G}^{\nub}_w u=0 \in (C_c^\infty(\R^3;\R^3))' $. 
Then by \Cref{prop:operators_from_L2},
$\mathfrak{C}^{\nub}_w  \mathfrak{G}^{\nub}_w u=0 \in (\cS^{\nub}_w(\Om;\R^3))^\ast$. 
 \Cref{divcurleq0_generalfct} can be shown similarly.
\end{proof}

\begin{rem}
\label{rem:2dvanishidentity}
For $d=2$, one can show by similar arguments as those in \Cref{lem:vanishidentity} and \Cref{lem:vanishidentity_generalfct} that
\[
 \mathfrak{D}^{\nub}_w  \begin{pmatrix}
    0 & 1\\
    -1 & 0
    \end{pmatrix}  \mathfrak{G}^{\nub}_w u= 0
\]
for $u\in \cS_w^{\nub}(\Om)$,
which can be seen as a 2d version of \cref{curlgradeq0_generalfct} or \cref{divcurleq0_generalfct}.
\end{rem}

\begin{lem}\label{lem:veccalid1_generalfct}
For $\ub\in \cS_w^{\nub}(\Om;\R^2)$,
\begin{equation}\label{veccalid1_generalfct}
   \mathfrak{D}_w^{-\nub}\mathfrak{G}_w^{\nub} \ub=\mathfrak{G}_{w}^{\nub} \mathfrak{D}_w^{-\nub} \ub-
    \begin{pmatrix}
    0 & 1\\
    -1 & 0
    \end{pmatrix}
    \mathfrak{G}_{w}^{-\nub} \mathfrak{D}_w^{\nub} \left[\begin{pmatrix}
    0 & 1\\
    -1 & 0
    \end{pmatrix}\ub\right].
\end{equation}
 \end{lem}
 \begin{proof}
Notice that the left hand side and the right hand side of \eqref{veccalid1_generalfct} are understood as 
 elements in $(\cS_w^{\nub}(\Om;\R^2))^\ast$ by \Cref{prop:operators_from_L2}, 
 i.e., for any $\vb \in\cS_w^{\nub}(\Om;\R^2)$, we need
 \beq
 \label{eq:veccalid1_wk}
( \mathfrak{G}_{w}^{\nub}\ub ,  \mathfrak{G}_{w}^{\nub}\vb )_{L^2(\R^2;\R^{2\times 2})}= (\mathfrak{D}_w^{-\nub}\ub, \mathfrak{D}_w^{-\nub}\vb)_{L^2(\R^2)} + ( \mathfrak{D}_w^{\nub} J \ub, \mathfrak{D}_w^{\nub} J \vb)_{L^2(\R^2;\R^2)},
 \eeq
 where $J=\begin{pmatrix}
    0 & 1\\
    -1 & 0
    \end{pmatrix}$.
 First notice that \cref{eq:veccalid1_wk} holds for all functions $\vb \in C_c^\infty(\Om;\R^2)$ by \Cref{lem:veccalid1} and the definitions of $\mathfrak{G}_{w}^{\nub}$ and $\mathfrak{D}_w^{\nub}$. Now for any $\vb \in\cS_w^{\nub}(\Om;\R^2)$, using \Cref{thm:density} and \Cref{prop:divisL2}, there exists a sequence $\vb^{(n)}\in C_c^\infty(\Om;\R^2)$ such that $\mathcal{G}_w^{\nub}\vb^{(n)}\to \mathfrak{G}_{w}^{\nub}\vb$ in $L^2(\R^2,\R^{2\times 2})$ and $\mathcal{D}_w^{\pm\nub}\vb^{(n)}\to \mathfrak{D}_w^{\pm\nub}\vb$ in $L^2(\R^2)$. Then \cref{eq:veccalid1_wk} holds for $\vb\in \cS_w^{\nub}(\Om;\R^2)$ by taking limits.     
 \end{proof}
 
 \begin{lem}\label{lem:veccalid2_generalfct}
For $\ub\in \cS_w^{\nub}(\Om;\R^3)$,
\begin{equation}\label{veccalid2_generalfct}
   \mathfrak{D}_w^{-\nub}\mathfrak{G}_w^{\nub} \ub=\mathfrak{G}_{w}^{\nub} \mathfrak{D}_w^{-\nub} \ub-\mathfrak{C}_{w}^{-\nub}\mathfrak{C}_{w}^{\nub} \ub.
\end{equation}
\end{lem}    
\begin{proof}
Notice that the left hand side and the right hand side of \eqref{veccalid2_generalfct} are understood as 
 elements in $(\cS_w^{\nub}(\Om;\R^3))^\ast$. The proof is similar to the proof of \Cref{lem:veccalid1_generalfct} by using \Cref{lem:veccalid2}.
\end{proof}

\section{Nonlocal Poincar\'e inequality for integrable kernels with compact support}
\label{sec:Poincare_int_kernel}
In this section, we prove the Poincar\'e inequality for integrable kernels with compact support. Throughout this section, we assume that $w \in L^1(\R^d)$ and \begin{equation}
\label{kernel_compact}
   \supp\ w \subset K \text{ for some compact set } K\subset\R^d.
\end{equation}
We also assume for the rest of this paper that $\Om\subset\R^d$ is a bounded domain. 
Our major result in this section is the Poincar\'e inequality stated below.  
\begin{thm}[Poincar\'e inequality for integrable kernels with compact support]\label{thm:wintpoincare}
Let $\Om\subset\R^d$ be a bounded domain. 
Assume that $w\in L^1(\R^d)$ and satisfies \eqref{kernel_compact}, then the Poincar\'e inequality holds for $u\in \cS_w^{\nub}(\Omega)$. That is, there exists a constant $\Pi = 
\Pi(w,\Om,\nub)>0$ such that
\begin{equation}
     \|u\|_{L^2(\Omega)}\le \Pi \|\mathfrak{G}^{\nub}_w u\|_{L^2(\R^d;\R^d)},\quad \forall u\in \cS_w^{\nub}(\Omega),
\end{equation}
where $\mathfrak{G}^{\nub}_w u=\mathcal{G}_w^{\nub} u\in L^2(\R^d;\R^d)$.
\end{thm}
In the following, we establish necessary ingredients for the proof of \Cref{thm:wintpoincare}. 
We first give a list of new notations that will be used in this section.
\begin{itemize}
    \item For the kernel $w$, let $\hat{c}$ be a constant depending only on $w$ defined as \begin{equation}\label{chat}
    \hat{c}:=\int_{\{z_1>0\}}\frac{z_1}{|\zb|}w(\zb)d\zb >0.
\end{equation}
    \item For a fixed unit vector $\nub\in\mathbb{R}^d$, define a vector-valued function ${\bm\beta}^{\nub}:\mathbb{R}^d\to\mathbb{R}^d$ by
    \begin{equation}\label{defbeta}
    {\bm\beta}^{\nub}(\zb):=\chi_{\nub}(\zb)\frac{\zb}{|\zb|}w(\zb).
\end{equation}
    \item Let $\cb_{\nub}$  be a constant vector depending only on $w$ and $\nub$ defined as
\begin{equation}\label{eq:veccn}
    \cb_{\nub}:=\int_{\mathbb{R}^d}{\bm\beta}^{\nub}(\zb)d\zb =\int_{\mathbb{R}^d}\chi_{\nub}(\zb)\frac{\zb}{|\zb|}w(\zb)d\zb .
\end{equation}
    \item Let $\bm F,\bm G:\mathbb{R}^d\to\mathbb{R}^d$ be vector-valued functions. Define their convolution as the following scalar-valued function
\[(\bm F*\bm G)(\xb):=\int_{\mathbb{R}^d} \bm F(\xb-\yb )\cdot \bm G(\yb)d\yb .\]
\end{itemize}

There are a few properties related to the above defined quantities. We list these properties here without proof since they are not hard to see. 

\begin{itemize}
    \item  For any $d\times d$ orthogonal matrix $R$, 
\begin{equation}\label{eq:cnrot}
    \cb_{R\nub}=R\cb_{\nub}.
\end{equation}
Consequently, 
\begin{equation}\label{eq:cnequalhatcn}
    \cb_{\nub}=\hat{c}\nub.
\end{equation}
    \item  From  Young's convolution inequality, for $1\le p\le \infty$ we have 
    \[\|\bm F*\bm G\|_{L^p(\mathbb{R}^d)}\le \|\bm F\|_{L^1(\mathbb{R}^d;\mathbb{R}^d)}\|\bm G\|_{L^p(\mathbb{R}^d;\mathbb{R}^d)}.\] 
\end{itemize}

With the integrability assumption of $w$, we notice that $\mathcal{G}_w^{\nub}$ is well-defined on $L^p(\R^d)$, $\mathcal{D}^{\nub}_w $ is well-defined on 
$L^p(\R^d; \R^d)$, and the limiting process in the definition of $\mathcal{G}_w^{\nub}$ and $\mathcal{D}^{\nub}_w$ can be dropped. In addition, each of $\mathcal{G}_w^{\nub}$ and $\mathcal{D}^{\nub}_w$ can be rewritten as a convolution operator plus a multiplication operator using the notations $\bm\beta^{\nub}$ and $\cb_{\nub}$ above. 
These lead to a stronger version of integration by parts formula and an equivalent characterization of $\cS_w^{\nub}(\Omega)$.

\begin{prop}\label{prop:properties}
The following statements are true. \\
 (1) For $1\le p\le \infty$, $\mathcal{G}_w^{\nub}: L^p(\R^d) \to L^p(\R^d;\R^d)$, $\mathcal{D}^{\nub}_w : L^p(\R^d;\R^d) \to L^p(\R^d) $, and $\mathcal{C}^{\nub}_w : L^p(\R^3;\R^3) \to L^p(\R^3;\R^3)$ are bounded operators with estimates 
\begin{align*}
&\| \mathcal{G}_w^{\nub} u \|_{L^p(\R^d;\R^d)}\le 2\|w\|_{L^1(\mathbb{R}^d)}\|u\|_{L^p(\mathbb{R}^d)}, \\
&\|\mathcal{D}^{\nub}_w  \vb\|_{L^p(\mathbb{R}^d)}\le C\|w\|_{L^1(\mathbb{R}^d)}\|\vb\|_{L^p(\mathbb{R}^d;\mathbb{R}^d)}, \\
&\|\mathcal{C}^{\nub}_w  \vb\|_{L^p(\mathbb{R}^3;\R^3)}\le C\|w\|_{L^1(\mathbb{R}^3)}\|\vb\|_{L^p(\mathbb{R}^3;\mathbb{R}^3)}.
\end{align*}
for some $C =C(d)>0$. Moreover, for $u\in L^p(\R^d)$ and $\vb\in L^p(\R^d;\R^d)$, 
\begin{equation}\label{Gconvplusmult}
    \mathcal{G}_w^{\nub} u(\xb)=\int_{\R^d} \bm\beta^{\nub}(\yb-\xb)u(\yb)d\yb -\cb_{\nub}u(\xb),\quad\text{a.e.}\ \xb\in\R^d,
\end{equation} 
\begin{equation}\label{Dconvplusmult}
    \mathcal{D}^{-\nub}_w  \vb(\xb)=-\bm\beta^{\nub}*\vb(\xb)+\cb_{\nub}\cdot\vb(\xb),\quad\text{a.e.}\ \xb\in\R^d.
\end{equation}
(2) Suppose $u\in L^p(\mathbb{R}^d)$ and $\vb\in L^{p'}(\mathbb{R}^d;\mathbb{R}^d)$, where $p'=\frac{p}{p-1}$ and $1\le p\le \infty$ ($p'=\infty$ for $p=1$ and $p'=1$ for $p=\infty$). Then 
\[
    \int_{\mathbb{R}^d} \mathcal{G}_w^{\nub} u(\xb)\cdot\vb(\xb)d\xb =-\int_{\mathbb{R}^d}u(\xb)\mathcal{D}^{-\nub}_w  \vb(\xb)d\xb .
\]
Similarly, for $\ub\in L^p(\R^3;\R^3)$ and $\vb\in L^{p'}(\R^3;\R^3)$, 
\[
    \int_{\R^3} \mathcal{C}^{\nub}_w \ub(\xb)\cdot\vb(\xb)d\xb=\int_{\R^3} \ub(\xb)\cdot\mathcal{C}^{-\nub}_w \vb(\xb)d\xb.
\]
\end{prop} 
\begin{proof}
Notice that since $w\in L^1(\R^d)$, the integrand in \cref{hbgrad} is Lebesgue integrable on $\R^d$ for $u\in L^p(\R^d)$, and therefore the limiting process can be dropped. The characterizations \eqref{Gconvplusmult} and \eqref{Dconvplusmult} follow directly from \Cref{def:operators_PV}, \eqref{defbeta} and \eqref{eq:veccn}. For instance, \eqref{Dconvplusmult} holds as \[
\begin{split}
\mathcal{D}^{-\nub}_w  \vb (\xb) &= \int_{\R^d} \frac{\yb-\xb}{|\yb-\xb|}\cdot \chi_{\nub}(\xb-\yb)(\vb(\yb) - \vb(\xb))w(\yb-\xb) d\yb \\
& = -\int_{\R^d}\chi_{\nub}(\xb-\yb)\frac{\xb-\yb}{|\xb-\yb|} w(\yb-\xb) \cdot\vb(\yb)d\yb + \int_{\R^d} \chi_{\nub}(\zb) \frac{\zb\cdot \vb(\xb)}{|\zb|} w(\zb) d\zb  \\
& = -{\bm\beta}^{\nub}\ast \vb(\xb) + \cb_{\nub}\cdot \vb(\xb) \quad\text{in}\ L^p(\R^d),
\end{split}
\]where the convolution term is well-defined thanks to the Young's convolution inequality and the fact that $w\in L^1(\R^d)$.
Suppose $1<p<\infty$. For $p=1$ and $p=\infty$ we can show the estimate similarly. Using Holder's inequality, we obtain 
\begin{align*}
   & \quad\, \int_{\mathbb{R}^d}|\mathcal{G}_w^{\nub} u(\xb)|^pd\xb =\int_{\mathbb{R}^d}\left|\int_{\mathbb{R}^d} (u(\yb)-u(\xb))\frac{\yb-\xb}{|\yb-\xb|}w(\yb-\xb)d\yb \right|^pd\xb \\
    &\le \int_{\mathbb{R}^d}\left(\int_{\mathbb{R}^d} |u(\yb)-u(\xb)|w(\yb-\xb)d\yb \right)^pd\xb \\
    &\le \int_{\mathbb{R}^d} \left(\int_{\mathbb{R}^d} w(\yb-\xb)d\yb \right)^{\frac{p}{p'}}\left(\int_{\mathbb{R}^d} w(\yb-\xb)|u(\yb)-u(\xb)|^{p}d\yb \right) d\xb \\
    &\le 2^{p-1}\|w\|_{L^1(\R^d)}^{\frac{p}{p'}} \int_{\mathbb{R}^d}\int_{\mathbb{R}^d} w(\yb-\xb)(|u|^p(\yb)+|u|^p(\xb))d\yb d\xb \\
    &=2^{p-1}\|w\|_{L^1(\R^d)}^{\frac{p}{p'}} \left(\int_{\mathbb{R}^d} |u|^p(\yb)d\yb \int_{\mathbb{R}^d} w(\yb-\xb)d\xb +\int_{\mathbb{R}^d} |u|^p(\xb)d\xb \int_{\mathbb{R}^d} w(\yb-\xb)d\yb \right)\\
    &\le 2^p\|w\|_{L^1(\R^d)}^p\|u\|_{L^p(\mathbb{R}^d)}^p,
\end{align*}
where $p'= p/(p-1)$. 
This shows (1). 
The estimates for $\mathcal{D}^{\nub}_w  \vb$ and 
$\mathcal{C}^{\nub}_w  \vb$ can be shown similarly. 

The integration by parts formulas in (2) can be shown by a change of integration order via Fubini's theorem, for example,
\begin{align*}
    \int_{\mathbb{R}^d} \mathcal{G}_w^{\nub} u(\xb)\cdot\vb(\xb)d\xb &=\int_{\mathbb{R}^d} \int_{\mathbb{R}^d} \chi_{\nub}(\yb-\xb)(u(\yb)-u(\xb))\frac{\yb-\xb}{|\yb-\xb|}w(\yb-\xb)\cdot\vb(\xb)d\yb d\xb \\
    &=\int_{\mathbb{R}^d} \int_{\mathbb{R}^d} -\chi_{\nub}(\xb-\yb )u(\xb)\frac{\yb-\xb}{|\yb-\xb|}w(\yb-\xb)\cdot\vb(\yb)d\yb d\xb \\&\quad\quad-\int_{\mathbb{R}^d} \int_{\mathbb{R}^d} \chi_{\nub}(\yb-\xb)u(\xb)\frac{\yb-\xb}{|\yb-\xb|}w(\yb-\xb)\cdot\vb(\xb)d\yb d\xb \\
    &=-\int_{\mathbb{R}^d}u(\xb)\mathcal{D}^{-\nub}_w  \vb(\xb)d\xb .
\end{align*}
Here Fubini's theorem is justified since $|u|(\yb)|\vb|(\xb)w(\yb-\xb)\in L^1(\mathbb{R}^d\times\mathbb{R}^d)$ for $u\in L^p(\R^d)$ and $\vb\in L^{p'}(\R^d;\R^d)$. Indeed, 
\begin{align*}
    \int_{\mathbb{R}^d} \int_{\mathbb{R}^d} |u|(\yb)|\vb|(\xb)w(\yb-\xb)d\yb d\xb &=\int_{\mathbb{R}^d} |u|(\yb)(w*|\vb|)(\yb)d\yb \\
    &\le \|u\|_{L^p(\mathbb{R}^d)}\|w*|\vb|\|_{L^{p'}(\mathbb{R}^d)}\\
    &\le C\|u\|_{L^p(\mathbb{R}^d)}\|w\|_{L^1(\mathbb{R}^d)}\|\vb\|_{L^{p'}(\mathbb{R}^d;\mathbb{R}^d)}<\infty,
\end{align*}
where we used Young's convolution inequality and $C$ is a constant only depending on the dimension $d$.
The second integration by parts formula in (2) can be shown similarly. 
\end{proof}
An immediate result from \Cref{prop:properties} is the following equivalent characterization
of $\cS_w^{\nub}(\Omega)$.
\begin{coro}[An equivalent characterization of $\cS_w^{\nub}(\Omega)$]\label{hbequichar}
With the integrability assumption of $w$, $\mathcal{G}^{\nub}_w u=\mathfrak{G}^{\nub}_w u$ for $u\in L^2(\R^d)$ with $u=0$ a.e. in $\Om^c$, and the function space $\cS_w^{\nub}(\Omega)$ defined by \cref{deffctsp_2} satisfies \[\cS_w^{\nub}(\Omega)=\{u\in L^2(\mathbb{R}^d): u=0 \text{ a.e. in }\Omega^c\}.\]
\end{coro}

We next show a crucial result for proving the Poincar\'e inequality. It claims that the operator $\mathcal{G}_w^{\nub}$ restricted to $\cS_w^{\nub}(\Omega)$ is injective.
\begin{prop}\label{prop:hbgradkerisnull}
Assume that $w\in L^1(\R^d)$ and satisfies \eqref{kernel_compact}.
If $u\in \cS_w^{\nub}(\Omega)$ satisfies $\mathcal{G}_w^{\nub} u(\xb)=0$ for a.e. $\bm x\in \mathbb{R}^d$, then $u\equiv 0$.
\end{prop}

\begin{proof}
Note that the nonlocal integration by parts formula in \Cref{prop:properties} also holds for complex-valued functions. That is, for $u\in L^p(\R^d;\C)$ and $\vb\in L^{p'}(\R^d;\C^d)$ with $p$ and $p'$ given by \Cref{prop:properties},
\begin{equation}\label{eq:cplxhbibp}
    (\mathcal{G}^{\nub}_w u,\vb)_{L^2(\R^d;\C^d)}=-(u,\mathcal{D}^{-\nub}_w \vb)_{L^2(\R^d;\C)},
\end{equation}where the $L^2$-norm is given by \[(\bm F,\bm G)_{L^2(\R^d;\C^n)}=\int_{\R^d}\bm F(\xb)^T\overline{\bm G(\xb)}d\xb,\quad\forall \bm F, \bm G\in L^2(\R^d;\C^n),\ n=1,d.\] Thus, for any $\bm\varphi$ in the Schwartz space $\mathscr{S}(\mathbb{R}^d;\mathbb{C}^d)\subset L^2(\mathbb{R}^d;\mathbb{C}^d)$, we have
\begin{equation}\label{eq:hatunull}
\begin{aligned}
    0&=(\mathcal{G}_w^{\nub} u,\bm\varphi)_{L^2(\R^d;\C^d)}=-(u, \mathcal{D}^{-\nub}_w  \bm\varphi)_{L^2(\R^d;\C)}\\
    &=-( \mathcal{F}u,\mathcal{F}(\mathcal{D}^{-\nub}_w  \bm\varphi))_{L^2(\R^d;\C)}=-(\hat{u},({\bm\lambda}_w^{-\nub})^T\hat{\bm\varphi})_{L^2(\R^d;\C)},
\end{aligned}
\end{equation}
where $\hat{u}=\mathcal{F}u\in L^2(\mathbb{R}^d;\mathbb{C})$ since $u\in L^2(\mathbb{R}^d)$ and $\mathcal{D}^{-\nub}_w  \bm\varphi\in L^2(\mathbb{R}^d;\C)$ by \cref{eq:cplxhbibp}. 

Since $L^2(\mathbb{R}^d;\mathbb{C})\subset \mathscr{S}'(\mathbb{R}^d;\mathbb{C})$, we view $\hat{u}=\mathcal{F}u\in \mathscr{S}'(\mathbb{R}^d;\mathbb{C})$ as a tempered distribution. Now we prove the following claim:
\begin{equation}
    \langle \hat{u},\phi\rangle=0,\quad\forall \phi\in C_c^\infty(\mathbb{R}^d\backslash \{\bm 0\};\mathbb{C}).
\end{equation}
Let $\bm\varphi:\mathbb{R}^d\to\mathbb{C}^d$ be defined as \[\bm\varphi:=\mathcal{F}^{-1}\left( \frac{\overline{{\bm\lambda}_w^{-\nub}}}{|{\bm\lambda}_w^{-\nub}|^2}\phi\right).\]
Since $|{\bm\lambda}_w^{-\nub}(\xib)|>0$ for $\xib\neq \bm 0$ by \Cref{prop:lambdatransandpositive} and $\phi(\xib)=0$ in a neighborhood of $\xib=\bm 0$, $\overline{{\bm\lambda}_w^{-\nub}(\xib)}\phi(\xib)/|{\bm\lambda}_w^{-\nub}(\xib)|^2$ is a well-defined vector-valued function on $\mathbb{R}^d$. Moreover, $\overline{{\bm\lambda}_w^{-\nub}(\xib)}\phi(\xib)/|{\bm\lambda}_w^{-\nub}(\xib)|^2\in C_c^\infty(\mathbb{R}^d\backslash \{\bm 0\},\mathbb{C}^d)\subset \mathscr{S}(\mathbb{R}^d;\mathbb{C}^d)$  since ${\bm\lambda}_w^{-\nub}\in C^\infty(\mathbb{R}^d;\mathbb{C}^d)$ by \Cref{prop:fouriersymsmooth} and $\phi\in C_c^\infty(\mathbb{R}^d\backslash \{\bm 0\},\mathbb{C})$. Hence $\bm\varphi\in \mathscr{S}(\mathbb{R}^d;\mathbb{C}^d)$ since $\mathcal{F}$ is an isomorphism on $\mathscr{S}(\mathbb{R}^d;\mathbb{C}^d)$.  Observing that ${\bm\lambda}_w^{-\nub}(\xib)^T\hat{\bm\varphi}(\xib)=\phi(\xib)$, the claim follows from \cref{eq:hatunull}.

Now from the claim, we have $\supp \ \hat{u}\subset \{\bm 0\}$. Then by Corollary 2.4.2 in \cite{grafakos2014classical}, $u$ is a polynomial, i.e., \[u(\xb)=\sum_{|\alpha|\le k} a_\alpha x^\alpha\] for some nonnegative integer $k$ and real numbers $a_\alpha$ for $|\al|\le k$. 
Since $u=0$ in $\Omega^c$, it follows $u\equiv 0$.
\end{proof}

The last ingredient of the proof of the Poincar\'e inequality is the weak lower semicontinuity of the Dirichlet integral $\int_{\mathbb{R}^d}|\mathcal{G}_w^{\nub} u(\xb)|^2d\xb$.
To establish this result, we need Proposition A.3 in \cite{BALL1984225} which is stated as a lemma below.

\begin{lem}[Proposition A.3 in \cite{BALL1984225}]
\label{lem:propa3}
Let $\Omega\subset \mathbb{R}^m$ be bounded open, and let $H:\mathbb{R}^s\to \R\cup\{\pm \infty\}$ be convex, lower semicontinuous and bounded below. Let $\theta_j,\, \theta\in L^1(\Omega; \mathbb{R}^s)$ with $\theta_j\stackrel{\ast}{\rightharpoonup}\theta$ (i.e., $\int_\Omega \theta_j\phi d\xb \to \int_\Omega \theta\phi d\xb $ for all $\phi\in C_c(\Omega)$). Then 
\[\int_\Omega H(\theta(\xb))d\xb \le \liminf_{j\to \infty} \int_\Omega H(\theta_j(\xb))d\xb .\]
\end{lem}

\begin{prop}\label{prop:lscofhbgrad}
Suppose that $\{u_n\}$ converges weakly to $u$ in $\cS_w^{\nub}(\Omega)$. Then
\begin{equation}\label{lscofhbgrad}
    \int_{\mathbb{R}^d}|\mathcal{G}_w^{\nub} u(\xb)|^2d\xb \le \liminf_{n\to\infty} \int_{\mathbb{R}^d}|\mathcal{G}_w^{\nub} u_n(\xb)|^2d\xb .
\end{equation}
\end{prop}

\begin{proof}
Let $H(\xb):=|\xb|^2$. Then $H$ is convex, continuous and bounded below. Let $\theta_n(\xb):=\mathcal{G}_w^{\nub} u_n(\xb)$ and $\theta(\xb):=\mathcal{G}_w^{\nub} u(\xb)$. Then for any open and precompact set $D\subset\R^d$, $\theta_n,\theta\in L^1(D;\mathbb{R}^d)$ because
    \[\int_{D} |\mathcal{G}_w^{\nub} u(\xb)|d\xb \le \left(\int_{D} |\mathcal{G}_w^{\nub} u(\xb)|^2d\xb \right)^{\frac{1}{2}}(\mu(D))^{\frac{1}{2}}<\infty.\]
    For any $\phi\in C_c(D)$, $1\le k\le d$, define a linear functional $F^k_\phi:\cS_w^{\nub}(\Omega)\to\mathbb{R}$ by
    \[F^k_\phi(u):=\int_{D} [\mathcal{G}_w^{\nub} u(\xb)]_k\phi d\xb ,\]
    where $[\mathcal{G}_w^{\nub} u(\xb)]_k$ denotes the $k$-th component of $\mathcal{G}_w^{\nub}$. 
    Then $F^k_\phi$ is a bounded linear functional since
    \begin{align*}
        \left|F^k_\phi(u)\right|
        \le \left(\int_{D} |[\mathcal{G}_w^{\nub} u(\xb)]_k|^2d\xb \right)^{\frac{1}{2}}\cdot \left(\int_{D} |\phi(\xb)|^2d\xb \right)^{\frac{1}{2}}\leq C \|u\|_{\cS_w^{\nub}(\Omega)}. 
    \end{align*}
    Now since $u_n\rightharpoonup u$ in $\cS_w^{\nub}(\Omega)$, we have $F^k_\phi(u_n)\to F^k_\phi(u)$ as $n\to \infty$. Therefore $\theta_j\stackrel{\ast}{\rightharpoonup}\theta \in L^1(D;\mathbb{R}^d)$ and this yields
    \[
    \int_{D}|\mathcal{G}_w^{\nub} u(\xb)|^2d\xb \le \liminf_{n\to\infty} \int_{D}|\mathcal{G}_w^{\nub} u_n(\xb)|^2d\xb \le \liminf_{n\to\infty} \int_{\R^d}|\mathcal{G}_w^{\nub} u_n(\xb)|^2d\xb.
    \]
  by \Cref{lem:propa3}.
  Since $D\subset\R^d$ is arbitrary, 
     \cref{lscofhbgrad} is true.
\end{proof}

Finally, we are ready  to prove \Cref{thm:wintpoincare}. 
\begin{proof}[Proof of \Cref{thm:wintpoincare}]
We argue by contradiction. Suppose there exists $\{u_n\}\subset \cS_w^{\nub}(\Omega)$ with $\|u_n\|_{L^2(\Omega)}=1$ such that $\|\mathcal{G}_w^{\nub} u_n\|_{L^2(\mathbb{R}^d;\R^d)}\to 0$. Then $\|u_n\|_{\cS_w^{\nub}(\Omega)}$ is bounded. Since $\cS_w^{\nub}(\Omega)$ is a Hilbert space by \Cref{thm:fctspcmplt}, there exists a subsequence of $\{u_n\}$, still denoted by $\{u_n\}$ for convenience, that convergences weakly to some $u\in \cS_w^{\nub}(\Omega)$.

In the first step, we show $u=0$, i.e., $u_n\rightharpoonup 0$ in $\cS_w^{\nub}(\Omega)$. By
the weakly lower semi-continuous result in \Cref{prop:lscofhbgrad}, we have
\[\int_{\mathbb{R}^d}|\mathcal{G}_w^{\nub} u(\xb)|^2d\xb \le \liminf_{n\to\infty} \int_{\mathbb{R}^d}|\mathcal{G}_w^{\nub} u_n(\xb)|^2d\xb .\]
Now that $\|\mathcal{G}_w^{\nub} u_n\|_{L^2(\mathbb{R}^d;\R^d)}\to 0$, \[\int_{\mathbb{R}^d}|\mathcal{G}_w^{\nub} u(\xb)|^2d\xb =0,\]and thus  $\mathcal{G}_w^{\nub} u(\xb)=0$ for a.e. $\xb\in \mathbb{R}^d$. By \Cref{prop:hbgradkerisnull}, $u\equiv 0$ and the first step is done.

 In the second step, we show $u_n$ converges to $0$ strongly in $L^2$, which contradicts the assumption $\|u_n\|_{L^2(\Om)}=1$. 
 
Using the integration by parts formula and the characterizations of $\mathcal{D}^{-\nub}_w$ and $\mathcal{G}_w^{\nub}$ consecutively in \Cref{prop:properties}, it follows that (recall that $u_n=0$ in $\Om^c$),
\begin{equation}\label{eq:strongconvest}
     \begin{aligned}
     &\|\mathcal{G}_w^{\nub} u_n\|_{L^2(\R^d;\R^d)}^2 \\
     =& - \int_{\R^d} u_n(\xb) \mathcal{D}^{-\nub}_w  \circ \mathcal{G}_w^{\nub} u_n (\xb) d\xb\\
     =&\int_{\R^d} u_n(\xb) \left({\bm\beta}^{\nub}\ast \mathcal{G}_w^{\nub} u_n (\xb)-\cb_{\nub}\cdot\mathcal{G}_w^{\nub} u_n (\xb) \right)d\xb\\
     =&\int_{\R^d} u_n(\xb)\left({\bm\beta}^{\nub}\ast \mathcal{G}_w^{\nub} u_n\right) (\xb)d\xb\\
     &\qquad-\int_{\R^d} u_n(\xb)\cb_{\nub}\cdot \left(\int_{\R^d}{\bm\beta}^{\nub}(\yb-\xb)u_n(\yb)d\yb-\cb_{\nub}u_n(\xb)\right)d\xb\\
     =&(u_n,{\bm\beta}^{\nub} * \mathcal{G}_w^{\nub} u_n)_{\Om} +|\cb_{\nub}|^2\|u_n\|^2_{L^2(\Om)}-(u_n,K u_n)_{\Om}
 \end{aligned}
\end{equation}
where $(\cdot, \cdot)_\Om$ denote the $L^2(\Om)$ inner product and  $K:L^2(\Omega)\to L^2(\Omega)$ is defined as \[Ku(\xb):=\int_{\Omega} \cb_{\nub}\cdot{\bm\beta}^{\nub}(\yb-\xb)u(\yb)d\yb .\]
    Note that $K$ is well-defined as $|Ku|\le |\cb_{\nub}|w*|u|\in L^2(\mathbb{R}^d)$. 
    Now notice that by Young's convolution inequality
    \[(u_n,{\bm\beta}^{\nub} * \mathcal{G}_w^{\nub} u_n)_{\Om}\le \|u_n\|_{L^2(\Om)}\|{\bm\beta}^{\nub} * \mathcal{G}_w^{\nub} u_n\|_{L^2(\Om)}\le \sqrt{d}\|w\|_{L^1(\mathbb{R}^d)}\|\mathcal{G}_w^{\nub} u_n\|_{L^2(\mathbb{R}^d;\mathbb{R}^d)}\to 0.\]
    as $n\to\infty$. In addition, $|\cb_{\nub}|^2=|\hat{c}|^2>0$ by \cref{chat,eq:cnequalhatcn}. Therefore, if $(u_n,K u_n)_{\Om}\to 0$, then we reach a contradiction since \cref{eq:strongconvest} implies $\| u_n\|_{L^2(\Om)}\to 0$.
    In the following, we proceed to show $(u_n,K u_n)_{\Om}\to 0$ as $n\to\infty$. 
    
    Notice that by definition
    \[
    \begin{split}
    Ku(\xb) &= \int_\Om \cb_{\nub}\cdot\chi_{\nub}(\yb-\xb)\frac{\yb-\xb}{|\yb-\xb|}w(\yb-\xb) u(\yb)d\yb \\
    &=  \int_\Om \hat{c} \cdot\chi_{\nub}(\yb-\xb)\frac{\nub\cdot(\yb-\xb)}{|\yb-\xb|}w(\yb-\xb)u(\yb)d\yb \\
    & =: \int_\Om k(\xb-\yb) u(\yb) d\yb,
    \end{split}
    \]
    where we have used \cref{eq:cnequalhatcn} and $k(\zb):=\hat{c} \cdot\chi_{\nub}(-\zb)\frac{\nub\cdot(-\zb)}{|\zb|}w(\zb)\geq 0$. Notice that $k\in L^1(\R^d)$ as $w\in L^1(\R^d)$, it follows from Corollary 4.28 of \cite{Brez11} that $K:L^2(\Omega)\to L^2(\Omega)$ is compact. From the first step we have $u_n\rightharpoonup 0$ in $\cS_w^{\nub}(\Omega)$ and thus $u_n\rightharpoonup 0$ in $L^2(\mathbb{R}^d)$. Thus $K u_n\to 0$ in $L^2(\Omega)$ as $K:L^2(\Omega)\to L^2(\Omega)$ is compact. Therefore,  
 \[
 |(u_n,K u_n)_{\Om}|\le \|u_n\|_{L^2(\Om)}\|K u_n\|_{L^2(\Omega)}\to 0,\ n\to\infty.
 \]
 Hence the proof is completed.
\end{proof}

\section{Nonlocal Poincar\'e inequality for general kernel functions}
\label{sec:Poincare_nonint_kernel}
Our main goal in this section is to prove the Poincar\'e inequality for general kernel functions beyond the integrable and compactly supported ones used in \Cref{sec:Poincare_int_kernel}.
Throughout this section, we assume that the kernel function satisfies \cref{eq:kernelassumption} and the assumptions given as follows.

\begin{assu}\label{asmp:int}
Assume that  $w$ satisfies \cref{eq:kernelassumption},  and either one of the following conditions holds true:
\begin{enumerate}
    \item $\int_{\R^d}|\xb|w(\xb)d\xb<\infty$;
    \item there exists $R>0$ such that $w(\xb)=\frac{c_0}{|\xb|^{d+\alpha}}$ for some $c_0>0$ and $\alpha\in (0,1]$ when $|\xb|>R$. 
\end{enumerate}
We use $\overline{w}$ to denote the radial representation of $w$, i.e., $\overline{w}:[0,\infty)\to [0,\infty)$ satisfies $\overline{w}(|\xb|)=w(\xb)$ for $\xb\in\R^d$.
\end{assu}

\begin{rem}
Notice that \Cref{asmp:int} covers many cases of kernel functions seen in the literature. For example, compactly supported kernels, 
the fractional kernel $w(\xb) =C |\xb|^{-d-\alpha}$ used to study the Riesz fractional derivatives in \cite{shieh2015new,shieh2018new},
as well as the 
tempered fractional kernel $w(\xb) =C e^{-\lambda |\xb|}|\xb|^{-d-\alpha}$ in \cite{sabzikar2015tempered}.
\end{rem}

Under \Cref{asmp:int}, we can show the following result. 

\begin{thm}[Poincar\'e inequality for general kernel functions]
\label{thm:wnonintpoincare}
Let $\Om$ be a bounded domain with a continuous boundary. Under Assumption \ref{asmp:int}, the Poincar\'e inequality holds, i.e., there exists a constant $\Pi= \Pi(w,\Om,\nub)>0$ such that 
\begin{equation}\label{poincare}
     \|u\|_{L^2(\Om)}\le \Pi \|\mathfrak{G}^{\nub}_wu\|_{L^2(\R^d;\R^d)},\quad\forall u\in \cS_w^{\nub}(\Om).
\end{equation}
\end{thm}

For general kernel functions, we do not have a direct analogue of \cref{eq:strongconvest} since the single integral defining $\mathcal{D}^{\nub}_w $ cannot be separated into two parts. Motivated by the fact that singular kernels usually correspond to stronger norms than integrable kernels, e.g., the Riesz fractional gradients lead to Bessel potential spaces \cite{shieh2015new}, it is a natural idea to choose an integrable and compactly supported kernel by which $w$ is bounded below, i.e., a kernel $\phi$ satisfying \cref{eq:kernelassumption} and 
\begin{equation}\label{truncknl}
    0\le \phi(\xb)\le w(\xb),\; \supp\ \phi\subset \overline{B_1(\bm{0})} \;\text{ and }\; 0<\int_{\R^d} \phi(\xb)d\xb<\infty, 
\end{equation}
and utilize the Poincar\'e inequality for integrable kernels with compact support. This further requires a comparison of the norms $\|\mathfrak{G}^{\nub}_wu\|_{L^2(\R^d;\R^d)}$ and $\|\mathfrak{G}_\phi^{\nub}u\|_{L^2(\R^d;\R^d)}$ which is not a trivial task. Here, we resort to the Fourier analysis. 
Let be ${\bm\lambda}_w^{\nub}$ and ${\bm\lambda}_\phi^{\nub}$ be the Fourier symbols are defined by \cref{deffouriersymboldiv}. 
Notice that if there exists a constant $C=C(w,d)>0$ independent of $\nub$ such that 
\begin{equation}\label{Fsymbest}
    |{\bm\lambda}_w^{\nub}(\xib)|\ge C|{\bm\lambda}_\phi^{\nub}(\xib)|,\quad\forall\xib\in\R^d, 
\end{equation}
then we have for any $u\in C^\infty_c(\Om)$,
\begin{align*}
    \|\mathcal{G}_w^{\nub}u\|^2_{L^2(\R^d;\R^d)}&=\int_{\R^d}\left|{\bm\lambda}_w^{\nub}(\xib)\right|^2|\hat{u}(\xib)|^2d\xib\\
    &\ge C^2\int_{\R^d}\left|{\bm\lambda}_\phi^{\nub}(\xib)\right|^2|\hat{u}(\xib)|^2d\xib=C^2\|\mathcal{G}_\phi^{\nub}u\|^2_{L^2(\R^d;\R^d)},
\end{align*}
and the Poincar\'e inequality for general kernels can be further inferred.

\begin{lem}\label{lem:cmproffouriersym}
Assume that $w$ satisfies \Cref{asmp:int}. Then there exists a kernel function $\phi$ satisfying \cref{eq:kernelassumption} and \cref{truncknl} such that \cref{Fsymbest} holds, where ${\bm\lambda}_w^{\nub}$ and ${\bm\lambda}_\phi^{\nub}$ are the Fourier symbols defined by \cref{deffouriersymboldiv}.
\end{lem}
\begin{proof}
    We divide the proof of \cref{Fsymbest} into two steps. Along the proof, the desired kernel function $\phi$ will be constructed, and more precisely, is defined by \cref{defphi}. Without loss of generality in the following steps we assume $d\ge 2$. The case for $d=1$ is similar.

    \textbf{Step I.} We prove that there exists $N_1=N_1(w,d)\in (0,1)$ and $C_1=C_1(w,d)>0$ such that 
    \begin{equation}\label{Fsymbestsmallxi}
        |{\bm\lambda}_w^{\nub}(\xib)|\ge C_1|{\bm\lambda}_\phi^{\nub}(\xib)|,\quad\forall |\xib|<N_1. 
    \end{equation}
    Since $\phi$ is integrable and satisfies \cref{truncknl}, there exists $C>0$ depending on $w$ (as $\phi$ itself depends on $w$) such that for $|\bm\xi|\le 1$, 
    \begin{equation}\label{L1kerFsymbbddbyxi}
        |{\bm\lambda}_\phi^{\nub}(\xib)|\le 2\sqrt{2}\pi|\bm\xi|\int_{|\zb|\le 1}\phi(\zb)d\zb=C|\bm\xi|.
    \end{equation}
    Observe that $\Im ({\bm\lambda}_w^{\nub})(\xib)$ is a scalar multiple of $\frac{\xib}{|\xib|}$ as a result of \Cref{lem:lambdaexplicitform}, i.e., $\Im ({\bm\lambda}_w^{\nub})(\xib)=\Lambda_w(|\xib|)\frac{\xib}{|\xib|}$ where $\Lambda_w$ is given by \cref{defLambda}.
Using polar coordinates, we obtain
\begin{align*}
    \Lambda_w (|\xib|)&=\frac{1}{2}\int_0^{\pi}\int_0^\infty \cos(\theta) \overline{w}(r)\sin(2\pi|\xib|r\cos(\theta))r^{d-1}\sin^{d-2}(\theta)drd\theta\\
    &\quad\quad\int_0^{\pi}\sin^{d-3}(\varphi_1)d\varphi_1\int_0^{\pi}\sin^{d-4}(\varphi_2)d\varphi_2\cdots\int_0^{\pi}\sin(\varphi_{d-3})d\varphi_{d-3}\int_0^{2\pi}d\varphi_{d-2}\\
    &=\frac{1}{2}\omega_{d-2}\int_0^{\pi}\int_0^\infty \cos(\theta) \overline{w}(r)\sin(2\pi|\xib|r\cos(\theta))r^{d-1}\sin^{d-2}(\theta)drd\theta\\
    &=\omega_{d-2}\int_0^{\frac{\pi}{2}}\int_0^\infty \cos(\theta)\sin^{d-2}(\theta)r^{d-1} \overline{w}(r)\sin(2\pi|\xib|r\cos(\theta))drd\theta,
\end{align*}
where $\omega_{d-1}=\frac{2\pi^{d/2}}{\Gamma(d/2)}$ is the surface area of $(d-1)$-sphere $\S^{d-1}$.
Now we claim that for $w$ in \Cref{asmp:int}, \[\Lambda_w (|\xib|)\ge C(w,d)|\xib|,\quad\forall |\xib|<N_1.\]Then \cref{Fsymbestsmallxi} holds by \cref{L1kerFsymbbddbyxi} and $|{\bm\lambda}_w^{\nub}(\xib)|\ge |\Im ({\bm\lambda}_w^{\nub})(\xib)|\ge \Lambda_w (|\xib|)$. We prove the claim by two cases to conclude Step I.

\textbf{Case (i).} Suppose $w$ satisfies \Cref{asmp:int} (1). Then $\int_0^\infty r^d \overline{w}(r)dr<\infty$. Since $g(r,\theta):=\cos^2(\theta)\sin^{d-2}(\theta) r^d\overline{w}(r)\in L^1\left((0,\infty)\times (0,\frac{\pi}{2})\right)$, by dominated convergence theorem,
\begin{align*}
    \frac{\Lambda_w (|\xib|)}{|\xib|}&=2\pi \omega_{d-2}\int_0^{\frac{\pi}{2}}\int_0^\infty \cos^2(\theta)\sin^{d-2}(\theta)r^{d} \overline{w}(r)\frac{\sin(2\pi|\xib|r\cos(\theta))}{2\pi|\xib|r\cos(\theta)}drd\theta\\
    &\to 2\pi \omega_{d-2}\int_0^{\frac{\pi}{2}}\cos^2(\theta)\sin^{d-2}(\theta)d\theta \int_0^\infty r^{d} \overline{w}(r)dr>0,\quad\text{as }|\xib|\to 0.
\end{align*}
Therefore, there exists $N_1=N_1(w,d)\in (0,1)$ such that the claim holds.

\textbf{Case (ii).} Suppose $w$ satisfies \Cref{asmp:int} (2). Assume without loss of generality that $R=1$. Then $\int_0^1 r^d \overline{w}(r)dr<\infty$ and $\overline{w}(r)=\frac{c_0}{r^{d+\alpha}}$ with $\alpha\in (0,1]$ for $r>1$. We estimate 
$\Lambda_w(|\xib|)$ by discussing $r\le 1$ and $r>1$. On the one hand, since $\sin x\ge \frac{1}{2}x$ for $|x|$ sufficiently small, there exists $N_1\in (0,1)$ such that $\int_0^1 r^{d-1}\overline{w}(r)\sin(2\pi |\xib|r\cos(\theta))dr\ge \pi|\xib|\cos(\theta)\int_0^1r^d\overline{w}(r)dr$ for $|\xib|<N_1$. On the other hand, there exist $C=C(\alpha,c_0)$ and $N_1=N_1(\alpha)\in (0,1)$ such that 
\begin{align*}
   &\quad \int_1^\infty r^{d-1}\frac{c_0}{r^{d+\alpha}}\sin(2\pi|\xib|r\cos(\theta))dr
    =c_0|\xib|^\alpha (\cos\theta)^\alpha\int_{|\xib|\cos\theta}^\infty \frac{1}{r^{1+\alpha}}\sin(2\pi r)dr\\
    & \ge C|\xib|^\alpha(\cos\theta)^\alpha\ge C|\xib|\cos(\theta),\quad \forall |\xib|<N_1,\ \theta\in \left(0,\frac{\pi}{2}\right),
\end{align*}
where we used $\int_0^\infty \frac{1}{r^{1+\alpha}}\sin(2\pi r)dr\in (0,\infty)$ and dominated convergence theorem for the second last inequality. Combining both cases for $r\le 1$ and $r>1$ yields the claim.

\textbf{Step II.} We prove that there exist $N_2=N_2(w,d)>1$ and $C_2=C_2(w,d)>0$ such that 
    \begin{equation}\label{Fsymbestlargexi}
        |{\bm\lambda}_w^{\nub}(\xib)|\ge C_2|{\bm\lambda}_\phi^{\nub}(\xib)|,\quad\forall |\xib|>N_2, 
    \end{equation}
    and $C_3=C_3(w,d)>0$ such that
    \begin{equation}\label{Fsymbestmediumxi}
        |{\bm\lambda}_w^{\nub}(\xib)|\ge C_3|{\bm\lambda}_\phi^{\nub}(\xib)|,\quad\forall |\xib|\in [N_1,N_2].
    \end{equation}
Since $\phi$ is integrable, 
    \begin{equation}\label{L1kerFsymbbddbyconst}
        |{\bm\lambda}_\phi^{\nub}(\xib)|\le 2\int_{|\zb|\le 1}\phi(\zb)d\zb=2I_\phi,\quad \xib\in\R^d,
    \end{equation}
where $I_\phi:=\int_{\mathbb{R}^d}\phi(\zb)d\zb\in (0,\infty)$ depends on $w$. Denote $\hat{\xib}:=\frac{\xib}{|\xib|}$. Recall that $\nub=R_{\nub}\eb_1$. Then
\begin{equation}\label{Fsymbrealptest}
    \begin{aligned}
        |{\bm\lambda}_w^{\nub}(\xib)|&\ge|\Re ({\bm\lambda}_w^{\nub})(\xib)|\ge |\nub\cdot \Re ({\bm\lambda}_w^{\nub})(\xib)|\\
        &=\int_{\{\zb\cdot\nub>0\}}\frac{\zb\cdot\nub}{|\zb|}w(\zb)(1-\cos(2\pi \xib\cdot\zb))d\zb\\
    &=\int_{\{z_1>0\}}\frac{z_1}{|\zb|}w(\zb)(1-\cos(2\pi |\xib| (R_{\nub}^T\hat{\xib})\cdot \zb))d\zb.
    \end{aligned}
\end{equation}
For any $f\in L^1(\R^d)$, we define a function $I:\R_+\times\S^{d-1}\to\R$ by \[I(\rho,\etab):=\int_{\R^d}f(z)(1-\cos(2\pi \rho\etab\cdot\zb))d\zb.\]

\textbf{Claim.} For $w$ in \Cref{asmp:int}, there exists $f\in L^1(\R^d)$ depending only on $w$ such that 
\begin{equation}\label{ineqonf}
    0\le f(\zb)\le \chi_{\eb_1}(\zb)\frac{z_1}{|\zb|}w(\zb)
\end{equation}
and \begin{equation}\label{estofI}
    I(|\xib|,\etab)\ge C(d)I_\phi,\quad\forall |\xib|>N_2,\ \etab\in\S^{d-1}.
\end{equation}
Once the claim is proved,  \cref{Fsymbestlargexi} and \cref{Fsymbestmediumxi} follows. Indeed, \cref{Fsymbestlargexi} holds by \cref{L1kerFsymbbddbyconst}, \cref{Fsymbrealptest}, \cref{ineqonf} and \cref{estofI}. Notice that $I$ is a continuous function and $I(|\xib|,\etab)>0$ for any $|\xib|\in [N_1,N_2]$ and $\etab\in\S^{d-1}$, \cref{Fsymbestmediumxi} holds as $\min_{[N_1,N_2]\times\S^{d-1}} I(\rho,\etab)>0$ and $I_\phi>0$. We prove the claim to conclude Step II and thus finish the whole proof.

\textbf{Proof of the claim.} We choose $\phi$ as
\begin{equation}\label{defphi}
    \phi(\xb) = \min(1, w(\xb)\chi_{B_1(\bm{0})}(\xb)).
\end{equation}
Then $\phi$ satisfies \cref{eq:kernelassumption} and \cref{truncknl}. Let $f(\zb):=\chi_{\eb_1}(\zb)\frac{z_1}{|\zb|}\phi(\zb)$, then $f\in L^1(\R^d)$ satisfies \cref{ineqonf} and for $V:=(0,\infty)\times (0,\pi/2)\times (0,\pi)^{d-3}\times (0,2\pi)$,
\begin{align*}
    &\quad\ \int_{\R^d}f(\zb)d\zb=\int_{\{z_1>0\}}\frac{z_1}{|\zb|}\phi(\zb)d\zb\\
    &=\idotsint_{V} \cos(\varphi_1)\bar{\phi}(r)r^{d-1}\sin^{d-2}(\varphi_1)\sin^{d-3}(\varphi_2)\cdots\sin(\varphi_{d-2})drd\varphi_1d\varphi_2\cdots d\varphi_{d-2}d\theta\\
    &=\omega_{d-2}\int_0^{\frac{\pi}{2}}\cos(\varphi_1)\sin^{d-2}(\varphi_1)d\varphi_1\int_0^\infty r^{d-1}\bar{\phi}(r)dr\\
    &= \frac{\omega_{d-2}}{(d-1)\omega_{d-1}}\int_{\R^d} \phi(\zb)d\zb,
\end{align*}
where $\bar\phi$ is the radial representation of $\phi$. Note that the above computation holds for $d\ge 3$ and can be easily done for $d=1$ or $2$.
Then by Riemann-Lebesgue lemma, there exists $N_2=N_2(w,d)>1$ such that $\forall |\xib|>N_2$, \[\int_{\R^d}f(\zb)\cos(2\pi |\xib| \etab\cdot \zb)d\zb<\frac{1}{2}\int_{\R^d}f(\zb)d\zb,\quad\forall\etab\in \S^{d-1}.\]
Then for $|\xib|>N_2$ and $\etab\in\S^{d-1}$, \[I(|\xib|,\etab)= \int_{\R^d}f(\zb)(1-\cos(2\pi |\xib| \etab\cdot \zb))d\zb\ge \frac{1}{2}\int_{\R^d}f(\zb)d\zb\ge C(d)I_\phi,\]where $C(d)=\frac{\omega_{d-2}}{2(d-1)\omega_{d-1}}$. Then the claim is proved.
\end{proof}

Now we are ready to prove the the Poincar\'e inequality for general kernels.

\begin{proof}[Proof of \Cref{thm:wnonintpoincare}]
For any $u\in C_c^\infty(\Om)$, we have
$\mathfrak{G}^{\nub}_wu = \mathcal{G}_w^{\nub}u$. 
By \Cref{lem:cmproffouriersym} and the comment below \cref{Fsymbest}, there exists a kernel function $\phi$ satisfying \cref{eq:kernelassumption} and \cref{truncknl} such that \[ \|\mathfrak{G}^{\nub}_wu\|_{L^2(\R^d;\R^d)}=\|  \mathcal{G}_w^{\nub}u\|_{L^2(\R^d;\R^d)} \geq C \| \mathcal{G}_\phi^{\nub}u\|_{L^2(\R^d;\R^d)}\] for some $C=C(w,d)>0$. Therefore, using \Cref{thm:wintpoincare} for the integrable and compactly supported kernel $\phi$, we obtain 
\[\|u\|_{L^2(\Om)}\le \Pi(\phi,\nub,\Om)\|\mathcal{G}_\phi^{\nub}u\|_{L^2(\R^d;\R^d)}\le C^{-1}\Pi(\phi,\nub,\Om) \|\mathfrak{G}^{\nub}_wu\|_{L^2(\R^d;\R^d)}, \; \forall u\in C_c^\infty(\Om).
\]
Denote $\Pi(w,\nub,\Om):=C^{-1}\Pi(\phi,\nub,\Om)$.
By the density result in \Cref{thm:density}, for every $u\in \cS_w^{\nub}(\Om)$, there exists $\{u_j\}_{j=1}^\infty\subset C_c^\infty(\Om)$ such that $u_j\to u$ in $\cS_w^{\nub}(\Om)$. Hence,
\[\|u_j-u\|_{L^2(\Om)}\to 0,\ \text{ and }\ \|\mathfrak{G}^{\nub}_wu_j-\mathfrak{G}^{\nub}_wu\|_{L^p(\R^d;\R^d)}\to 0,\quad j\to\infty.\]Since \[ \|u_j\|_{L^2(\Om)}\le \Pi(w,\nub,\Om)\|\mathfrak{G}^{\nub}_wu_j\|_{L^2(\R^d;\R^d)},\]letting $j\to\infty$ yields \eqref{poincare}.
\end{proof}

\section{Applications}\label{sec:applications}
In this section, we provide some applications of the nonlocal Poincar\'e inequality. Assume that $w$ is a kernel function satisfying \Cref{asmp:int}, $\nub\in\R^d$ is a fixed unit vector, $\Om\subset\R^d$ is an open bounded domain with a continuous boundary. Note that by the nonlocal Poincar\'e inequality \Cref{thm:wintpoincare} and \Cref{thm:wnonintpoincare}, the full norm $\|u\|_{\cS_w^{\nub}(\Om)}$ is equivalent to the seminorm $\|\mathfrak{G}^{\nub}_w u\|_{L^2(\R^d;\R^d)}$ for $u\in \cS_w^{\nub}(\Om)$. Thus in this section we abuse the notation and use $\|\cdot\|_{\cS_w^{\nub}(\Om)}$ to denote the seminorm.

\subsection{Nonlocal convection-diffusion equation}
In \Cref{sec:operators} we have defined nonlocal gradient and divergence operator for the fixed unit vector $\nub$. It turns out that one can define these notions corresponding to a unit vector field $\nb=\nb(\xb)$ as well. Specifically, for a measurable function $u:\R^d\to\R$ and a measurable vector field $\vb:\R^d\to\R^d$, $\mathcal{G}^{\nb}_w u:\Om\to\R$ and $\mathcal{D}^{\nb}_w \vb:\Om\to\R^d$ are defined by
 \[\mathcal{G}^{\nb}_w u(\xb):=\lim_{\epsilon\to 0}\int_{\R^d\backslash B_\epsilon(\xb)}\chi_{\nb(\xb)}(\yb-\xb)(u(\yb)-u(\xb))\frac{\yb-\xb}{|\yb-\xb|}w(\yb-\xb)d\yb\]
 and
 \[\mathcal{D}^{\nb}_w\vb(\xb):=\lim_{\epsilon\to 0}\int_{\R^d\backslash B_\epsilon(\xb)}\frac{\yb-\xb}{|\yb-\xb|}\cdot(\chi_{\nb(\yb)}(\yb-\xb)\vb(\yb)+\chi_{\nb(\xb)}(\xb-\yb)\vb(\xb))w(\yb-\xb)d\yb,\]
 respectively. Let $\phi$ be an integrable kernel with compact support satisfying \cref{eq:kernelassumption}. Then the integration by parts formula in \Cref{prop:properties} (2) holds for the vector field $\nb$ and kernel $\phi$. The proof is similar and thus omitted.
 
For a diffusivity function $\ep=\ep(\xb)\in L^\infty(\R^d)$ with a positive lower bound $\ep_1>0$, a vector field $\bm{b}\in L^\infty(\R^d;\R^d)$ and a function $f\in (\cS_w^{\nub}(\Omega))^\ast$, we consider the nonlocal convection-diffusion model problem formulated as
\begin{equation}\label{eq:cddeq}
    \left\{
    \begin{aligned}
    - \mathfrak{D}^{-\nub}_w  (\epsilon \mathfrak{G}^{\nub}_w u) + \bm{b}\cdot\mathcal{G}_{\phi}^{\nb} u&=f\quad \text{in }\ \Omega,\\
    u&=0\quad\text{in }\ \R^d\backslash\Omega.
    \end{aligned}\right.
\end{equation}
\Cref{eq:cddeq} is a  nonlocal analogue of the classical convection-diffusion equation, see, e.g., \cite{d2017nonlocal,leng2022petrov,tian2017conservative,tian2015nonlocal} for related discussions. 
The new formulation using  $\mathfrak{D}^{-\nub}_w  (\epsilon \mathfrak{G}^{\nub}_w u)$ for the nonlocal diffusion allows the possibility to explore mixed-type numerical methods for \cref{eq:cddeq} in the future.

\begin{rem}
\label{rem:convdiff_compactkernel}
If the kernel function $w$ has compact support, the boundary condition in \cref{eq:cddeq} only needs to be imposed on a bounded domain outside $\Omega$. For example, assume $\supp\ w \subset B_\del(\bm{0})$ for $\del>0$, then for the first equation to be well-defined on $\Om$, we only need $u=0$ on $\Om_{2\del}\backslash\Om$ where $\Om_{2\del}=\{\xb\in \R^d:\dist(\xb,\Om)<2\del \}$.
\end{rem}
We define the bilinear form $b(\cdot,\cdot):\cS_w^{\nub}(\Omega)\times \cS_w^{\nub}(\Omega)\to\R$ associated with \cref{eq:cddeq} by
\begin{equation}
     b(u,v):=(\epsilon \mathfrak{G}^{\nub}_wu, \mathfrak{G}^{\nub}_wv)_{L^2(\R^d;\R^d)}+(\bm{b}\cdot\mathcal{G}_{\phi}^{\nb} u,v)_{L^2(\R^d)}.
\end{equation}
Then the weak formulation is given as follows.
\begin{equation}\label{eq:wkformcddeq}
    \left\{
    \begin{aligned}
         & \text{Find }u\in \cS_w^{\nub}(\Omega) \text{ such that: } \\
         & b(u,v)=\langle f,v\rangle,\quad \forall v\in \cS_w^{\nub}(\Omega).
    \end{aligned}
    \right.
\end{equation}
  The vector field $\nb$ is given in terms of $\bm{b}$ by the following relation 
  \beq
  \label{eq:vector_n}
  \nb=-\frac{\bm{b}}{|\bm{b}|}.
  \eeq
To establish the well-posedness of the model problem \eqref{eq:cddeq}, we give an additional assumption on the velocity field $\bm{b}$.

\begin{assu}\label{asmp:bbdd}
Assume the velocity field $\bm{b}\in L^\infty(\R^d;\R^d)$ satisifies either one of the following assumptions:
\begin{enumerate}[(i)]
    \item $\mathcal{D}_{\phi}^{-\nb}\ \bm{b} \leq 0$, or 
    \item $|\mathcal{D}_{\phi}^{-\nb}\ \bm{b}| \leq \eta $ where $\eta < 2\ep_1/\Pi^2$. 
\end{enumerate}
\end{assu}
We further present a result on the convection part of the bilinear form $b(u, v)$. Similar result can be found in \cite{leng2022petrov}.
\begin{lem}\label{lem:convest}
Let $v\in \cS_w^{\nub}(\Omega)$ and $\nb$ be defined as \cref{eq:vector_n}. Then
$$(\bm{b}\cdot\mathcal{G}_{\phi}^{\nb} v,v)_{L^2(\Omega)}\ge -\frac{1}{2} (v^2, \mathcal{D}_{\phi}^{-\nb}\bb)_{L^2(\Om)}. $$
\end{lem}
\begin{proof}
By the integration by parts formula for vector field $\nb$, 
\begin{align*}
    &\quad\ (\bm{b}\cdot\mathcal{G}_{\phi}^{\nb} v,v)_{L^2(\Omega)} = (\bm{b}\cdot\mathcal{G}_{\phi}^{\nb} v,v)_{L^2(\R^d)} = -(v, \mathcal{D}_{\phi}^{-\nb}(\bm{b}v))_{L^2(\R^d)}\\
    &=-\int_{\R^d} v(\xb)\int_{\R^d} \frac{\yb-\xb}{|\yb-\xb|}\cdot(\chi_{\nb(\yb)}(\xb-\yb)\bm{b}(\yb)v(\yb)+\\
    &\qquad\qquad\qquad\qquad\qquad\qquad\qquad\qquad \chi_{\nb(\xb)}(\yb-\xb)\bm{b}(\xb)v(\xb))\phi(\yb-\xb)d\yb d\xb\\
    &=\frac{1}{2}\iint_{\R^{2d}} (v(\yb)-v(\xb))\frac{\yb-\xb}{|\yb-\xb|}\cdot(\chi_{\nb(\yb)}(\xb-\yb)\bm{b}(\yb)v(\yb)+\\
    &\qquad\qquad\qquad\qquad\qquad\qquad\qquad\qquad \chi_{\nb(\xb)}(\yb-\xb)\bm{b}(\xb)v(\xb))\phi(\yb-\xb)d\yb d\xb\\
    &=\frac{1}{2}\iint_{\R^{2d}} (v(\yb)-v(\xb))v(\xb)\frac{\yb-\xb}{|\yb-\xb|}\cdot(\chi_{\nb(\yb)}(\xb-\yb)\bm{b}(\yb)+\\
    &\qquad\qquad\qquad\qquad\qquad\qquad\qquad\qquad\chi_{\nb(\xb)}(\yb-\xb)\bm{b}(\xb))\phi(\yb-\xb)d\yb d\xb\\
    &\quad\quad+\frac{1}{2}\int_{\R^d} \int_{\R^d} (v(\yb)-v(\xb))^2\frac{\yb-\xb}{|\yb-\xb|}\cdot\chi_{\nb(\yb)}(\xb-\yb)\bm{b}(\yb)\phi(\yb-\xb)d\yb d\xb\\
    &=-\frac{1}{2}\int_{\R^d}  v(\xb)^2\int_{\R^d}\frac{\yb-\xb}{|\yb-\xb|}\cdot(\chi_{\nb(\yb)}(\xb-\yb)\bm{b}(\yb)+\chi_{\nb(\xb)}(\yb-\xb)\bm{b}(\xb))\phi(\yb-\xb)d\yb d\xb\\
    &\quad\quad+\frac{1}{2}\int_{\R^d} \int_{\R^d} (v(\yb)-v(\xb))^2\frac{\xb-\yb}{|\yb-\xb|}\cdot\chi_{\nb(\xb)}(\yb-\xb)\bm{b}(\xb)\phi(\yb-\xb)d\yb d\xb\\
    &=-\frac{1}{2}\int_\Omega v(\xb)^2\mathcal{D}_{\phi}^{-\nb}\ \bm{b}(\xb) d\xb \\
    &\quad\quad +\frac{1}{2}\int_{\R^d}\int_{\{(\yb-\xb)\cdot\nb(\xb)> 0\}} (v(\yb)-v(\xb))^2\frac{\xb-\yb}{|\yb-\xb|}\cdot\bm{b}(\xb)\phi(\yb-\xb)d\yb d\xb \\
    &\ge -\frac{1}{2}\int_\Omega v(\xb)^2\mathcal{D}_{\phi}^{-\nb}\ \bm{b}(\xb) d\xb.
\end{align*}
where we used \cref{eq:vector_n} in the last inequality.
\end{proof}
Use the above lemma, we can  establish the coercivity of the bilinear form and the well-posedness of \cref{eq:cddeq} further by the Lax-Milgram theorem.

\begin{thm}\label{thm:wellposedwkform}
Assume that \Cref{asmp:bbdd} is satisfied. The nonlocal convection-diffusion problem \eqref{eq:wkformcddeq} is well-posed. More precisely, for any $f\in (\cS_w^{\nub}(\Omega))^\ast$, there exists a unique solution $u\in \cS_w^{\nub}(\Omega)$ such that
\[\|u\|_{\cS_w^{\nub}(\Omega)}\le c\|f\|_{(\cS_w^{\nub}(\Omega))^\ast},\]
where $c=c(\ep, \bb, w,\phi, \nub, \Om)$ is a positive constant.
\end{thm}
\begin{proof}
Notice that $b(\cdot,\cdot)$ is coercive under \Cref{asmp:bbdd}. Indeed, if \Cref{asmp:bbdd} (i) is satisfied, then by \Cref{lem:convest}, $b(v,v)\ge \epsilon_1 a(v,v)=\epsilon_1 \|v\|_{\cS_w^{\nub}(\Omega)}^2$, $\forall v\in \cS_w^{\nub}(\Omega)$.
On the other hand, if \Cref{asmp:bbdd} (ii) is satisfied, then by the nonlocal Poincar\'e inequality,
\[
\frac{1}{2}|(v^2, \mathcal{D}_{\phi}^{-\nb}\bb)_{L^2(\Omega)}| \leq \frac{1}{2} \eta \| v\|_{L^2(\Om)}^2 \leq \frac{1}{2} \eta \Pi^2 \| \mathfrak{G}_{w}^{\nub}  v\|_{L^2(\Om;\R^d)}^2 ,
\]
where $\Pi=\Pi(w,\nub,\Om)$ is Poincar\'e constant. Hence, by \Cref{lem:convest}, \[(\bm{b}\cdot\mathcal{G}_{\phi}^{\nb} v,v)_{L^2(\Omega)}\ge -\frac{1}{2}|(v^2, \mathcal{D}_{\phi}^{-\nb}\bb)_{L^2(\Omega)}|\ge -\frac{1}{2} \eta \Pi^2 \| \mathfrak{G}_{w}^{\nub}  v\|_{L^2(\Om;\R^d)}^2.\]  
Therefore
if \Cref{asmp:bbdd} (ii) is satisfied, we have $b(v,v)\ge (\epsilon_1-\frac{1}{2}\eta \Pi^2) \|v\|_{\cS_w^{\nub}(\Omega)}^2$, $\forall v\in \cS_w^{\nub}(\Omega)$.

The boundedness of $b(\cdot,\cdot)$ follows from the nonlocal Poincar\'e inequality and the estimate \[\|\mathcal{G}^{\nb}_\phi v\|_{L^2(\R^d;\R^d)}\le 2\|\phi\|_{L^1(\R^d)}\|v\|_{L^2(\R^d)},\quad\forall v\in L^2(\R^d),\]which is an analog of \Cref{prop:properties} (1) for $\nb$ when $p=2$. 
Finally, the Lax-Milgram theorem yields the well-posedness of \cref{eq:wkformcddeq}.
\end{proof}

\subsection{Nonlocal correspondence models of isotropic linear elasticity} 
For $\ub\in \cS^{\nub}_w(\Om;\R^N)$, define distributional nonlocal vector Laplacian 
\begin{equation}\label{defveclap}
    \mathcal{L}_w^{\nub}  \ub:=\mathfrak{D}_w^{-\nub}\mathfrak{G}_{w}^{\nub} \ub.
\end{equation}
Then $\mathcal{L}_w^{\nub}  \ub\in (\cS^{\nub}_w(\Om;\R^N))^\ast$ and
$\mathcal{L}_w^{\nub}:\cS^{\nub}_w(\Om;\R^N)\to(\cS^{\nub}_w(\Om;\R^N))^\ast$ is a bounded linear operator with operator norm no more than $1$ by \Cref{prop:operators_from_L2}. For the rest of the paper, we consider $N=d$.

For a displacement field $\ub:\R^d\to\R^d$, we study the elastic potential energy given by 
\begin{equation}\label{elastenergy}
    \mathcal{E}  (\ub)=\frac{1}{2}\lambda \|\mathfrak{D}_w^{-\nub}(\ub)\|_{L^2(\R^d)}^2+\mu \|e^{\nub}_w  (\ub)\|_{L^2(\R^d;\R^{d\times d})}^2,
\end{equation}
where $\lambda$ and $\mu$ are Lam\'e coefficients such that $\mu>0$ and $\lambda+2\mu>0$ and $e^{\nub}_w  (\ub)$ is the nonlocal strain tensor 
\begin{equation}
    e^{\nub}_w  (\ub)=\frac{\mathfrak{G}^{\nub}_w \ub+(\mathfrak{G}^{\nub}_w \ub)^T}{2}.
\end{equation}
We also introduce the nonlocal Navi\'er operator $\mathcal{P}^{\nub}_w  $ acting on $\ub$ as 
\begin{equation}\label{defnavierop}
    \mathcal{P}^{\nub}_w   (\ub):=-\mu \mathcal{L}_w^{\nub}  \ub-(\lambda+\mu)\mathfrak{G}_{w}^{\nub} \mathfrak{D}_w^{-\nub} \ub
\end{equation}
in $\Om$.
The goal of this subsection is to show the well-posedness of the following equation 
\begin{equation}\label{elasteqDiribry}
    \left\{
    \begin{aligned}
    \mathcal{P}^{\nub}_w   (\ub)&=\fb\quad \text{in }\ \Om,\\
    \ub&=\bm{0}\quad \text{in }\ \R^d\backslash\Om.
    \end{aligned}
    \right.
\end{equation}
Similarly as \Cref{rem:convdiff_compactkernel}, when the kernel function $w$ is supported on $B_\del(\bm{0})$, we only need the boundary condition to be imposed on $\Om_{2\delta}\backslash\Om$.
The associated function space to the problem is  $$\cS_w^{\nub}(\Om;\R^d)=\{\ub=(u_1, u_2, \cdots, u_d)^T: u_i\in \cS_w^{\nub}(\Om), i=1, 2, \cdots, d\}.$$
The weak formulation of the problem is given by 
\begin{equation}\label{eq:wkform_elasteq}
    \left\{
    \begin{aligned}
         & \text{Find }\ub\in \cS_w^{\nub}(\Omega;\R^d) \text{ such that: } \\
         & B(\ub,\vb)=\langle \fb,\vb\rangle,\quad \forall \vb\in \cS_w^{\nub}(\Omega;\R^d),
    \end{aligned}
    \right.
\end{equation}
where $\fb\in (\cS_w^{\nub}(\Om;\R^d))^\ast$ and the bilinear form $B(\cdot, \cdot):\cS_w^{\nub}(\Om;\R^d)\times\cS_w^{\nub}(\Om;\R^d)\to\R$ is defined as
\[
B(\ub,\vb):=\mu\sum_{j=1}^d(\mathfrak{G}_{w}^{\nub} u_j,\mathfrak{G}_{w}^{\nub} v_j)_{L^2(\R^d;\R^d)}+(\lambda+\mu)(\mathfrak{D}_w^{-\nub}\ub,\mathfrak{D}_w^{-\nub}\vb)_{L^2(\R^d)}.
\]

To make sense of the weak formulation, one need to show that $\mathfrak{D}_w^{-\nub} \ub\in L^2(\R^d)$ for $\ub\in \cS_w^{\nub}(\Omega;\R^d)$. This is proved in \Cref{prop:divisL2}.

The following lemma verifies that $\mathcal{E} $ is indeed the energy for the problem (\ref{eq:wkform_elasteq}).
\begin{lem}\label{lem:equivelasenergy}
For $\ub\in \cS_w^{\nub}(\Omega;\R^d)$, 
\begin{equation}
    B(\ub,\ub)=2\mathcal{E}  (\ub).
\end{equation}
\end{lem}
\begin{proof}
Note that $B(\ub,\ub)=\mu\|\mathfrak{G}^{\nub}_w \ub\|_{L^2(\R^d;\R^{d\times d})}^2+(\lambda+\mu)\|\mathfrak{D}^{-\nub}_w  \ub\|_{L^2(\R^d)}^2$. By \Cref{prop:divisL2}, there exists $\{\ub^{(n)}\}_{n=1}^\infty\subset C^\infty_c(\Om;\R^d)$ such that $\mathcal{G}^{\nub}_w \ub^{(n)}\to \mathfrak{G}^{\nub}_w \ub$ in $L^2(\R^d;\R^{d\times d})$ and $\mathcal{D}_w^{-\nub} \ub^{(n)}\to \mathfrak{D}_w^{-\nub} \ub$ in $L^2(\R^d)$ as $n\to\infty$. Therefore, it suffices to show $B(\ub^{(n)},\ub^{(n)})=2\mathcal{E}  (\ub^{(n)})$ and let $n$ tend to infinity. Using the notation $\mathcal{G}_i$ and  \cref{eq:divL2normeqgradip} in the proof of \Cref{prop:divisL2}, we obtain
\begin{align*}
    &2\|e^{\nub}_w  (\ub^{(n)})\|_{L^2(\R^d;\R^{d\times d})}^2=2\sum_{i,j=1}^d\int_{\R^d}\left(\frac{1}{2}\left(\mathcal{G}_iu^{(n)}_j(\xb)+\mathcal{G}_ju^{(n)}_i(\xb)\right)\right)^2d\xb\\
    =&\|\mathcal{G}^{\nub}_w \ub^{(n)}\|_{L^2(\R^d;\R^{d\times d})}^2+\sum_{i,j=1}^d\int_{\R^d}\mathcal{G}_iu^{(n)}_j(\xb)\mathcal{G}_ju^{(n)}_i(\xb)d\xb\\
    =&\|\mathcal{G}^{\nub}_w \ub^{(n)}\|_{L^2(\R^d;\R^{d\times d})}^2+\|\mathcal{D}^{-\nub}_w \ub^{(n)}\|_{L^2(\R^d)}^2.
\end{align*}
Thus,
\[
\begin{split}
2\mathcal{E}  \left(\ub^{(n)}\right)&=\lambda \|\mathcal{D}^{-\nub}_w \ub^{(n)}\|_{L^2(\R^d)}^2+\mu\left(\|\mathcal{G}^{\nub}_w \ub^{(n)}\|_{L^2(\R^d;\R^{d\times d})}^2+\|\mathcal{D}^{-\nub}_w \ub^{(n)}\|_{L^2(\R^d)}^2\right)\\
&=B\left(\ub^{(n)},\ub^{(n)}\right).
\end{split}
\]Letting $n\to\infty$ finishes the proof.
\end{proof}

Now we are ready to establish the well-posedness of problem (\ref{eq:wkform_elasteq}). In fact, an analogue of the classical Korn's inequality holds in the nonlocal setting.
\begin{lem}[Nonlocal Korn's inequality]\label{lem:Kornieq}
There exists a constant $C=\frac{1}{2}\min(\lambda+2\mu,\mu)>0$ such that 
\begin{equation}
    \mathcal{E} (\ub)\ge C\|\mathfrak{G}^{\nub}_w \ub\|_{L^2(\R^d;\R^{d\times d})}^2,\quad\forall \ub\in \cS_w^{\nub} (\Om;\R^d).
\end{equation}
\end{lem}
\begin{proof}
By \Cref{lem:equivelasenergy} and \Cref{thm:density}, it suffices to show \[B(\ub,\ub)\ge \min(\lambda+2\mu,\mu)\|\mathcal{G}^{\nub}_w \ub\|_{L^2(\R^d;\R^{d\times d})}^2,\quad\forall \ub\in C^\infty_c(\Om;\R^d).\]
Using the notations and Plancherel's theorem as in the proof of \Cref{prop:divisL2} yields
\begin{align*}
    &B(\ub,\ub)=\mu\|\mathcal{G}^{\nub}_w \ub\|_{L^2(\R^d;\R^{d\times d})}^2+(\lambda+\mu)\|\mathcal{D}^{-\nub}_w \ub\|_{L^2(\R^d)}^2\\
    =&\mu\int_{\R^d}\left|\bm \lambda_w^{\nub}(\xib)\right|^2\left|\hat{\ub}(\xib)\right|^2d\xib+(\lambda+\mu)\int_{\R^d}\left|\bm\lambda_w^{-\nub}(\xib)^T\hat{\ub}(\xib)\right|^2 d\xib\\
    \ge &\min(\lambda+2\mu,\mu)\int_{\R^d}\left|\bm \lambda_w^{\nub}(\xib)\right|^2\left|\hat{\ub}(\xib)\right|^2d\xib=\min(\lambda+2\mu,\mu)\|\mathcal{G}^{\nub}_w \ub\|_{L^2(\R^d;\R^{d\times d})}^2.
\end{align*}
\end{proof}

\begin{thm}
The nonlocal linear elasticity problem (\ref{eq:wkform_elasteq}) is well-posed. More precisely, for any $\fb\in (\cS_w^{\nub}(\Om;\R^d))^\ast$, there exists a unique solution $\ub\in \cS_w^{\nub}(\Om;\R^d)$ such that 
\[\|\ub\|_{\cS_w^{\nub}(\Om;\R^d)}\le c\|\fb\|_{(\cS_w^{\nub}(\Om;\R^d))^\ast},\]where $c=\min(\lambda+2\mu,\mu)^{-1}$ is a positive constant.
\end{thm}
\begin{proof}
The bilinear form $B(\cdot,\cdot)$ is coercive by \Cref{lem:equivelasenergy} and \Cref{lem:Kornieq}, and is bounded by \Cref{prop:divisL2} and \Cref{thm:density}. Applying the Lax-Milgram theorem yields the result. 
\end{proof}

\subsection{Nonlocal Helmholtz decomposition} 
\label{subsec:helmhotlz}
In this subsection, we always assume that $d=2$ or $d=3$. The nonlocal vector calculus identities in \Cref{subsec:vector_identities_general} will be used to obtain the nonlocal Helmholtz decomposition for $d=2$ and $d=3$.
These results extend similar studies in \cite{lee2020nonlocal} for periodic functions.

\begin{thm}\label{thm:helm2d}
Let $\ub\in (\cS_w^{\nub}(\Om;\R^2))^\ast$. There exist scalar potentials $p^{\nub},\, q^{\nub}\in L^2(\R^2)$ such that 
\[\ub=\mathfrak{G}_{w}^{\nub} p^{\nub}+\begin{pmatrix}
    0 & 1\\
    -1 & 0
    \end{pmatrix}\mathfrak{G}_{w}^{-\nub} q^{\nub} \in (\cS_w^{\nub}(\Om;\R^2))^\ast. \]
    In addition, there exists a constant $C$ depending on the Poincar\'e constant $\Pi$ such that 
    \[\|p^{\nub}\|_{L^2(\R^2)}+\|q^{\nub}\|_{L^2(\R^2)}\le C\|\ub\|_{(\cS_w^{\nub}(\Om;\R^2))^\ast}.\]
\end{thm}
\begin{proof}
Applying \Cref{thm:wellposedwkform} with $\ep=1$ and $\bb=\bm{0}$ componentwise, it follows that there exists a unique function $\fb \in \cS_w^{\nub}(\Om;\R^2)$ such that  $-\mathcal{L}_w^{\nub}  \fb =\ub$ with
\[
    \|\fb \|_{\cS_w^{\nub}(\Om;\R^2)}\le c\|\ub\|_{(\cS_w^{\nub}(\Om;\R^2))^\ast},
\]
where $c=c(w,\nub,\Om)>0$. Let \[p^{\nub}=-\mathfrak{D}_w^{-\nub} \fb \ \text{ and }\ q^{\nub}=\mathfrak{D}_w^{\nub}\left[\begin{pmatrix}
    0 & 1\\
    -1 & 0
    \end{pmatrix}\fb \right].\]
By \Cref{prop:divisL2}, we have $p^{\nub},q^{\nub} \in L^2(\R^2)$, and
$\|p^{\nub}\|_{L^2(\R^2)}+\|q^{\nub}\|_{L^2(\R^2)}\le \tilde{C}\|\fb \|_{\cS_w^{\nub}(\Om;\R^2)}\leq C \|\ub\|_{(\cS_w^{\nub}(\Om;\R^2))^\ast}$.
Then by \Cref{lem:veccalid1} we obtain
    \begin{align*}
        \ub&=-\mathcal{L}_w^{\nub}  \fb=\mathfrak{G}_{w}^{\nub} p^{\nub}+\begin{pmatrix}
    0 & 1\\
    -1 & 0
    \end{pmatrix}\mathfrak{G}_{w}^{-\nub} q^{\nub}. 
    \end{align*}
    This finishes the proof.
\end{proof}

\begin{thm}\label{thm:helm3d}
Let $\ub\in (\cS_w^{\nub}(\Om;\R^3))^\ast$. There exist a scalar potential $p^{\nub}\in L^2(\R^3)$ and a vector potential $\vb^{\nub}\in L^2(\R^3;\R^3)$ such that
\begin{equation}\label{helm3d}
    \ub=\mathfrak{G}_{w}^{\nub} p^{\nub}+\mathfrak{C}_{w}^{-\nub} \vb^{\nub},
\end{equation}
with \begin{equation}\label{helm3ddivvanish}
    \mathfrak{D}_w^{\nub}\vb^{\nub}=0,
\end{equation}
where the above equations are understood in $(\cS_w^{\nub}(\Om;\R^3))^\ast$ and $(\cS_w^{-\nub}(\Om))^\ast$, respectively.
In addition, there exists a constant $C$ depending on the Poincar\'e constant $\Pi$ such that 
    \[\|p^{\nub}\|_{L^2(\R^3)}+\|\vb^{\nub}\|_{L^2(\R^3;\R^3)}\le C\|\ub\|_{(\cS_w^{\nub}(\Om;\R^3))^\ast}.\]
\end{thm}
\begin{proof}
As in the proof of Theorem \ref{thm:helm2d}, there exists $\fb \in \cS_w^{\nub}(\Om;\R^3)$ such that $-\mathcal{L}_w^{\nub}  \fb =\ub$ and
\beq\label{apriori1}
    \|\fb \|_{\cS_w^{\nub}(\Om;\R^3)}\le c\|\ub\|_{(\cS_w^{\nub}(\Om;\R^3))^\ast}.
\eeq
 We choose \[p^{\nub}=-\mathfrak{D}_w^{-\nub} \fb \ \text{ and }\ \vb^{\nub}=\mathfrak{C}_{w}^{\nub}\fb,\] and use \cref{veccalid2} to derive \cref{helm3d}. The computation is staightforward and thus omitted. By \Cref{lem:vanishidentity_generalfct}, \cref{helm3ddivvanish} holds.
Similar to the proof of Theorem \ref{thm:helm2d}, the final estimate follows from \Cref{prop:divisL2}, \Cref{prop:curlisL2} and \cref{apriori1}.
\end{proof}

\begin{rem}
If the kernel function $w$ is integrable, then by \Cref{prop:properties}, for any $u\in L^2(\R^d)$ and $\vb\in L^2(\R^d;\R^d)$, we have $\mathfrak{G}_w^{\nub} u = \mathcal{G}_w^{\nub} u \in L^2(\R^d;\R^d)$ and $ \mathfrak{C}_w^{\nub} \vb = \mathcal{C}_w^{\nub} \vb\in L^2(\R^d;\R^d)$. Therefore, the two components in the Helmholtz decomposition in \Cref{thm:helm2d} or \Cref{thm:helm3d} are orthogonal in $L^2(\R^d;\R^d)$ as a result of integration by parts together with \Cref{rem:2dvanishidentity} ($d=2$) or \Cref{lem:vanishidentity_generalfct} ($d=3$).
\end{rem}

\begin{rem}
    If the kernel function $w$ has compact support, then the potentials in \Cref{thm:helm2d} and \Cref{thm:helm3d} vanish outside a compact set. More specifically, if $\supp\ w \subset B_\del(\bm{0})$ for $\del>0$, then the  $p^{\nub}, q^{\nub} \in L^2(\Om_\del)$ and $\vb^{\nub}\in L^2(\Om_\del;\R^3)$ where  $\Om_\del=\{\xb\in \R^d:\dist(\xb,\Om)<\del \}$. 
\end{rem}

\section{Conclusion}\label{sec:conclusion}
In this paper, we have studied nonlocal half-ball gradient, divergence and curl operators with a rather general class of kernels. These nonlocal operators can be generalized to distributional operators upon which a Sobolev-type space is defined. For this function space, the set of smooth functions with compact support is proved to be dense. Moreover, a nonlocal Poincar\'e inequality on bounded domains is established, 
which is crucial to study the well-posedness of nonlocal Dirichlet boundary value problems such as nonlocal convection-diffusion and nonlocal correspondence model of linear elasticity and to prove a nonlocal Helmholtz decomposition. 

This work provides a rigorous mathematical analysis on the stability of some linear nonlocal problems with homogeneous Dirichlet boundary, thus generalizes the analytical results in \cite{lee2020nonlocal} where the domains are periodic cells. While we mainly focused on the analysis of these nonlocal problems, 
standard Galerkin approximations to these problems are also natural based on the Poincar\'e inequality and the density result. 
It would also be interesting to investigate Petrov-Galerkin methods for the nonlocal convection-diffusion problems \cite{leng2022petrov}, as well as mixed-type methods for them \cite{DFHT21}. 
Other problems such as nonlocal elasticity models in heterogeneous media and the Stokes system in \cite{DuTi20,lee2020nonlocal} may also be studied in the future. As for the analysis, our approach relies heavily on Fourier analysis which is powerful but limited to $L^2$ formulation. The nonlocal $L^p$ Poincar\'e inequality for half-ball gradient operator on bounded domains is still open to investigation. In addition, Poincar\'e inequality for Neumann type boundary is also interesting to be explored in the future. 
We note that in this work the dependence of the Poincar\'e constant on the kernel function is implicit as a result of argument by contradiction. Further investigation on how the constant depends on the kernel function is needed, and following \cite{MeDu14a}, a sharper version of Poincar\'e inequality may be considered by establishing compactness results analogous to those in \cite{BBM01}. 
Last but not least, it remains of great interest to develop nonlocal exterior calculus and geometric structures that connect the corresponding discrete theories and continuous local theories \cite{Arno18,bartholdi2012hodge,hirani2003discrete,le2013nonlocal}. 

\section*{Acknowledgements}
This research was supported in part by NSF grants DMS-2111608 and DMS-2240180. The authors thank Qiang Du, Tadele Mengesha, and James Scott for their helpful discussions. The authors would also like to thank the anonymous reviewers for their valuable comments and suggestions. 
\appendix
\section{}

\begin{proof}[Proof of \Cref{lem:opforW1pfct}]
    Let $u\in W^{1,p}(\R^d)$.
    To use Lebesgue dominated convergence theorem to show the principal value integral coincide with the usual Lebesgue integral, we construct the majorizing function \[
    g_{\xb}(\yb):=|u(\yb)-u(\xb)|w(\yb-\xb),\quad \yb\in\R^d,
    \]
    and show that $g_{\xb}\in L^1(\R^d)$ for a.e. $\xb\in\R^d$. This follows from the fact that the function $\xb\mapsto \int_{\R^d} g_{\xb}(\yb)d\yb\in L^p(\R^d)$. When $p=\infty$, this is obvious. To show this fact for $1\le p<\infty$, first note that\[
    \begin{split}
        \int_{\R^d}\left|\int_{\R^d} g_{\xb}(\yb)d\yb\right|^pd\xb&=\int_{\R^d}\left|\int_{|\yb-\xb|<1} g_{\xb}(\yb)d\yb+\int_{|\yb-\xb|>1} g_{\xb}(\yb)d\yb\right|^pd\xb\\
        \le 2^{p-1} \int_{\R^d}&\left(\int_{|\yb-\xb|<1} g_{\xb}(\yb)d\yb\right)^p d\xb+\int_{\R^d}\left(\int_{|\yb-\xb|>1} g_{\xb}(\yb)d\yb\right)^p d\xb .
    \end{split}
    \]
    Then by H\"older's inequality,
    \[
    \begin{split}
        &\int_{\R^d}\left(\int_{|\yb-\xb|<1} g_{\xb}(\yb)d\yb\right)^p d\xb\\
        =&\int_{\R^d} \left(\int_{|\zb|<1} \frac{|u(\xb+\zb)-u(\xb)|}{|\zb|}|\zb|w(\zb)d\zb\right)^p d\xb\\
        \le &\left(\int_{|\zb|<1} |\zb|w(\zb)d\zb\right)^{p-1}\int_{\R^d}\int_{|\zb|<1} |\zb|w(\zb)\frac{|u(\xb+\zb)-u(\xb)|^p}{|\zb|^p}d\zb d\xb\\
        \le & (M_w^1)^p\|\nabla u\|_{L^p(\R^d)}^p,
    \end{split}
    \]
    where we used inequality (see Proposition 9.3 in \cite{Brez11}) \[\int_{\R^d} \frac{|u(\xb+\zb)-u(\xb)|^p}{|\zb|^p} d\xb\le \|\nabla u\|_{L^p(\R^d)}^p,\quad \zb\in\R^d\backslash \{\bm 0\}.\]
    Applying the same techniques it follows that 
    \[
    \begin{split}
        &\int_{\R^d}\left(\int_{|\yb-\xb|>1} g_{\xb}(\yb)d\yb\right)^p d\xb\\
        \le & 2^{p-1}\int_{\R^d} \left(\int_{|\yb-\xb|>1}|u(\yb)|w(\yb-\xb)d\yb\right)^p + \left(\int_{|\yb-\xb|>1}|u(\xb)|w(\yb-\xb)d\yb\right)^p d\xb\\
        \le & 2^p (M_w^2)^p \|u\|_{L^p(\R^d)}^p.
    \end{split}
    \]
    Combining the above estimates, there exists a constant $C>0$ depending on $p$ such that 
    \begin{equation}\label{estmajorfct}
        \left(\int_{\R^d}\left|\int_{\R^d} g_{\xb}(\yb)d\yb\right|^pd\xb\right)^{\frac{1}{p}}\le C\left(M_w^1 \|\nabla u\|_{L^p(\R^d)}+M_w^2 \|u\|_{L^p(\R^d)}\right),\quad u\in W^{1,p}(\R^d).
    \end{equation}
    Therefore, $g_{\xb}\in L^1(\R^d)$ for a.e. $\xb\in\R^d$ and by Lebesgue dominated convergence theorem, equalities \eqref{Gwithoutpv} hold for a.e. $x\in\R^d$. Since $|\mathcal{G}^{\nub}_w u(\xb)|\le \int_{\R^d} g_{\xb}(\yb)d\yb$, the estimate \eqref{LpestgradW1p} follows from \eqref{estmajorfct}. Similar proofs hold for $\mathcal{D}^{\nub}_w$ and $\mathcal{C}^{\nub}_w$ and are omitted.
\end{proof}

\begin{proof}[Proof of \Cref{prop:hbibp}]
\begin{enumerate}
    \item Since $w(\xb-\yb)\left|\ub(\xb)-\ub(\yb)\right|\in L^1(\R^d\times \R^d)$ and $\vb\in C^1_c(\R^d;\R^{d\times N})$, one can show by Lebesgue dominated convergence theorem that
\begin{equation}\label{DCT1}
    \int_{\R^d} \mathcal{G}^{\nub}_w u(\xb):\vb(\xb)d\xb=\lim_{\epsilon\to 0} \iint_{\R^{2d}_\epsilon}\chi_{\nub}(\yb-\xb)w(\yb-\xb)\frac{\yb-\xb}{|\yb-\xb|}\otimes(\ub(\yb)-\ub(\xb)): \vb(\xb)d\yb d\xb,
\end{equation}
where $\R^{2d}_\epsilon:=\R^d\times \R^d\backslash\{(\xb,\yb)\in \R^{2d}:|\xb-\yb|\le \epsilon\}$. Similarly, 
\begin{align*}
  &  -\int_{\R^d} \ub(\xb)\cdot\mathcal{D}^{-\nub}_w \vb (\xb)d\xb=-\int_{\R^d}\ub(\xb) \cdot\lim_{\epsilon\to 0}g_\epsilon(\xb)d\xb\\
    =&-\lim_{\epsilon\to 0}\iint_{\R^{2d}_\epsilon}\ub(\xb)\cdot \left[\frac{\yb^T-\xb^T}{|\yb-\xb|}(\chi_{\nub}(\xb-\yb)\vb(\yb)+\chi_{\nub}(\yb-\xb)\vb(\xb))\right]^T w(\yb-\xb)d\yb d\xb,
\end{align*}
where 
\begin{align*}
    g_\epsilon(\xb): &=\int_{\R^d\backslash B_\epsilon(\xb)}\left[\frac{\yb^T-\xb^T}{|\yb-\xb|}(\chi_{\nub}(\xb-\yb)(\vb(\yb)-\vb(\xb)))\right]^T w(\yb-\xb)d\yb,\\
    &=\int_{\R^d\backslash B_\epsilon(\xb)}\left[\frac{\yb^T-\xb^T}{|\yb-\xb|}(\chi_{\nub}(\xb-\yb)\vb(\yb)+\chi_{\nub}(\yb-\xb)\vb(\xb))\right]^T w(\yb-\xb)d\yb \\
     & \qquad -  \int_{\R^d\backslash B_\epsilon(\xb)}\left[\frac{\yb^T-\xb^T}{|\yb-\xb|}\vb(\xb)\right]^T w(\yb-\xb)d\yb \\
     &= \int_{\R^d\backslash B_\epsilon(\xb)}\left[\frac{\yb^T-\xb^T}{|\yb-\xb|}(\chi_{\nub}(\xb-\yb)\vb(\yb)+\chi_{\nub}(\yb-\xb)\vb(\xb))\right]^T w(\yb-\xb)d\yb
\end{align*}
where we have used $\chi_{\nub}(\xb-\yb)+\chi_{\nub}(\yb-\xb)=1$. The change of order of limitation and integration is again justified by Lebesgue dominated convergence theorem due to $\ub\in L^1(\R^d;\R^N)$ and $\vb\in C_c^1(\R^d;\R^{d\times N})$. 

Therefore, it suffices to prove that
\begin{align*}
    &\quad \iint_{\R^{2d}_\epsilon}\chi_{\nub}(\yb-\xb)w(\yb-\xb)\frac{\yb-\xb}{|\yb-\xb|}\otimes(\ub(\yb)-\ub(\xb)): \vb(\xb)d\yb d\xb  \\
    &=-\iint_{\R^{2d}_\epsilon}\ub(\xb)\cdot \left[\frac{\yb^T-\xb^T}{|\yb-\xb|}(\chi_{\nub}(\xb-\yb)\vb(\yb)+\chi_{\nub}(\yb-\xb)\vb(\xb))\right]^T w(\yb-\xb)d\yb d\xb .
\end{align*}
Applying Fubini's theorem completes the proof as
\begin{align*}
    &\quad \iint_{\R^{2d}_\epsilon}\chi_{\nub}(\yb-\xb)w(\yb-\xb)\frac{\yb-\xb}{|\yb-\xb|}\otimes(\ub(\yb)-\ub(\xb)): \vb(\xb)d\yb d\xb  \\
    &=\iint_{\R^{2d}_\epsilon}\chi_{\nub}(\yb -\xb  )w(\yb -\xb  )\frac{\yb -\xb  }{|\yb -\xb  |}\otimes \ub(\yb): \vb (\xb  )d\yb d\xb  \\
    &\quad-\iint_{\R^{2d}_\epsilon}\chi_{\nub}(\yb -\xb  )w(\yb -\xb  )\frac{\yb -\xb  }{|\yb -\xb  |}\otimes \ub(\xb):  \vb (\xb  )d\yb d\xb  \\
    &=-\iint_{\R^{2d}_\epsilon}\chi_{\nub}(\xb -\yb  )w(\yb -\xb  )\frac{\yb -\xb  }{|\yb -\xb  |}\otimes \ub(\xb): \vb (\yb  )d\yb  d\xb\\
    &\quad-\iint_{\R^{2d}_\epsilon}\chi_{\nub}(\yb -\xb  )w(\yb -\xb  )\frac{\yb -\xb  }{|\yb -\xb  |}\otimes \ub(\xb):  \vb (\xb  )d\yb  d\xb  \\
    &=-\iint_{\R^{2d}_\epsilon}\ub(\xb)\cdot \left[\frac{\yb^T-\xb^T}{|\yb-\xb|}(\chi_{\nub}(\xb-\yb)\vb(\yb)+\chi_{\nub}(\yb-\xb)\vb(\xb))\right]^T w(\yb-\xb)d\yb  d\xb  .
\end{align*}
\item Since $w(\xb-\yb)|\ub(\xb)-\ub(\yb)|\in L^1(\R^d\times\R^d)$ and $\vb\in C^1_c(\R^d;\R^N)$, by Lebesgue dominated convergence theorem one can show that
\begin{equation}\label{DCT2}
\begin{split}
    &\int_{\R^d} \mathcal{D}^{\nub}_w \ub(\xb)  \cdot \vb(\xb)d\xb \\
    =&\lim_{\epsilon\to 0}\iint_{\R^{2d}_\epsilon}\chi_{\nub}(\yb-\xb)w(\yb-\xb)\left[\frac{\yb^T-\xb^T}{|\yb-\xb|}(\ub(\yb)-\ub(\xb))\right]^T  \cdot \vb(\xb)d\yb d\xb,
 \end{split}   
\end{equation}
where $\R^{2d}_\epsilon:=\R^d\times\R^d\backslash\{(\xb,\yb)\in \R^{2d}:|\xb-\yb|\le \epsilon\}$. Similarly, by the same reasoning as in the proof of \Cref{prop:hbibp}(1),
\begin{align*}
    &-\int_{\R^d} \ub(\xb):\mathcal{G}^{\nub}_w \vb(\xb)d\xb\\
    =&-\lim_{\epsilon\to 0}\iint_{\R^{2d}_\epsilon} \chi_{\nub}(\yb-\xb)w(\yb-\xb)\ub(\xb): \frac{\yb-\xb}{|\yb-\xb|}\otimes(\vb(\yb)-\vb(\xb))d\yb d\xb.
\end{align*}
Therefore, it suffices to prove that 
\begin{align*}
    & \iint_{\R^{2d}_\epsilon}\chi_{\nub}(\yb-\xb)w(\yb-\xb)\left[\frac{\yb^T-\xb^T}{|\yb-\xb|}(\ub(\yb)-\ub(\xb))\right]^T  \cdot \vb(\xb)d\yb d\xb\\
    =& -\iint_{\R^{2d}_\epsilon} \chi_{\nub}(\yb-\xb)w(\yb-\xb)\ub(\xb): \frac{\yb-\xb}{|\yb-\xb|}\otimes(\vb(\yb)-\vb(\xb))d\yb d\xb.
\end{align*}
Applying Fubini's theorem as in the proof of \Cref{prop:hbibp}(1) gives the desired result.
\item By similar reasoning as the proof of \Cref{prop:hbibp}(1) and \Cref{prop:hbibp}(2), we have 
\[
\int_{\R^d} \mathcal{C}^{\nub}_w \ub(\xb)  \cdot \vb(\xb)d\xb = \lim_{\epsilon\to 0} \iint_{\R^{2d}_\epsilon} 
\chi_{\nub}(\yb-\xb)w(\yb-\xb)\frac{\yb-\xb}{|\yb-\xb|}\times(\ub(\yb)-\ub(\xb))\cdot \vb(\xb) d\yb d\xb, 
\]
and 
\[
\int_{\R^d}  \ub(\xb)  \cdot \mathcal{C}^{-\nub}_w\vb(\xb)d\xb = \lim_{\epsilon\to 0} \iint_{\R^{2d}_\epsilon} 
\chi_{\nub}(\xb-\yb)w(\yb-\xb)\frac{\yb-\xb}{|\yb-\xb|}\times(\vb(\yb)-\vb(\xb))\cdot \ub(\xb) d\yb d\xb. 
\]
Using Fubini's theorem and the two identities $\chi_{\nub}(\xb-\yb)+\chi_{\nub}(\yb-\xb)=1$ and $\ab \cdot (\bb\times \cb) = -\cb \cdot (\bb\times \ab)$, one can show
\[
\begin{split}
&\iint_{\R^{2d}_\epsilon} 
\chi_{\nub}(\yb-\xb)w(\yb-\xb)\frac{\yb-\xb}{|\yb-\xb|}\times(\ub(\yb)-\ub(\xb))\cdot \vb(\xb) d\yb d\xb \\
=&
\iint_{\R^{2d}_\epsilon} 
\chi_{\nub}(\xb-\yb)w(\yb-\xb)\frac{\yb-\xb}{|\yb-\xb|}\times(\vb(\yb)-\vb(\xb))\cdot \ub(\xb) d\yb d\xb,
\end{split}
\]
for any $\ep>0$, and therefore the desired result is implied.
\end{enumerate}

\end{proof}

\begin{proof}[Proof of \Cref{lem:contitrans}]
Note that $\tau_{\ab} u \in L^2(\R^d)$ is obvious. 
To show $\tau_{\ab} u \in \cS_w^{\nub}(\R^d;\R^d)$, it suffices to show $\mathfrak{G}^{\nub}_w (\tau_{\ab} u) \in L^2(\R^d;\R^d)$. We claim that $\mathfrak{G}^{\nub}_w (\tau_{\ab} u) =\tau_{\ab}(\mathfrak{G}^{\nub}_w u)\in L^2(\R^d;\R^d)$. Indeed, for any $\bm\phi\in C^\infty_c(\R^d;\R^d)$,
\begin{align*}
    \int_{\R^d} \tau_{\ab} u (\xb) \mathcal{D}^{-\nub}_w \bm\phi(\xb)d\xb &=\int_{\R^d} u(\xb) (\mathcal{D}^{-\nub}_w \bm\phi)(\xb -\ab)d\xb \\
    &=\int_{\R^d} u(\xb) (\tau_{-\ab}\mathcal{D}^{-\nub}_w \bm\phi)(\xb)d\xb \\
    &=\int_{\R^d} u(\xb) (\mathcal{D}^{-\nub}_w (\tau_{-\ab}\bm\phi))(\xb)d\xb \\
    &=-\int_{\R^d} \mathfrak{G}^{\nub}_w u(\xb)\cdot \tau_{-\ab}\bm\phi(\xb)d\xb \\
    &=-\int_{\R^d} \tau_{\ab}(\mathfrak{G}^{\nub}_w u)(\xb)\cdot \bm\phi(\xb)d\xb ,
\end{align*}
where $\mathcal{D}^{-\nub}_w (\tau_{-\ab}\bm\phi)=\tau_{-\ab}\mathcal{D}^{-\nub}_w \bm\phi$ can be easily checked. Therefore, the claim is true and thus $\tau_{\ab} u \in \cS_w^{\nub}(\R^d)$.

To show the continuity, first note that \[\lim_{|\ab|\to 0}\|\tau_{\ab} u -u\|_{L^2(\R^d)}=0\] by continuity of translation in $L^2(\R^d)$. Then using the claim above and the continuity of translation in $L^2(\R^d;\R^d)$, we have
\[\|\mathfrak{G}^{\nub}_w (\tau_{\ab} u) -\mathfrak{G}^{\nub}_w u\|_{L^2(\R^d;\R^d)}=\|\tau_{\ab}(\mathfrak{G}^{\nub}_w u)-\mathfrak{G}^{\nub}_w u\|_{L^2(\R^d;\R^d)}\to 0,\quad|\ab|\to 0.\]Hence, \[\lim_{|\ab|\to 0}\|\tau_{\ab} u -u\|_{\cS_w^{\nub}(\R^d)}=0.\]
\end{proof}

\begin{proof}[Proof of \Cref{lem:mollif}]
Since $u\in L^2(\R^d)$, by the property of mollification, $\eta_\epsilon*u\in L^2(\R^d)$ and \[\lim_{\epsilon\to 0}\|\eta_\epsilon*u-u\|_{L^2(\R^d)}=0.\]  We claim that \[\mathfrak{G}^{\nub}_w (\eta_\epsilon*u)=\eta_\epsilon*\mathfrak{G}^{\nub}_w u\in L^2(\R^d;\R^d).\]
To show the claim, we need to prove that 
\begin{equation}\label{eq:claiminmoll}
    \int_{\R^d} (\eta_\epsilon*\mathfrak{G}^{\nub}_w u)(\xb)\bm\phi(\xb)d\xb =-\int_{\R^d}(\eta_\epsilon*u)(\xb)\mathcal{D}^{-\nub}_w \bm\phi(\xb)d\xb ,\quad\forall\bm\phi\in C^\infty_c(\R^d;\R^d).
\end{equation}
For the right-hand side, we use Fubini's theorem to get
\begin{align*}
    -\int_{\R^d}(\eta_\epsilon*u)(\xb)\mathcal{D}^{-\nub}_w \bm\phi(\xb)d\xb &=-\int_{\R^d}\int_{\R^d} \eta_\epsilon(\xb-\yb )u(\yb)d\yb \mathcal{D}^{-\nub}_w \bm\phi(\xb)d\xb \\
    &=-\int_{\R^d} u(\yb)\int_{\R^d}\eta_\epsilon(\yb-\xb)\mathcal{D}^{-\nub}_w \bm\phi(\xb)d\xb d\yb \\
    &=-\int_{\R^d}u(\xb)(\eta_\epsilon*\mathcal{D}^{-\nub}_w \bm\phi)(\xb)d\xb ,
\end{align*}
For the left-hand side,  use Fubini's theorem again to obtain 
\begin{align*}
    \int_{\R^d} (\eta_\epsilon*\mathfrak{G}^{\nub}_w u)(\xb)\cdot\bm\phi(\xb)d\xb &=\int_{\R^d}\int_{\R^d}\eta_\epsilon(\xb-\yb )\mathfrak{G}^{\nub}_w u(\yb)d\yb \cdot\bm\phi(\xb)d\xb \\
    &=\int_{\R^d}\mathfrak{G}^{\nub}_w u(\yb)\cdot\int_{\R^d}\eta_\epsilon(\yb-\xb)\bm\phi(\xb)d\xb d\yb \\
    &=\int_{\R^d}\mathfrak{G}^{\nub}_w u(\yb)\cdot(\eta_\epsilon*\bm\phi)(\yb)d\yb \\
    &=-\int_{\R^d} u(\yb)\mathcal{D}^{-\nub}_w(\eta_\epsilon*\bm\phi)(\yb)d\yb . 
\end{align*}
One can check that 
$\mathcal{D}^{-\nub}_w (\eta_\epsilon*\bm\phi)(\xb)=(\eta_\epsilon*\mathcal{D}^{-\nub}_w \bm\phi)(\xb)$ and therefore 
\[
  \int_{\R^d} (\eta_\epsilon*\mathfrak{G}^{\nub}_w u)(\xb)\bm\phi(\xb)d\xb  =-\int_{\R^d} u(\yb)(\eta_\epsilon*\mathcal{D}^{-\nub}_w\bm\phi)(\yb)d\yb ,
\] 
Comparing the left-hand and right-hand side, \cref{eq:claiminmoll} is proved and $\mathfrak{G}^{\nub}_w (\eta_\epsilon*u)=\eta_\epsilon*\mathfrak{G}^{\nub}_w u\in L^2(\R^d;\R^d) $. Therefore
\[\lim_{\epsilon\to 0}\|\mathfrak{G}^{\nub}_w (\eta_\epsilon*u)-\mathfrak{G}^{\nub}_w u\|_{L^2(\R^d;\R^d)}=\lim_{\epsilon\to 0}\|\eta_\epsilon*\mathfrak{G}^{\nub}_w u-\mathfrak{G}^{\nub}_w u\|_{L^2(\R^d;\R^d)}= 0,\] 
and thus the lemma is proved.

\end{proof}

\bibliographystyle{abbrv} 

\bibliography{ref}

 \end{document}